\documentclass{article}

\usepackage{amsmath,  amsthm, amssymb}
\usepackage{graphicx}
\usepackage[colorlinks,citecolor=blue]{hyperref}
\usepackage{tikz}
\usetikzlibrary{arrows,shapes,chains,matrix,positioning,scopes}
\usetikzlibrary{decorations.markings,calc,patterns}

\usepackage{bbm}
\usepackage{float}
\usepackage{enumitem}
\usepackage{comment}

\usepackage[left=3cm, right=3cm, bottom=4cm]{geometry}

\newcommand\blfootnote[1]{%
  \begingroup
  \renewcommand\thefootnote{}\footnote{#1}%
  \addtocounter{footnote}{-1}%
  \endgroup
}

\title{Load Balancing in Hypergraphs \blfootnote{Parts of this work have been presented in the $55^\text{th}$ Annual Allerton Conference on Communication, Control, and Computing}}

\author{Payam Delgosha and Venkat Anantharam\\[2mm]
\small Department of Electrical Engineering and Computer Sciences\\
\small University of California, Berkeley\\
\small \{pdelgosha, ananth\} @ berkeley.edu
}

\newcommand{\ev}[1]{\mathbb{E} \left [ #1 \right ] }
\newcommand{\pr}[1]{\mathbb{P} \left ( #1 \right ) }
\newcommand{\norm}[1]{\left \Vert #1 \right \Vert}
\newcommand{\snorm}[1]{\Vert #1 \Vert}
\newcommand{\one}[1]{\mathbbm{1} \left [ #1 \right ]}
\newcommand{\oneu}[1]{\mathbbm{1}_{#1}}

\newtheorem{lem}{Lemma}
\newtheorem{thm}{Theorem}
\newtheorem{definition}{Definition}
\newtheorem{prop}{Proposition}
\newtheorem{rem}{Remark}
\newtheorem{cor}{Corollary}

\newcommand{\mH}{\mathcal{H}}
\newcommand{\mbH}{\bar{\mathcal{H}}}

\newcommand{\bmu}{\bar{\mu}}
\newcommand{\vbmu}{\vec{\bar{\mu}}}
\newcommand{\vnu}{\vec{\nu}}

\newcommand{\mA}{\mathcal{A}}
\newcommand{\mP}{\mathcal{P}}
\newcommand{\mU}{\mathcal{U}}
\newcommand{\mL}{\mathcal{L}}

\newcommand{\mM}{\mathcal{M}}
\newcommand{\mF}{\mathcal{F}}
\newcommand{\mT}{\mathcal{T}}
\newcommand{\mC}{\mathcal{C}}

\newcommand{\mQ}{\mathcal{Q}}

\newcommand{\bT}{\mathbb{T}}
\newcommand{\bbT}{\bar{\mathbb{T}}}
\newcommand{\wtP}{\widetilde{P}}

\newcommand{\bH}{\bar{H}}

\newcommand{\tT}{{\tilde{T}}}

\newcommand{\reals}{\mathbb{R}}

\newcommand{\nats}{\mathbb{N}}

\newcommand{\pt}{\partial \theta}
\newcommand{\pbt}{\partial_b \theta}
\newcommand{\pbpt}{\partial_{b'} \theta}

\newcommand{\rf}{\rho}
\newcommand{\rft}[2]{\rf_{T_{#1 \rightarrow #2}}}

\newcommand{\gn}{\gamma^{(n)}}
\newcommand{\bi}{\mathbf{i}}

\newcommand{\vmu}{\vec{\mu}} 

\newcommand{\gwt}{\text{GWT}}
\newcommand{\ugwt}{\text{UGWT}}

\newcommand{\ER}{Erd\H{o}s--R\'{e}nyi }

\newcommand{\natsz}{\mathbb{N}_0} 
\newcommand{\natszf}{\Lambda} 
\newcommand{\typee}{\mathbbm{e}} 
\newcommand{\nvertex}{\mathbb{N}_\text{vertex}}
\newcommand{\nedge}{\mathbb{N}_\text{edge}}

\newcommand{\norminf}[1]{h(#1)}

\newcommand{\evpair}{\Psi} 

\let\oldmarginpar\marginpar
\renewcommand{\marginpar}[2][rectangle,draw,rounded corners,text width = 3cm, scale=0.7]{%
        \oldmarginpar{%
          \tikz \node at (0,0) [#1]{#2};}%
        }

\usepackage{etoolbox}

\makeatletter
\patchcmd{\@addmarginpar}{\ifodd\c@page}{\ifodd\c@page\@tempcnta\m@ne}{}{}
\makeatother
\reversemarginpar

\newcommand{\edit}{}

\colorlet{Cyan}{cyan}
\colorlet{Orange}{orange}
\tikzstyle{Node} = [circle,fill,inner sep=1pt]
\tikzstyle{Root} = [circle,fill=magenta,inner sep=1.7pt]

\begin{document}
\maketitle

\begin{abstract}
Consider a simple locally finite hypergraph on a countable vertex set, where each edge represents one unit of load which should be distributed among the
vertices defining the edge.
An allocation of load is called balanced if load cannot be moved from a 
vertex to another that is carrying less load.
We analyze the properties of balanced allocations of load.
We extend the concept of balancedness from finite hypergraphs to their local weak limits in the sense of Benjamini and Schramm (2001)
and Aldous and Steele (2004).
To do this, we define a notion of unimodularity for hypergraphs which could be considered an extension of unimodularity in graphs.
We give a variational formula for the balanced load distribution and,
in particular, we characterize it 
in the special case of unimodular hypergraph Galton~Watson processes.
 Moreover, we prove the convergence of the maximum load under some conditions.
Our work is an extension to hypergraphs of Anantharam and Salez (2016),
which considered load balancing in graphs, and is aimed at more 
comprehensively
resolving conjectures of Hajek (1990).
\\[2mm]
\textbf{Keywords:} Load balancing, sparse graph limits, hypergraphs, local weak convergence, objective method, unimodularity, configuration model
\end{abstract}

\section{Introduction}
\label{sec:introduction}


The problem of load balancing is ubiquitous in networks.
As examples, consider the problem of routing
traffic through a communication network or 
of assigning tasks among the servers in a cloud computing framework. 
What is common in these scenarios is a number of servers and a number of tasks whose load should be distributed among the servers.
Examples of servers are 
paths through the network from a given source to a given destination in the routing scenario, or processors in a cloud computing framework. Examples of tasks are 
the  amount of traffic to be routed from the source to the destination or the computational work to be done at the servers, respectively. Further, there are typically restrictions as to which resources are available
to a given task. Performance considerations require that
the allocation of the load of a task among the resources available to it
should be done in a way that optimizes a measure of performance, such as 
delay or queue length. When the problem size is large, it may be expensive to compute the 
detailed characteristics of an optimal or sufficiently good allocation of the load. Instead, it is
interesting to focus  on the statistical characteristics of the allocation, such as the 
empirical distribution of the load faced by a typical resource in the network.  This paper is 
concerned with developing such a viewpoint in the context of a specific kind of load balancing
problem which has broad applicability.



\subsection{Model and Prior Work}
\label{sec:model-prior-work}

We model the load balancing problem by a bipartite graph in which every node on the right represents a task and every node on the left represents a server. Each server is accessible to a certain subset of the tasks. Equivalently each task has access to only a certain subset of the servers.
We view the bipartite graph as a hypergraph, with each vertex of the hypergraph representing a server 
and each hyperedge representing a task. The vertices of a hyperedge are then the servers that are accessible to it. 
Let $\{v_1, \dots, v_n\}$ and $\{e_1, \dots, e_m\}$ denote the set of servers and tasks, or equivalently vertices and hyperedges, respectively. Therefore, $v_i \in e_j$ means that server $v_i$ can be used to contribute to the performance of task $e_j$.  See Figure~\ref{fig:bipartite-hypergraph-LB} for an example.  In general, we might want to consider the scenario where
 task $e_j$ has an amount of load equal to $l_j$, which could be arbitrarily allocated  among the servers $v_i \in e_j$. For simplicity, we will consider in this paper 
 only the case where all the $l_j$ equal $1$, but we leave the discussion general for the moment. Let $\theta$ be an allocation of the load of tasks among the servers, i.e. $\theta(e_j, v_i)$ is the amount of load coming from task $e_j$ assigned to server $v_i$. Hence, $\theta(e_j, v_i) \geq 0$ and $\sum_{v_i \in e_j} \theta(e_j, v_i) = l_j$. For a server $v_i$, let $\partial \theta(v_i)$ be the total amount of load assigned to $v_i$, i.e. $\partial \theta(v_i) = \sum_{e_j: v_i \in e_j} \theta(e_j, v_i)$. 

This formulation of load balancing was 
studied by Hajek \cite{hajek1990performance} who, in particular, formulated
the notion of a {\em balanced allocation}. 
It is natural to expect that a task would be happier to use servers that
are currently handling less load, if available, as opposed to those
handling more load.
An allocation $\theta$ is said to be balanced if
no task desires to change the allocation of its load. For finite load balancing problems, this turns out to be equivalent to the statement that the allocation minimizes $\sum_i f(\partial \theta(v_i))$ for any given fixed strictly convex function $f$. One can think of $\sum_i f(\partial \theta(v_i))$ as the aggregate cost we need to pay to process all the tasks. 
Hajek showed the existence of balanced allocations and uniqueness of the total load at nodes 
under any balanced allocation, 
and suggested algorithms to find a balanced allocation. It is particularly remarkable that
the notion of a balanced allocation does not depend on the specific choice of the 
strictly convex cost function $f$.

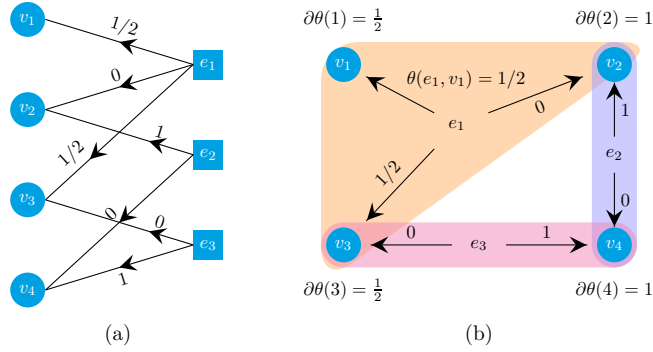
\begin{figure}
  \centering
    \begin{tikzpicture}[scale=0.6]
  \tikzstyle{task} = [fill=Cyan!90!blue,text=white,regular polygon,regular polygon sides=4, inner sep=1pt,scale=0.7]
  \tikzstyle{server} = [fill=Cyan!90!blue,text=white,regular polygon,circle,inner sep=1mm,scale=0.7]
  \tikzstyle{myarrow} = [postaction={decorate}, decoration={markings,
    mark=at position 0.5 with {\arrow[scale=2]{stealth}}}]
    \tikzstyle{myarrowbegin} = [postaction={decorate}, decoration={markings,
      mark=at position 0.3 with {\arrow[scale=2]{stealth}}}]
        \tikzstyle{myarrowend} = [postaction={decorate}, decoration={markings,
          mark=at position 0.7 with {\arrow[scale=2]{stealth}}}]
        \tikzstyle{bigarrow} = [postaction={decorate}, decoration={markings,
          mark=at position 0.99 with {\arrow[scale=2]{stealth}}}]
  \begin{scope}[xshift=-5cm]
    \node[server] (v1) at (-2,3) {$v_1$};
    \node[server] (v2) at (-2,1) {$v_2$};
    \node[server] (v3) at (-2,-1) {$v_3$};
    \node[server] (v4) at (-2,-3) {$v_4$};
    
    \node[task] (e1) at (2,2) {$e_1$};
    \node[task] (e2) at (2,0) {$e_2$};
    \node[task] (e3) at (2,-2) {$e_3$};
    
    \draw[myarrow] (e1.west) -- node[above,sloped,scale=0.7] {$1/2$} (v1.east);
    \draw[myarrow] (e1.west) -- node[above,sloped,scale=0.7] {$0$} (v2.east);
    \draw[myarrowend] (e1.west) -- node[near end,sloped,above,scale=0.7] {$1/2$} (v3.east);
    
    \draw[myarrowbegin] (e2.west) -- node[near start, sloped,above,scale=0.7] {$1$} (v2.east);
    \draw[myarrow] (e2.west) -- node[above,sloped,scale=0.7] {$0$} (v4.east);
    
    \draw[myarrowbegin] (e3.west) -- node[near start, sloped,above,scale=0.7] {$0$} (v3.east);
    \draw[myarrow] (e3.west) -- node[below,sloped,scale=0.7] {$1$} (v4.east);
\node[scale=0.8] at (0,-4) {(a)};
  \end{scope}
  \begin{scope}[xshift=3cm]
    \coordinate (v1h) at (-3,2);
    \coordinate (v2h) at (3,2);
    \coordinate (v3h) at (-3,-2);
    \coordinate (v4h) at (3,-2);
    \coordinate (e1h) at ($0.33*(v1h) + 0.33*(v2h) + 0.33*(v3h)+(0.5,0)$);
    \coordinate (e2h) at ($0.5*(v2h) + 0.5*(v4h)$);
    \coordinate (e3h) at ($0.5*(v3h) + 0.5*(v4h)$);
    
    \fill [rounded corners=0.5cm, inner sep=1cm,fill=Orange!60,opacity=0.5] ($(v1h)+(-0.5,0.5)$) -- ($(v2h)+(1,0.5)$) -- ($(v3h)+(-0.5,-0.9)$) -- cycle;
    \fill [rounded corners=0.3cm, inner sep=3mm, fill = blue!40,opacity=0.5] ($(v4h)+(-0.5,-0.5)$) rectangle ($(v2h)+(0.5,0.5)$);
    \fill [rounded corners=0.3cm, inner sep=3mm, fill = magenta!60,opacity=0.5] ($(v3h)+(-0.5,-0.5)$) rectangle ($(v4h)+(0.5,0.5)$);
    
    \node[server] at (v1h) {$v_1$};
    \node[server] at (v2h) {$v_2$};
    \node[server] at (v3h) {$v_3$};
    \node[server] at (v4h) {$v_4$};
    \node[scale=0.7] at (e1h) {$e_1$};
    \node[scale=0.7] at (e2h) {$e_2$};
    \node[scale=0.7] at (e3h) {$e_3$};
    
    \draw[bigarrow] ($(e1h)!0.2!(v1h)$) -- node[above right,scale=0.7] {$\theta(e_1, v_1) = 1/2$} ($(e1h)!0.8!(v1h)$);
    \draw[bigarrow] ($(e1h)!0.2!(v2h)$) -- node[above,scale=0.7,sloped,below] {$0$} ($(e1h)!0.8!(v2h)$);
    \draw[bigarrow] ($(e1h)!0.2!(v3h)$) -- node[above,scale=0.7,sloped] {$1/2$} ($(e1h)!0.8!(v3h)$);
    
    \draw[bigarrow] ($(e2h)!0.2!(v2h)$) -- node[right,scale=0.7] {$1$} ($(e2h)!0.8!(v2h)$);
    \draw[bigarrow] ($(e2h)!0.2!(v4h)$) -- node[right,scale=0.7] {$0$} ($(e2h)!0.8!(v4h)$);
    
    \draw[bigarrow] ($(e3h)!0.2!(v3h)$) -- node[above,scale=0.7] {$0$} ($(e3h)!0.8!(v3h)$);
    \draw[bigarrow] ($(e3h)!0.2!(v4h)$) -- node[above,scale=0.7] {$1$} ($(e3h)!0.8!(v4h)$);
    
    \node[scale=0.7] at ($(v1h) + (0,1)$) {$\partial \theta(1) = \frac{1}{2}$};
    \node[scale=0.7] at ($(v3h) + (0,-1)$) {$\partial \theta(3) = \frac{1}{2}$};
    \node[scale=0.7] at ($(v2h) + (0,1)$) {$\partial \theta(2) = 1$};
    \node[scale=0.7] at ($(v4h) + (0,-1)$) {$\partial \theta(4) = 1$};
\node[scale=0.8] at (0,-4) {(b)};
  \end{scope}
\end{tikzpicture}
\caption{\label{fig:bipartite-hypergraph-LB} Load balancing with 3 tasks and 4 servers. (a) illustrates the bipartite representation together with an allocation. While the load of $e_1$ could be served by nodes in $\{v_1, v_2, v_3\}$, half of its load is being assigned to $v_1$ and the other half to $v_3$, i.e. $\theta(e_1, v_1) = 1/2$, $\theta(e_1, v_2) = 0$ and $\theta(e_1, v_3) = 1/2$. (b) shows the hypergraph representation. The total load at a node $i$ is denoted by $\partial \theta(i)$. For instance, $\partial \theta(v_3) = \theta(e_1, v_3) + \theta(e_3, v_3) = 1/2$.}
\end{figure}

With the aim of understanding the statistical characteristics of balanced allocations in 
large load balancing problems, 
Hajek assumed that each task could be performed by only two servers -- 
hence the underlying hypergraph reduces to a graph -- and he assumed that
each edge in this graph carries one unit of load. 
He then studied 
such a load balancing problem  in large
random graphs \cite{hajek1990performance}. 
In particular, Hajek considered the sparse \ER  model to generate these graphs, where $\alpha n$ edges are distributed among $n$ vertices uniformly at random, with $\alpha$ being a fixed parameter. It is well known that the asymptotic structure of the local neighborhood of a typical vertex in a sparse \ER model is given by a Poisson Galton--Watson tree (see, for instance, \cite[Proposition~2.6]{dembo2010gibbs}
for a precise formulation of this statement). 
This suggests that  the behavior of balanced allocations in Galton--Watson trees might be a good proxy for  the load distribution in large \ER graphs. Hajek conjectured that the recursive nature of a Galton--Watson process helps one analyze the distribution of balanced allocations by studying fixed point equations. He was even able to suggest the form of the fixed point equation for the Poisson Galton--Watson tree. However, it turns out that this approach is more subtle than it looks. For one thing, Hajek realized that the notion of balanced allocation in an infinite graph as a proxy for large graphs is not well defined \cite{hajek1996balanced}. See Figure~\ref{fig:hajek-counter} for an example. 

\begin{figure}
  \centering
  \begin{tikzpicture}[scale=0.6]
\begin{scope}[xshift=-4cm, scale=0.7]
    \draw[Cyan,thick] (0,0) -- node[sloped,above, near start,scale=0.9] {$\longrightarrow$} (-3,-2);
    \draw[Cyan,thick] (0,0) -- node[sloped,above,scale=0.7] {$\longleftarrow$} (0,-2);
    \draw[Cyan,thick] (0,0) -- node[sloped,above, near start,Orange,scale=0.9] {$\longrightarrow$} (3,-2);
    \node[Root] at (0,0) {};
    \node[Node] at (-3,-2) {};
    \node[Node] at (0,-2) {};
    \node[Node] at (3,-2) {};

    \begin{scope}[xshift=-3cm, yshift=-2cm]
      \node[Node] at (-1,-2) {};
      \node[Node] at(1,-2) {};
      \draw[Cyan,thick] (0,0) -- node[sloped, above, near start,scale=0.9] {$\longrightarrow$} (-1,-2);
      \draw[Cyan,thick] (0,0) -- node[sloped, above, near start,scale=0.9] {$\longleftarrow$}  (1,-2);
    \end{scope}

    \begin{scope}[xshift=0cm, yshift=-2cm]
      \node[Node] at (-1,-2) {};
      \node[Node] at(1,-2) {};
      \draw[Cyan,thick] (0,0) -- node[sloped, above, near start,scale=0.9] {$\longrightarrow$} (-1,-2);
      \draw[Cyan,thick] (0,0) -- node[sloped, above, near start,scale=0.9] {$\longleftarrow$}  (1,-2);
    \end{scope}

    \begin{scope}[xshift=3cm, yshift=-2cm]
      \node[Node] at (-1,-2) {};
      \node[Node] at(1,-2) {};
      \draw[Cyan,thick] (0,0) -- node[sloped, above, near start,scale=0.9] {$\longrightarrow$} (-1,-2);
      \draw[Cyan,thick] (0,0) -- node[sloped, above, near start,Orange,scale=0.9] {$\longrightarrow$}  (1,-2);
    \end{scope}
    
    \node[rotate=90] at (-3,-5) {$\dots$};
    \node[rotate=90] at (0,-5) {$\dots$};
    \node[rotate=90] at (3,-5) {$\dots$};

    \node[scale=0.85] at (0,-6.5) {(a)};
\end{scope}

\begin{scope}[xshift=4cm, scale=0.7]
    \draw[Cyan,thick] (0,0) -- node[sloped,above, near start,scale=0.9] {$\longleftarrow$} (-3,-2);
    \draw[Cyan,thick] (0,0) -- node[sloped,above,scale=0.9] {$\longrightarrow$} (0,-2);
    \draw[Cyan,thick] (0,0) -- node[sloped,above, near start,Orange,scale=0.9] {$\longleftarrow$} (3,-2);
    \node[Root] at (0,0) {};
    \node[Node] at (-3,-2) {};
    \node[Node] at (0,-2) {};
    \node[Node] at (3,-2) {};

    \begin{scope}[xshift=-3cm, yshift=-2cm]
      \node[Node] at (-1,-2) {};
      \node[Node] at(1,-2) {};
      \draw[Cyan,thick] (0,0) -- node[sloped, above, near start,scale=0.9] {$\longleftarrow$} (-1,-2);
      \draw[Cyan,thick] (0,0) -- node[sloped, above, near start,scale=0.9] {$\longrightarrow$}  (1,-2);
    \end{scope}

    \begin{scope}[xshift=0cm, yshift=-2cm]
      \node[Node] at (-1,-2) {};
      \node[Node] at(1,-2) {};
      \draw[Cyan,thick] (0,0) -- node[sloped, above, near start,scale=0.9] {$\longleftarrow$} (-1,-2);
      \draw[Cyan,thick] (0,0) -- node[sloped, above, near start,scale=0.9] {$\longrightarrow$}  (1,-2);
    \end{scope}

    \begin{scope}[xshift=3cm, yshift=-2cm]
      \node[Node] at (-1,-2) {};
      \node[Node] at(1,-2) {};
      \draw[Cyan,thick] (0,0) -- node[sloped, above, near start,scale=0.9] {$\longleftarrow$} (-1,-2);
      \draw[Cyan,thick] (0,0) -- node[sloped, above, near start,Orange,scale=0.9] {$\longleftarrow$}  (1,-2);

    \end{scope}
    
    \node[rotate=90] at (-3,-5) {$\dots$};
    \node[rotate=90] at (0,-5) {$\dots$};
    \node[rotate=90] at (3,-5) {$\dots$};
    \node[scale=0.85] at (0,-6.5) {(b)};
\end{scope}
  \end{tikzpicture}

\caption{\label{fig:hajek-counter} Hajek's example to show non--uniqueness of the load for infinite graphs. Consider the rooted 3--regular graph with infinite depth as shown. We send all of the unit load corresponding to each edge in the direction of the shown arrows. The red path goes to infinity in both pictures. The  allocation in (a) makes the total load at every vertex equal to 2 while that in (b) makes the total load at all vertices equal to 1. Therefore, both are balanced.}
\end{figure}
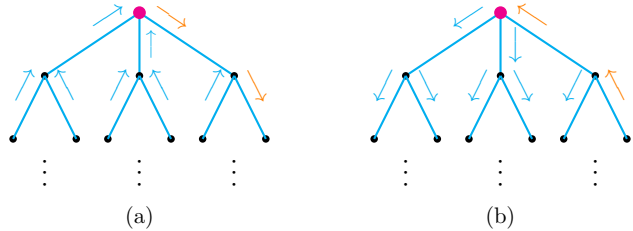

Hajek's conjecture for the graph regime (i.e. when each task could only be distributed among two servers) was settled by Anantharam and Salez \cite{anantharam2016densest}. They employed the framework of {\em local weak convergence}, also known as the
{\em objective method} \cite{benjamini2001schramm}, \cite{aldous2004objective}, \cite{aldous2007processes}. This framework introduces a notion of convergence for a sequence of finite graphs by representing each finite graph as a probability distribution on
rooted graphs and then discussing convergence in probabilistic terms. Roughly speaking, 
the operational meaning of this technique is that convergence holds when the distribution of the local neighborhood of a vertex chosen uniformly at random in the graph converges to that of the local neighborhood around the root in the limit. For instance,
a sequence of sparse \ER graphs converges  in this sense  to a Poisson Galton-Watson tree, consistent with 
our established intuition in this case.

A random rooted graph that can appear as the limit of a sequence of finite graphs is called {\em sofic}.
Not all random rooted graphs are sofic. The reason is that all the vertices in a finite graph have the same chance of being chosen as the root. This should manifest itself as some form of stationarity in the limit, i.e.\ the limiting object should be invariant under changing the root. This property is called {\em unimodularity}, which is a necessary condition for being sofic. Whether the converse is true is an open problem \cite{aldous2007processes}.

Anantharam and Salez settled Hajek's conjecture by first defining a notion of balancedness for unimodular random rooted graphs. Moreover, they showed that if a sequence of finite graphs $G_n$ converges to a random rooted graph in the above local weak sense, the total load associated to a balanced allocation  at a vertex chosen uniformly at random in $G_n$ converges in distribution to  the total load associated to the balanced allocation at the root of the limit. Additionally, they managed to express a certain functional of the  distribution of the load at the root of the Galton--Watson tree in terms of a fixed point distributional equation, settling Hajek's conjecture in the graph regime. Beyond this, they also  proved the convergence of the maximum load for a sequence of finite graphs resulting from a certain configuration model to that of their local weak limit, under some additional conditions.

\subsection{Our Contributions}
\label{sec:our-contributions}

We study the above load balancing problem in the more general regime where each task could have access to more than two servers, i.e.\ the underlying network is a hypergraph instead of a graph. 

Our machinery for deriving results analogous to those in the graph regime will be a generalized method of local weak convergence on hypergraphs. One novelty of our development  is to introduce a notion analogous to unimodularity for processes on random rooted hypergraphs. We believe that this generalized framework could be of independent interest in a variety of problems in which the underlying model is best expressed in terms of hypergraphs rather than graphs. 

In particular, we prove that for any unimodular probability distribution on the set of rooted hypergraphs with finite expected degree, there exists a balanced allocation which is consistent with the local weak limit theory, i.e.\ the load distribution of a sequence of hypergraphs converges to that of the limit. For a special class of branching process on rooted hypertrees which is a generalization of Galton--Watson processes, we show that the distribution of the load at the root can be specified via a fixed point distributional equation. Finally, we study the convergence of the maximum load for a sequence of random hypergraphs generated from a configuration model to that of the limit, under some additional conditions.


\section{Prerequisites and Notations}
\label{sec:prerequisite}

Throughout this paper $\reals$ and $\reals_{\geq 0}$ denote the set of real numbers and nonnegative real numbers, respectively. Moreover, $\mathbb{Q}$ denotes the set of rational numbers. 
$\nats$ denotes the set of positive integers and $\natsz := \nats \cup
\{0\}$.
For a Polish space $X$, let $\mP(X)$ and $\mM(X)$ respectively denote the set of probability measures and nonnegative finite measures on $X$, with respect to the Borel $\sigma$--field of $X$.  
We use the abbreviations ``a.s.'' and ``a.e.'' for the phrases ``almost surely'' and ``almost everywhere'', respectively.
For a real number $x \in \reals$, we denote $\max \{ x, 0 \}$ by $x^+$, and we denote
$\min \{ x^+,1 \}$ by $\left[ x \right]_0^1$.

For a Polish space $X$, we say that a sequence of probability measures $\mu_n$ converges weakly to a probability measure $\mu \in \mP(X)$, and write $\mu_n \Rightarrow \mu$, if for any bounded continuous function $f: X \rightarrow \reals$ we have 
\begin{equation*}
  \lim_{n \rightarrow \infty} \int f d \mu_n = \int f d \mu.
\end{equation*}
See \cite{billingsley1971weak} and \cite{billingsley2013convergence} for more details on weak convergence of probability measures. 

Given two measurable spaces $(X_1, \mF_1)$ and $(X_2, \mF_2)$, a measurable mapping $f: X_1 \rightarrow X_2$, and a nonnegative measure $\mu_1$ on $\mF_1$, the pushforward measure $f_*(\mu_1)$ on $\mF_2$ is defined by
\begin{equation*}
  f_*(\mu_1)(A) = \mu_1(f^{-1}(A)),
\end{equation*}
for $A \in \mF_2$.

We write $:=$ for equality by definition. We use the terms ``node" and ``vertex" interchangeably.


\subsection{Hypergraphs}
\label{sec:hypergraphs}

We work with simple hypergraphs defined on a countable vertex set, where each edge is a finite subset of the vertex set. For a hypergraph $H$, the sets of vertices and edges are denoted by $V(H)$ and $E(H)$, respectively. We write $H$ as $\langle V, E \rangle$, where $V = V(H)$ and $E = E(H)$. 
We say a hypergraph is simple if
$E(H) \subset 2^{V(H)}$. This means that in any edge each vertex
can show up at most once, and that each subset of vertices can show up at most once as an 
edge. All hypergraphs appearing in this paper will be simple, unless 
otherwise stated.
For a vertex $i \in V(H)$, denote its degree by $\deg_H(i) := |\{ e \in E(H): e \ni i \}|$. 
For a given hypergraph $H$, let
  \begin{equation}
    \label{eq:evpair-define}
    \evpair(H) := \{ (e, i) : e \in E(H), i \in e \}
  \end{equation}
denote the set of all edge-vertex pairs in the hypergraph.


\begin{definition}
A hypergraph $H$ is said to be locally finite if $\deg_H(i)$ is finite for all $i \in V(H)$ and  $e$ is finite for all $e \in E(H)$.
\end{definition}

Note that the above definition does not imply that there is a uniform bound on edge sizes or vertex degrees. Hence, a locally finite hypergraph can have arbitrarily large edges or vertex degrees. Throughout this paper, we assume that all the hypergraphs are locally finite, unless otherwise stated.
Thus, by default, the term hypergraph in this paper means a 
simple, locally finite hypergraph on a countable vertex set.

For technical reasons, it is sometimes easier to work with bounded 
hypergraphs in proofs and then relax the boundedness condition.
\begin{definition}
  \label{def:bounded-hypergraph}
  A hypergraph $H$ is said to be bounded if $\deg_H(i) \leq \Delta$ for all $i \in V(H)$ and also $|e| \leq L$ for all $e \in E(H)$, for finite constants $\Delta$ and $L$.
\end{definition}

A path from node $i$ to node $j$ is an alternating vertex--edge sequence $i_0, e_1, i_1, e_2, i_2, \dots, e_n, i_n$ 
with $i_0 = i$ and $i_n = j$, and $i_k \in e_{k+1}$ for $0 \leq k < n$.
The length of such a path is defined to be $n$. The distance between vertices $i$ and $j$, denoted by $d_H(i,j)$, is defined to be the length of the shortest path between $i$ and $j$ if $i \neq j$, and $0$ when $i = j$.
A path is called closed if $i_0 = i_n$.


\begin{definition}
  \label{def:hypertree}
A hypergraph $H$ is called a hypertree if there is no closed path $i_0, e_1, i_1, \dots, e_{n-1},i_{n-1}, e_n, i_n$ with $n \ge 2$ 
such that $i_j \neq i_l$ and $e_j \neq e_l$ for 
$1 \le j \neq l \le n$.
\end{definition}


\begin{rem}
\label{rem:hyperforest}
Note that
a hypertree
need not be connected.
It is straightforward to prove that if there is a 
path between vertices $i$ and $j$ in a hypertree, then the shortest
path between these vertices is unique.
\end{rem}


\begin{definition}
\label{def:E-H_VHid_DHid}
For a hypergraph $H$ and a subset $W \subset V(H)$, define $E_H(W) := \{ e \in E(H): e \subset W \}$. $E_H(W)$ is comprised of the edges of $H$ with all endpoints in the set $W$. For $i \in V(H)$ and $d \geq 0$, define $V^H_{i, d} := \{j \in V(H): d_H(i, j) \leq d \}$ and $D^H_{i, d} := \{ j \in V(H): d_H(i, j) = d \}$. In particular, $V^H_{i,0} = D^H_{i,0} = \{i\}$.  
\end{definition}

\begin{definition}
\label{def:rooted}
A vertex rooted hypergraph is a hypergraph $H$ with a distinguished vertex $i \in V(H)$. We denote this by $(H, i)$. 
An edge-vertex rooted hypergraph is a hypergraph with a 
distinguished edge $e \in E(H)$ and a distinguished vertex $i \in V(H)$ such that $i \in e$. This is denoted by $(H, e, i)$ .
\end{definition}


\begin{definition}
  \label{def:isomorphism}
  We say that two hypergraphs $H$ and $H'$ are isomorphic and write $H \equiv H'$ when there is a bijection $\phi: V(H) \rightarrow V(H')$ such that $e \in E(H)$ if and only if $\phi(e) := \{ \phi(j) : j \in e \} \in E(H')$. Also we say two vertex rooted hypergraphs $(H, i)$ and $(H', i')$ are isomorphic and write $(H, i) \equiv (H', i')$ if the above bijection $\phi$ exists and we have $\phi(i) = i'$ as well. Furthermore, we say two edge-vertex rooted hypergraphs $(H, e, i)$ and $(H', e', i')$ are isomorphic and write $(H, e, i) \equiv (H', e', i')$ if the above bijection exists, and we have
  $\phi(i) = i'$ and $\phi(e) = e'$.
\end{definition}


Instead of working with global isomorphisms as above, we can  consider local isomorphisms, i.e. comparing two rooted hypergraphs up to some given depth.

\begin{definition}
  \label{def:isomorphism-depth-d}
  We say two vertex rooted hypergraphs $(H, i)$ and $(H', i')$  are isomorphic up to depth $d$ and write $(H, i) \equiv_d (H', i')$ if their truncations up to depth $d$ are isomorphic, i.e. $\langle V^H_{i,d} , E_H(V^H_{i,d}) \rangle \equiv \langle V^{H'}_{i',d} , E_{H'}(V^{H'}_{i',d}) \rangle$ and also $\phi(i) = i'$, where $\phi:V^H_{i, d} \rightarrow V^{H'}_{i',d}$ is the vertex bijection establishing this isomorphism. Also, for $d \geq 1$, we say two  edge-vertex rooted hypergraphs $(H, e, i)$ and $(H', e', i')$ are isomorphic up to depth $d$ and write $(H, e, i) \equiv_d (H', e', i')$ if $\langle V^H_{i,d} , E_H(V^H_{i,d}) \rangle \equiv \langle V^{H'}_{i',d} , E_{H'}(V^{H'}_{i',d}) \rangle$, $\phi(i) = i'$, and $\phi(e) = e'$, where $\phi(e) := \{ \phi(j): j\in e\}$. Here $\phi:V^H_{i, d} \rightarrow V^{H'}_{i',d}$ is the vertex bijection
establishing this isomorphism.
\end{definition}


\begin{definition}
  \label{def:hypergraph-local-embedding}
  Given two hypergraphs $H$ and $H'$, for $i \in V(H)$ and $i' \in V(H')$ we say that $(H, i)$ has a local embedding up to depth $d\geq 1$ into $(H', i')$ and write $(H, i) \hookrightarrow_d (H', i')$ if there is an injective mapping $\phi: V^H_{i, d} \hookrightarrow V^{H'}_{i', d}$ such that:
  \begin{enumerate}
  \item $\phi(i) = i'$, and
  \item for all $e \in E_H(V^H_{i,d})$, we have $\phi(e) \in E(H')$ where  $\phi(e) := \{ \phi(j) : j \in e \}$.
  \end{enumerate}
\end{definition}

  \begin{definition}
    \label{def:H-i_d}
    Given a hypergraph $H$, a node $i \in V(H)$ and $d \geq 1$, let $(H, i)_d$ denote the vertex rooted hypergraph $(\langle V^H_{i, d}, E_H(V^H_{i,d}) \rangle, i)$. In fact, $(H, i)_d$ is  the $d$--neighborhood of vertex $i$.
  \end{definition}




\subsection{Balanced allocations on a hypergraph}
\label{sec:balanc-alloc-single}

\begin{definition}
\label{def:allocation}
An allocation on  hypergraph $H=\langle V,E \rangle$ is a mapping $\theta: \evpair(H) \rightarrow [0,1]$ such that $\theta(e ,i)$ with $i \in e \in E$ tells us how much load from resource $e$ is being given to node $i$. 
More formally, it is characterized by the properties:
\begin{equation*}
  \begin{gathered}
    \theta(e ,i) \geq 0~, \qquad \forall e \in E(H),  i \in e, \mbox{ and }\\
    \sum_{j \in e} \theta(e ,j) = 1~, \qquad \forall e \in E(H)~.
  \end{gathered}
\end{equation*}
\end{definition}

In any allocation, a given vertex $i \in V(H)$ receives a portion $\theta(e ,i)$ of the total unit load of resource $e$. The total load at the vertex is then the sum of portions it receives from resources $e \ni i$. The following
definition establishes the notation to discuss this load.

\begin{definition}
\label{def:partial-single-hypergraph}
Given an allocation $\theta$ on a hypergraph $H = \langle V, E \rangle$, define the function $\partial \theta: V(H) \rightarrow \reals_{\geq 0}$ by
\begin{equation*}
  \partial \theta(i) := \sum_{e: i \in e} \theta(e ,i)~, \qquad \text{for all } i \in V(H).
\end{equation*}
\end{definition}

\begin{definition}
  \label{def:balanced-allocation}
  For a hypergraph $H = \langle V, E \rangle$, an allocation $\theta$ is called balanced if for all $e \in E$ and $i,j \in e$ we have 
  \begin{equation*}
    \partial \theta(i) > \partial \theta(j) \quad \Rightarrow \quad \theta(e, i) = 0.
  \end{equation*}
\end{definition}

Much of the paper is concerned with understanding the structure 
of balanced allocations on hypergraphs, and of the load resulting from such allocations. 
As we will soon see via examples, balanced allocations can exhibit
phenomena analogous to phase transitions in statistical mechanics
models. This is because the hard constraint defining balancedness can 
be thought of as analogous to a zero temperature limit. 
Following this analogy further, it is therefore convenient to deal with 
what might be called a positive temperature notion of balancedness, and
then to send the temperature to zero. This is captured in the concept of $\epsilon$--balance.

\begin{definition}
  \label{def:epsilon-balanced-specific-hypergraph}
  For a hypergraph $H = \langle V, E \rangle$, an allocation $\theta$ is called $\epsilon$--balanced, if for all $e \in E$ and $i \in e$ we have 
  \begin{equation*}
    \theta(e, i) = \frac{\exp (- \partial \theta(i)/\epsilon)}{\sum_{j \in e} \exp (- \partial \theta(j)/ \epsilon)}.
  \end{equation*}
\end{definition}

\begin{rem}
  \label{rem:motivation-of-epsilon-balanced allocation}
  Let $\theta$ be an $\epsilon$-balanced allocation on a hypergraph $H$.
  Note that if $e \in E$ and $i, j \in e$ are such that $\partial \theta(i) > \partial \theta(j)$ then 
 \begin{equation*}
   \theta(e, i) = \frac{\exp(- \partial \theta(i) / \epsilon)}{\sum_j \exp ( - \partial \theta(j) / \epsilon)} \leq \frac{1}{1+ \exp \left ( \frac{\partial \theta(i) - \partial \theta(j)}{\epsilon} \right )}.
 \end{equation*}
 Roughly speaking, if $\partial \theta(i) > \partial \theta (j)$ and $\epsilon$ is small, then $\theta(e, i) \approx 0$ and hence $\theta$ is approximately balanced. 
 Also, roughly speaking, the smaller $\epsilon$ is, the more balanced an $\epsilon$-balanced allocation is.
\end{rem}

In the above, we defined balancedness when all the loads come from the edges of the hypergraph. We can generalize this to the case where, in addition to the internal load imposed by the edges, we have external load as well. External load is modeled by a function $b: V(H) \rightarrow \reals$, called the {\em baseload function}. For a vertex $i \in V(H)$, $b(i)$ denotes the external load applied to node $i$. Throughout this paper, we assume that each baseload function is bounded, but we do not assume a uniform bound on all baseload functions. 
More precisely, for each baseload function $b$ on a given hypergraph $H$, we assume that there exists $M < \infty$ such that $|b(i)| < M$ for all $i \in V(H)$, where the constant $M$ may depend on $H$ and $b$.
The concept of balancedness can be extended to the scenario with baseloads as follows.

\begin{definition}
  \label{def:balanced-allocation-with-baseload}
  For a hypergraph $H = \langle V, E \rangle$, together with a baseload function $b: V(H) \rightarrow \reals$, an allocation $\theta:\evpair(H) \rightarrow [0,1]$ is called balanced with respect to the baseload $b$, if for all $e \in E$ and $i,j \in e$ we have 
  \begin{equation*}
    \partial \theta(i) + b(i) > \partial \theta(j) +b(j) \quad \Rightarrow \quad \theta(e, i) = 0.
  \end{equation*}
\end{definition}

Note that $\partial \theta(i) + b(i)$ is the total load at node $i$ where $\partial \theta(i)$ is the internal load and $b(i)$ is the contribution from the external load. We use the notation $\partial_b\theta$ as a shorthand for $\partial \theta + b$.

The concept of an $\epsilon$--balanced allocations can be similarly 
extended to the scenario with baseloads.

\begin{definition}
\label{def:epsilon-balanced-with-baseload}
For a given hypergraph $H = \langle V, E \rangle$, together with a baseload function $b: V(H) \rightarrow \reals$, we say an allocation $\theta:\evpair(H) \rightarrow [0,1]$ is $\epsilon$--balanced with respect to the baseload $b$, if for all $e \in E(H)$ and $i \in e$ we have 
\begin{equation}
  \label{eq:epsilon-baseload-definition}
  \theta(e, i) = \frac{\exp \left ( - \pbt(i)/\epsilon \right )}{\sum_{j \in e} \exp \left ( - \pbt(j)/\epsilon \right ) }.
\end{equation}
\end{definition}

It is known that if the hypergraph is finite, then balanced allocations exist 
with respect to any baseload and the resulting load vector is the same for all balanced allocations for the given baseload (see Theorem~2 and Corollary~5 in \cite{hajek1990performance}). This result is stated in the following
proposition.


\begin{prop}
  \label{prop:total-load-unique-finite-hypergraphs}
  If $H = \langle V, E \rangle$ is a finite hypergraph, and $b: V(H) \rightarrow \reals$ is a given baseload function, then there exists at least one balanced allocation $\theta$ on $H$ with respect to the baseload $b$. Moreover, if $\theta$ and $\theta'$ are two balanced allocations on $H$ with respect to $b$, then $\partial_b \theta(i) = \partial_b \theta'(i)$ for all $i \in V(H)$. 
\end{prop}



Later, in Section~\ref{sec:epsilon-balancing}, we will study $\epsilon$--balanced allocations with baseload for hypergraphs that are not necessarily finite. In particular, we will show in Corollary~\ref{sec:epsilon-balancing} therein that for bounded hypergraphs, for any baseload function, the total load at any vertex corresponding to any $\epsilon$--balanced allocation is uniquely defined.

Note that for finite hypergraphs, although the balanced allocations with respect to a given baseload might not be uniquely defined, the total loads at the vertices resulting from any two balanced allocations necessarily have to be the same. See Figure~\ref{fig:trianlge-graph-two-balanced} for an example. The case for infinite hypergraphs is more complicated though; in this case, the loads at the vertices also may not be unique, see Figure~\ref{fig:hajek-counter}. See \cite{hajek1996balanced} for more discussion on this. 
However, we can state a weak uniqueness result in this case. The proof is given in Appendix~\ref{sec:weak-uniq-balan}.


\begin{figure}
  \centering
     \begin{tikzpicture}
      \begin{scope}
        \draw[Cyan,thick] (90:1) -- (-30:1) -- (210:1) -- cycle;
        \node[Node] at (90:1) {};
        \node[Node] at (-30:1) {};
        \node[Node] at (210:1) {};

        \draw[->,Orange,thick] (40:0.7) -- node[above right,scale=0.8] {1} +(120:0.5);
        \draw[->,Orange,thick] (160:0.7) -- node[above left,scale=0.8] {1} +(240:0.5);
        \draw[->,Orange,thick] (280:0.7) -- node[below right,scale=0.8] {1} +(0:0.5);

      \end{scope}

      \begin{scope}[xshift=5cm]
        \draw[Cyan,thick] (90:1) -- (-30:1) -- (210:1) -- cycle;
        \node[Node] at (90:1) {};
        \node[Node] at (-30:1) {};
        \node[Node] at (210:1) {};
        
        \draw[->,Orange,thick] (40:0.7) -- node[above right,scale=0.8] {$1/2$} +(120:0.5);
        \draw[->,Orange,thick] (20:0.7) -- node[above right,scale=0.8] {$1/2$} +(300:0.5);
        
        \draw[->,Orange,thick] (160:0.7) -- node[above left,scale=0.8] {$1/2$} +(240:0.5);
        \draw[->,Orange,thick] (140:0.7) -- node[above left,scale=0.8] {$1/2$} +(60:0.5);
        
        \draw[->,Orange,thick] (280:0.7) -- node[below right,scale=0.8] {$1/2$} +(0:0.5);
        \draw[->,Orange,thick] (260:0.7) -- node[below left,scale=0.8] {$1/2$} +(180:0.5);

      \end{scope}
      
    \end{tikzpicture}


\caption{\label{fig:trianlge-graph-two-balanced} A graph with 3 vertices and three edges and two different balanced allocations with zero baseload. Note that the total load at each vertex is the same for the two allocations and is equal to 1 at each vertex.}
\end{figure}
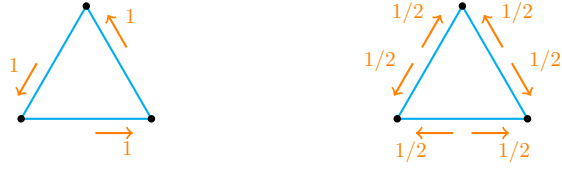



\begin{prop}
  \label{prop:weak-uniqueness}
  Given the hypergraph $H = \langle V, E \rangle$ with the baseload function $b: V(H) \rightarrow \reals$, suppose $\theta$ and $\theta'$ are two balanced allocations on $H$ with respect to the baseload $b$. If $\sum_{i \in V(H)} |\pbt(i) - \pbt'(i) | < \infty$ then $\pbt(i) = \pbt'(i)$ for all $i \in V(H)$.
\end{prop}


\subsection{$\mH_*$ and $\mH_{**}$}
\label{sec:mh_-mh}

It is easy to check that the isomorphism between vertex rooted 
hypergraphs defined in Definition~\ref{def:isomorphism} is an equivalence relation. For a vertex rooted hypergraph $(H, i)$, let $[H, i]$ denote the equivalence class corresponding to $(H, i)$. Also, the isomorphism between 
edge-vertex rooted hypergraphs defined in Definition~\ref{def:isomorphism} is an equivalence relation. Let $[H, e, i]$ be the equivalence class corresponding to $(H, e, i)$. 


\begin{definition}
\label{def:H*-H**}
Let $\mH_*$ be the set of all equivalence classes of connected vertex rooted hypergraphs and $\mH_{**}$ the set of all 
equivalence classes of connected edge-vertex rooted hypergraphs. Hence, each element of $\mH_*$ is of the form $[H, i]$, where $[H,i]$ denotes the equivalence class of $(H, i)$, where $i\in V(H)$ for some
connected hypergraph $H = \langle V, E \rangle$. Similarly, each element of $\mH_{**}$ is of the form $[H, e, i]$, where $e \in E(H)$ and $i\in V(H)$ such that $i \in e$ for some connected hypergraph $H = \langle V, E \rangle$.
\end{definition}


\begin{definition}
 For two vertex rooted hypergraphs $(H, i)$ and $(H', i')$, define
  \begin{equation*}
    d_{*}((H, i) , (H', i')) := \frac{1}{1 +  m^*}~,
  \end{equation*}
where $m^* := \sup \{ m \geq 1: (H, i) \equiv_m (H', i') \}$, and $m^* := 0$ if there is no $m \geq 1$ satisfying this.
  For two equivalence classes $[H, i] \in \mH_*$ and $[H', i'] \in \mH_*$, define $d_{\mH_*}([H, i], [H', i'])$ to be $d_*((H, i), (H', i'))$ where $(H, i)$ and $(H', i')$ are arbitrary members of $[H, i]$ and $[H', i']$, respectively.
  For two edge-vertex rooted hypergraphs $(H, e, i)$ and $(H', e', i')$, define 
  \begin{equation*}
    d_{**}((H, e, i) , (H', e', i')) := \frac{1}{1 +  m^* },
  \end{equation*}
 where $m^* := \sup \{ m \geq 1: (H, e, i) \equiv_m (H', e', i') \}$, and $m^* := 0$ if there is no $m \geq 1$ satisfying this.
  For two equivalence classes $[H, e, i] \in \mH_{**}$ and $[H', e', i'] \in \mH_{**}$, define $d_{\mH_{**}}([H, e, i], [H', e', i'])$ to be $d_{**}((H, e, i), (H', e', i'))$ where $(H, e, i)$ and $(H', e', i')$ are arbitrary members of $[H, e, i]$ and $[H', e', i']$, respectively.
\end{definition}

Since all members of $[H, i]$ are isomorphic, it is not difficult to see that 
$d_{\mH_*}$ is well--defined. 
Note that $(H, i) \equiv_m (H', i')$ and $(H', i') \equiv_{m'} (H'', i'')$ implies $(H, i) \equiv_{\min \{ m, m' \}} (H'', i'')$. Hence $d_{\mH_*}$ satisfies the triangle inequality. Moreover, $d_{\mH_*}([H, i], [H', i']) = 0$ iff $(H, i) \equiv_m (H', i')$ for all $m$, i.e. $[H, i] = [H', i']$. Hence $d_{\mH_*}$ 
defines a metric on $\mH_*$.
We will show in Appendix~\ref{sec:mh_-mh_-are} that $\mH_*$ with 
the metric $d_{\mH_*}$ is a Polish space (see Corollary~\ref{cor:H*-H**-Polish}).
Similarly, $d_{\mH_{**}}$ is well defined and gives a metric on $\mH_{**}$. In Appendix~\ref{sec:mh_-mh_-are} we will also show that $\mH_{**}$ with $d_{\mH_{**}}$ is a Polish space. 


\begin{rem}
\label{rem:H*-function-equivalence-class-relaxation}
One can think of a function $f$ on $\mH_*$ as a function on vertex rooted hypergraphs which is loyal to the isomorphism relation, i.e. $f((H, i)) = f((H', i'))$ whenever $(H, i) \equiv (H', i')$. This
allows us to abuse notation and write $f(H, i)$ instead of $f([H, i])$.
We will follow a similar convention for functions on $\mH_{**}$.
\end{rem}





\begin{definition}
$\mT_*$ and $\mT_{**}$ denote the set of equivalence classes of connected vertex rooted hypertrees and connected edge--vertex rooted hypertrees, respectively. It can be checked that 
$\mT_*$ (respectively $\mT_{**}$) is a closed subset of 
$\mH_*$ (respectively $\mH_{**}$).
\end{definition}


\subsection{Operators for total load and average load: $\partial$ and $\nabla$}
\label{sec:partial-nabla}

\begin{definition}
  \label{def:partial}
  For a function $f: \mH_{**} \rightarrow \reals$, define $\partial f: \mH_* \rightarrow \reals$ as follows. For an equivalence class $[H, i] \in \mH_*$, pick an arbitrary $(H', i') \in [H, i]$ and define
  \begin{equation*}
    \partial f([H, i]) = \sum_{e' \in E(H'), e' \ni i'} f([H', e', i']) 
  \end{equation*}
\end{definition}

\begin{rem}
\label{rem:partial-well-defined}
Note that in the above definition, $[H', e', i']$ denotes the equivalence class of edge--vertex rooted hypergraph $(H', e', i')$. Also, since all the representatives of $[H, i]$ are isomorphic, it is easy to check that the above expression is not dependent on the specific choice of $(H', i')$. More precisely, if $(H_1, i_1)$ and $(H_2, i_2)$ are both members of $[H, i]$, then if $\phi:V(H_1) \rightarrow V(H_2)$ is the function establishing the isomorphism, $\phi$ gives a one to one mapping between the set $\{e \in E(H_1), e \ni i_1\}$ and $\{e \in E(H_2), e \ni i_2\}$ and also $[H_1, e, i_1] \equiv [H_2, \phi(e), i_2]$. Hence
\begin{equation*}
  \sum_{e \in E(H_1), e \ni i_1} f([H_1, e, i_1]) = \sum_{e \in E(H_2), e \ni i_2} f([H_2, e, i_2]),
\end{equation*}
which shows that $\partial f$ is well--defined.
\end{rem}

\begin{rem}
  \label{rem:partial-simplified-notation}
  By the above discussion, we may write $\partial f(H, i) = \sum_{e \ni i} f(H, e, i)$, where by $(H, i)$ we mean any arbitrary member of $[H, i]$.
\end{rem}

\begin{rem}
  \label{rem:partial-theta-H**-unify}
By abuse of notation, we can think of 
$\partial f$
as a function on $\mH_{**}$ by identifying 
$\partial f(H, e, i):= \partial f(H, i)$. 
This will be helpful when we have functions both on $\mH_*$ and $\mH_{**}$ and want to unify the domain.
\end{rem}


\begin{definition}
  \label{def:nabla}
  For a function $f: \mH_{**} \rightarrow \reals$, define the function $\nabla f: \mH_{**} \rightarrow \reals$ as follows. Given $[H, e, i] \in \mH_{**}$, take an arbitrary representative $(H', e', i')\in [H, e, i]$ and define 
  \begin{equation*}
    \nabla f([H, e, i]) := \frac{1}{|e'|} \sum_{j' \in e'} f([H', e', j']).
  \end{equation*}
\end{definition}


\begin{rem}
  \label{rem:cap-del-well-defined}
  As in our discussion in Remark~\ref{rem:partial-well-defined}, it can be easily checked that the above expression does not depend on the specific choice $(H', e', i')$. We can therefore abuse notation and write $\nabla f(H, e, i) = \frac{1}{|e|} \sum_{j \in e} f(H, e, j)$.
\end{rem}


\begin{definition}
  \label{def:deg-var}
  For a distribution $\mu \in \mP(\mH_*)$, define
  \begin{equation*}
    \deg(\mu) := \int_{\mH_*} \deg_H(i) d \mu.
  \end{equation*}
\end{definition}



\subsection[Directed Measure]{ From \texorpdfstring{$\mu \in \mP(\mH_*)$}{mu in P(H*)} to its directed version \texorpdfstring{$\vmu \in \mM(\mH_{**})$}{}}
\label{sec:vmu-}


\begin{definition}
  Given $\mu \in \mP(\mH_*)$ 
  with $\deg(\mu) < \infty$,  define the measure $\vmu \in \mM(\mH_{**})$ as the the one with the property that
  for any Borel function $f: \mH_{**} \rightarrow [0,\infty)$, we have 
  \begin{equation*}
    \int_{\mH_{**}} f d \vmu = \int_{\mH_*} \partial f d \mu.
  \end{equation*}
\end{definition}



Note that 
  $\deg(\mu) = \int_{\mH_{**}} 1 d \vmu = \vmu(\mH_{**})$
is the total mass of $\vmu$. Hence the assumption $\deg(\mu) < \infty$ guarantees $\vmu(\mH_{**}) < \infty$ and so $\vmu \in \mM(\mH_{**})$.

This following useful lemma is proved in Appendix~\ref{sec:prop-unim-meas}.


\begin{lem}
  \label{lem:A-as-Atilde-ae}
  $(i)$ Assume $A \subseteq \mH_*$ happens $\mu$--almost surely\ Then $\tilde{A} \subseteq \mH_{**}$ defined as 
  \begin{equation*}
    \tilde{A} := \{ [H, e, i] \in \mH_{**}: [H, i] \in A \},
  \end{equation*}
  happens $\vmu$--almost everywhere, i.e.\ $\vmu(\tilde{A}) = \vmu(\mH_{**})$. 

$(ii)$ Assume $B \subseteq \mH_{**}$ happens $\vmu$--almost everywhere. Then, $\tilde{B}:= \{[H,i] \in \mH_*: [H, e, i] \in B\, \, \forall \, e \ni i\}$ happens $\mu$--almost surely.
\end{lem}

The following fact relating the convergence of a sequence of functions on $\mH_*$ and that of their counterparts on $\mH_{**}$ will be useful later. This is proved in Appendix~\ref{sec:prop-unim-meas}.

\begin{lem}
  \label{lem:a.s.in-mu--a.e.in-vmu}
  Let $f_k: \mH_* \rightarrow \reals, k \geq 0,$ be 
  measurable functions such that we have $\lim_{k \rightarrow \infty} f_k = f_0$, $\mu$--almost surely. Then, if we define their $\mH_{**}$ counterparts $\tilde{f}_k : \mH_{**} \rightarrow \reals$ via 
  \begin{equation*}
    \tilde{f}_k(H, e, i) := f_k(H, i),
  \end{equation*}
  then we have $\tilde{f}_k \rightarrow \tilde{f}_0$, $\vmu$--almost everywhere.
\end{lem}


Another useful lemma is the following one, which relates the convergence of
a sequence of functions on $\mH_{**}$ to that of their $\partial$ on $\mH_*$.
This lemma is also proved in Appendix~\ref{sec:prop-unim-meas}.

\begin{lem}
  \label{lem:vmu-a.e.-del-mu-a.s.}
  Given $\mu \in \mP(\mH_*)$ and a sequence of functions $f_k :\mH_{**} \rightarrow \reals$, $k \ge 0$, assume that we have $f_k \rightarrow f_0$ $\vmu$--almost everywhere. Then we have $\partial f_k \rightarrow \partial f_0$ $\mu$--almost surely.
\end{lem}

\subsection{Local Weak Convergence}
\label{sec:local-weak-conv}

For a finite hypergraph $H$, define 
$u_H \in \mP(\mH_*)$
by choosing a vertex uniformly at random in $H$ as the root. More precisely, if $i$ is a vertex in $H$ and $H(i)$ denotes the connected component of $i$, define
\begin{equation*}
  u_H := 
  \frac{1}{|V(H)|} \sum_{i \in V(H)} \delta_{[H(i), i]}.
\end{equation*}
The reason why we take the connected component of $H$ is that $\mH_*$ is the space of equivalence classes of connected vertex rooted hypergraphs.

If, for a sequence of finite hypergraphs $\{ H_n \}$, 
$u_{H_n}$ 
converges weakly to some measure $\mu \in \mP(\mH_*)$, we say that $\mu$ is the ``local weak limit'' of the sequence $H_n$. 
The following lemma is useful in checking when local weak convergence occurs. See Appendix \ref{sec:proof-lemma-equivalent-condition-for-lwc} for a proof.


\begin{lem}
  \label{lem:eq-condition-local-weak-convergence}
  Given a sequence $\{ \mu_n \}_{n \geq 1}$ in $\mP(\mH_*)$, and $\mu \in \mP(\mH_*)$ such that $\text{supp}( \mu) \subseteq \mT_*$,  $\mu_n \Rightarrow \mu$ iff the following condition is satisfied: for all $d \geq 1$ and for all rooted hypertrees $(T, i)$ with depth at most $d$, if
  \begin{equation*}
    A_{(T, i)}:= \{ [H, j] \in \mH_*: (H, j)_d \equiv (T, i) \},
  \end{equation*}
then $\mu_n(A_{(T_, i)}) \rightarrow \mu(A_{(T,i)})$.
\end{lem}



The way to understand $u_{H_n} \Rightarrow \mu$ is that the local structure around a uniformly chosen vertex in $H_n$, where local means up to a fixed depth, looks more and more similar to that corresponding to $\mu$, hence the term ``local'' weak convergence.
In particular, Lemma \ref{lem:eq-condition-local-weak-convergence} says that if we have a sequence of finite hypergraphs $H_n$, 
then $u_{H_n} \Rightarrow \mu$,
where $\mu \in \mP(\mT_*)$, 
if and only if 
for each $d \geq 1$ and all rooted hypertrees $(T, i)$, if we choose a vertex $v$ in $H_n$ uniformly at random, the probability that  the local structure of $H_n$ rooted at $v$, i.e. $(H_n, v)_d$, is isomorphic to $(T, i)$
of depth at most $d$ converges to the probability that the rooted tree with law $\mu$ up to depth $d$ is isomorphic to $T$.  
See \cite{benjamini2001schramm}, \cite{aldous2004objective}, \cite{aldous2007processes}, \cite{Bordenave14randomgraphs} 
for  a review of the notion of local weak convergence in the graph regime.


\subsection{Balanced allocations with respect to a distribution on $\mH_*$}
\label{sec:balanc-alloc-with}

\begin{definition}
\label{def:Borel-allocation}
A function $\Theta: \mH_{**} \rightarrow [0,1]$ is called a Borel  allocation,
or just an allocation, if $\Theta$ is a Borel function and also 
\begin{equation*}
  \nabla \Theta(H, e, i) = \frac{1}{|e|}.
\end{equation*}
\end{definition}


From Definition \ref{def:nabla} and using our simplified notational
conventions, we can equivalently say that $\Theta$ is
an allocation precisely when it is a Borel function and $\sum_{i \in e} \Theta(H, e, i) = 1$ for the edge--vertex rooted hypergraph 
$(H,e,i)$. It may seem strange that the condition is required only at the
root edge. This will become more clear when we discuss unimodular 
measures in the next subsection, see especially Proposition~\ref{prop:everyting-shows-at-the-root}.

\begin{definition}
\label{def:borel-balanced-allocation}
  Assume $\mu \in \mP(\mH_*)$. A Borel allocation $\Theta: \mH_{**} \rightarrow [0,1]$ is called balanced with respect to $\mu$ if for $\vmu$ almost all $[H, e, i] \in \mH_{**}$ and any $(H', e', i')
  \in [H, e, i]$ we have:
  \begin{equation*}
    \forall j_1, j_2 \in e', \quad \partial \Theta([H', e', j_1]) > \partial \Theta([H', e', j_2]) \qquad \Rightarrow \qquad \Theta([H', e', j_1]) = 0.
  \end{equation*}
\end{definition}

\begin{rem}
  \label{rem:balanced-H*-equivalence-class}
  As in our discussion in Remarks~\ref{rem:partial-well-defined} and \ref{rem:cap-del-well-defined}, the above predicate does not depend on the specific choice of $(H', e', i')
  \in [H, e, i]$. Hence, by abuse of notation, we may write the above predicate simply as ``for $\vmu$ almost all $(H, e, i)$ and $j_1, j_2 \in e$,  $\partial \Theta(H, j_1) > \partial \Theta(H, j_2)$ implies $\Theta(H, e, j_1) = 0$''.
\end{rem}

Similar to our notion of $\epsilon$--balanced allocation for a specific hypergraph, we can define $\epsilon$--balanced allocations with respect to a measure $\mu \in \mP(\mH_*)$. As in the discussion in Remark~\ref{rem:balanced-H*-equivalence-class}, we use a simplified language in describing the definition.
\begin{definition}
  \label{def:epsilon-balanced-allocation-H*}
  Assume $\mu \in \mP(\mH_*)$. A Borel allocation $\Theta_\epsilon: \mH_{**} \rightarrow [0,1]$ is called $\epsilon$--balanced with respect to $\mu$ if for $\vmu$ almost all $[H, e, i] \in \mH_{**}$ we have 
  \begin{equation*}
    \Theta_\epsilon(H, e, i) = \frac{\exp \left ( - \partial \Theta_\epsilon(H, i)/ \epsilon \right )}{ \sum_{j \in e} \exp \left ( - \partial \Theta_\epsilon(H, j) / \epsilon \right )}.
  \end{equation*}
\end{definition}




\subsection{Unimodularity}
\label{sec:unimodularity}

\begin{definition}
  \label{def:unimodularity}
  A probability measure $\mu \in \mP(\mH_*)$ is called unimodular if for every Borel function $f: \mH_{**} \rightarrow [0,\infty)$ we have 
  \begin{equation*}
    \int f d \vmu = \int \nabla f d \vmu.
  \end{equation*}
We denote the set of unimodular measures on $\mH_*$ by $\mU$.
\end{definition}

See \cite{aldous2007processes} for a definition of unimodularity for graphs. It can be easily checked that our definition of unimodularity reduces to the definition in \cite{aldous2007processes} when we restrict to graphs, i.e. when we restrict $\mu$ to have support on hypergraphs with all edges having size two. 


It can be shown that for a finite hypergraph, $u_H$ defined in Section~\ref{sec:local-weak-conv} is unimodular. Moreover, if a sequence of finite hypergraphs has a local weak limit $\mu$, then $\mu$ is unimodular\footnote{Whether the converse is true is an open question, even in the graph regime}. See Appendix~\ref{sec:everything-shows-at} for a proof. Roughly speaking, unimodular measures are extensions of the kinds of measures on equivalence classes of vertex rooted hypergraphs that arise from choosing the root uniformly at random in finite hypergraphs.


The following property of unimodular measures is crucial in our analysis. It essentially says that ``everything shows at the root'' of a unimodular measure. See Appendix~\ref{sec:everything-shows-at} for its proof.

\begin{prop}
  \label{prop:everyting-shows-at-the-root}
  Assume $\tau : \mH_{**} \rightarrow \reals$ and $\mu \in \mU$ is a unimodular probability measure such that $\tau = 1$ $\vmu$--almost everywhere. Then there exists some $A \subset \mH_{**}$ such that $\vmu(A^c) = 0$ and
  \begin{equation*}
    \forall [H, e, i] \in A \quad \tau( [H, e', i']) = 1, \quad \forall e' \in E(H), i' \in e', \forall (H,e,i) \in [H,e,i]~.
  \end{equation*}
\end{prop}


Note that this statement is consistent with our intuition regarding unimodular measures: when some property holds at the root, since the root is chosen ``uniformly'' and so all vertices have the same ``weight'', that property should hold everywhere. 
See Lemma~2.3 in \cite{aldous2007processes} for a version of the above statement for graphs. 

\subsection{Unimodular Galton--Watson hypertrees}
\label{sec:unim-galt-wats}

In this section, we introduce an analogue of Galton Watson processes on graphs for hypertrees. In the graph regime, a Galton--Watson process is defined by generating the degree of the root at random from a given distribution and then, iteratively, the degree of each child is generated at random and so forth. In order to generalize the notion of a Galton--Watson process to hypertrees, since there might exist edges of different sizes, one needs to make sense of the ``degree'' of edges of each possible size at each node. 

To this end, we introduce the notion of {\em type} as a generalization of the notion of degree. The type of each node is a vector of integers specifying how many edges of each size a node is connected to. More precisely, since we want all the hypergraphs to be locally finite, we define the set of types, denoted by $\natszf$, as 
\begin{equation}
\label{eq:def-natsz}
  \natszf := \{ \gamma \in \natsz^{\{2, 3, \dots\}}: \gamma(k) = 0, k> k_0 \, \, \text{ for some } k_0 \geq 2 \},
\end{equation}
where $\natsz := \nats \cup \{ 0 \}$. For a type $\gamma \in \natszf$, $\gamma(k)$ determines the number of edges of size $k$ a node is connected to. For instance $(2, 1,0, 1)$ means a node is connected to 2 edges of size 2, 1 edge of size 3 and 1 edge of size 5 (we haven't shown the rest of the sequence, which is zero). 

For $\gamma \in \natszf$ define 
\begin{equation}
  \label{eq:type-norm-one}
  \norm{\gamma}_1 := \sum_{k \geq 2} \gamma(k),
\end{equation}
and 
\begin{equation}
  \label{eq:type-norm-infty}
  \norminf{\gamma} := \max\{k \geq 2: \gamma(k) > 0 \}.
\end{equation}
For $k \geq 2$, define $\typee_k \in \natszf$ to be the vector with value $1$ at coordinate $k$ and zero elsewhere. 


Assume $P \in \mP(\natszf)$ is a probability distribution over the set of 
types, such that $\ev{\Gamma(m)} < \infty$ for $m\geq 2$, where $\Gamma$ is a random variable with law $P$. For $m \geq 2$ such that $\ev{\Gamma(m)} > 0$, we define the size biased distributions $\hat{P}_m$  as
\begin{equation}
  \label{eq:size-biased-distirbution}
  \hat{P}_m(\gamma) = \frac{(\gamma(m) + 1)P(\gamma + \typee_m)}{\ev{\Gamma(m)}},
\end{equation}
where it is easy to check that the normalizing term makes $\hat{P}_m$ a probability distribution. In case $\ev{\Gamma(m)} = 0$, or equivalently $\Gamma(m) = 0$ with probability one, we define $\hat{P}_m$ to be an arbitrary distribution, e.g. $\hat{P}_m(\gamma) = 1$ when $\gamma$ is the type with all coordinates being zero.

Let $\emptyset$ denote the root of the Galton--Watson hypertree and let $\nvertex$ denote the set of vertices, so $\emptyset \in \nvertex$.
Let $\nedge$ denote the set of hyperedges of the Galton--Watson hypertree.
Each non-root element of $\nvertex$ will be of the form
$(s_1, e_1, i_1, \dots, s_k, e_k, i_k)$
where $s_j \geq 2$, $e_j \ge 1$, and $1 \leq i_j \leq s_j - 1$ for all $1 \leq j \leq k$. The semantics of $(s_1, e_1, i_1)$ is that it is the vertex 
numbered $i_1$ of the $e_1$-th copy of a hyperedge of size $s_1$
attached to the root, and so on. Thus, for example $(3,5,2,5,8,3)$
represents the vertex numbered $3$ of the eight hyperedge of size $5$
that attaches to the vertex labeled $2$ of the fifth hyperedge of size $3$
that attaches to the root. 
The elements of $\nedge$ are thus of the form
$(s_1, e_1, i_1, \dots, s_k, e_k)$,
where $s_j \geq 2$, $e_j \geq 1$, and $1 \leq i_j \leq s_j - 1$
for all $1 \leq j \leq k-1$.

 Given a sequence of types $\{ \gamma_a \}_{a \in \nvertex}$, we can construct a hypertree with vertex set and edge set $\nvertex$ and $\nedge$, respectively, where for $ a\in \nvertex$, $\gamma_a$ determines the type of the node $a$ in the subtree below node $a$. See Figure~\ref{fig:gamma-GWT} for an example. 


\begin{figure}
  \centering
    \begin{tikzpicture}[scale=1.2]
      \node (0) [Root,label={right:\color{magenta}$\emptyset$}] at (0,0) {};
      \node (212) [Node] at (-0.5,-1) {};
      \node (211) [Node] at (-1.5,-1) {};
      
      \node (311) [Node] at (0.5,-1) {};
      \node (312) [Node] at (1.5,-1) {};
      

      \node (312211) [Node] at (1.5,-1.7) {};
      

      \draw[thick, Cyan] (0,0) -- (-1.5,-1) {};
      \draw[thick, Cyan] (0,0) -- (-0.5,-1) {};
      
      \draw[thick, Cyan,rounded corners=10pt] ($(0,0)+(-0.3,0.4)$) -- ($(0.5,-1)+(-0.1,-0.2)$) -- ($(1.5,-1)+(0.4,-0.2)$) -- cycle;
      

      \draw[thick, Cyan] (1.5,-1) -- (1.5,-1.7);
      
      \node [below =1mm of 211,scale=0.7, magenta] {$(2,1,1)$};
      \node [below =1mm of 212,scale=0.7, magenta] {$(2,2,1)$};
      \node [below =1mm of 311,scale=0.7, magenta] {$(3,1,1)$};
      \node [right =2mm of 312,scale=0.7, magenta] {$(3,1,2)$};
      \node [below =1mm of 312211,scale=0.7, magenta] {$(3,1,2,2,1,1)$};
      
    \end{tikzpicture}
    \caption{\label{fig:gamma-GWT}The tree rooted at $\emptyset$ generated by the sequence $\{\gamma_a\}_{a \in \nvertex}$ where $\gamma_{\emptyset} = (2,1)$, $\gamma_{(3,1,2)}$, which determines the type of $(3,1,2)$ in the subtree below $(3,1,2)$, is equal to $(2)$ (which results in the single edge of size 2 below $(3,1,2)$), and $\gamma_a$ is the zero vector for all other nodes $a$.}
\end{figure}
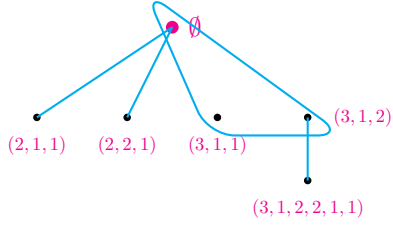


\begin{definition}
\label{def:ugwt(P)}
Let $P \in \mP(\natszf)$ such that $\ev{\Gamma(m)} < \infty$ for $m \geq 2$, where $\Gamma$ has the distribution $P$. Construct a random rooted tree $(T,\emptyset)$ by generating $(\Gamma_a, a \in \nvertex)$ independently such that $\Gamma_\emptyset$ has law $P$ and for any non--root node $a = (s_1, e_1, i_1, \dots, s_k, e_k, i_k)$, $\Gamma_a$ has law $\hat{P}_{s_k}$. Then $\ugwt(P)$ is the law of $[((T,\emptyset),\emptyset)]$ where $[((T,\emptyset),\emptyset)]$ denotes the equivalence class of $((T,\emptyset),\emptyset)$ in $\mH_*$,  with $(T,\emptyset)$ being the connected component of $\emptyset$ in $T$.
\end{definition}

One important observation is that $\ugwt(P)$, as in
Definition \ref{def:ugwt(P)}, is unimodular. See Appendix~\ref{sec:unimodularity-ugwtp} for the proof.
\begin{prop}
  \label{prop:ugwt-unimodular}
  Assume $P \in \mP(\natszf)$ is a distribution over types such that $\ev{\Gamma(k)} < \infty$ for $k \geq 2$, then $\ugwt(P)$ is unimodular. 
\end{prop}

Note that if $(T,\emptyset)$ is generated as in Definition \ref{def:ugwt(P)}
and $(s_1, e_1, i_1)$ is a child of the root present in $T$, 
the subtree rooted at $i_1$
 has a similar distribution to $(T,\emptyset)$ except that the type of its root has law $\hat{P}_{s_1}$. It is useful for our subsequent discussion to define a notation for this distribution.



\begin{definition}
  \label{def:gwt_k(P)}
Let $P \in \mP(\natszf)$ such that $\ev{\Gamma(l)} < \infty$ for $l \geq 2$ and fix some $k \geq 2$.  Construct a random rooted tree $(T,\emptyset)$ by generating $(\Gamma_a, a \in \nvertex)$ independently such that $\Gamma_\emptyset$ has law $\hat{P}_k$ and for any non--root node $a = (s_1, e_1, i_1, \dots, s_r, e_r, i_r)$, $\Gamma_a$ has law $\hat{P}_{s_r}$. Then $\gwt_k(P)$ denotes the law of $[{\color{red}(T,\emptyset)},\emptyset]$ where $[{\color{red}(T,\emptyset)},\emptyset]$ denotes the equivalence class of 
$({\color{red}(T,\emptyset)},\emptyset)$ in $\mH_*$.  
\end{definition}

\subsection[Networks]{Equivalence classes of marked hypergraphs: \texorpdfstring{$\mbH_*(\Xi)$}{H*} and \texorpdfstring{$\mbH_{**}(\Xi)$}{H**}}
\label{sec:mbh-mbh}


Recall that $\mH_*$ and $\mH_{**}$ are Polish spaces of isomorphism classes of vertex rooted hypergraphs and edge--vertex rooted hypergraphs respectively. We can extend the procedure by which these spaces were created to hypergraphs with marks on their edges. Hypergraphs with marks on their edges would be called ``hypernetworks",
following the terminology in Aldous and Lyons \cite{aldous2007processes}.
However, we prefer to call them {\em marked hypergraphs}.



\begin{definition}
  \label{def:makred-rooted-hypergraph}
  Assume $H$ is a locally finite simple hypergraph on a countable vertex set and $\Xi$ is a complete separable metric space. A $\Xi$--valued edge mark on $H$ is a function 
  \begin{equation*}
    \xi : \evpair(H) \rightarrow \Xi,
  \end{equation*}
  where we recall that $\evpair(H)$ denotes the set of 
  edge-vertex pairs of $H$.
 A hypergraph carrying such a mark is called a marked hypergraph, and
 is denoted $(H, \zeta)$.

  A vertex rooted marked hypergraph is a marked hypergraph $(H,{\color{red} \xi})$ together with a distinguished vertex $i \in V(H)$, and is denoted
  $((H,{\color{red} \xi}), i)$. An edge-vertex rooted marked hypergraph is a marked hypergraph $(H, {\color{red} \xi})$ together with a distinguished
  edge $e \in E(H)$, and a distinguished vertex $i \in e$. It is denoted
  $((H,{\color{red} \xi}), e, i)$.
\end{definition}

\begin{rem}
\label{rem:Hypergraph--notation}
To simplify the notation, we employ the notation $\bH$ to denote a marked hypergraph, where its mark function is denoted by $\xi_{\bH}$, and its underlying unmarked hypergraph is denoted by $H$. Moreover, we may use $V(\bH)$, $E(\bH)$ and $\evpair(\bH)$ instead of $V(H)$, $E(H)$ and $\evpair(H)$.
\end{rem}



\begin{definition}
  \label{def:rooted-marked-isomorphism}
  We call two vertex rooted  marked hypergraphs $(\bH_1, i_1)$ and $(\bH_2, i_2)$ isomorphic, and write $(\bH_1, i_1) \equiv (\bH_2, i_2)$  if there is a bijection $\phi: V(\bH_1) \rightarrow V(\bH_2)$, such that $\phi(i_1) = i_2$, and $e \in E(\bH_1)$ iff $\phi(e) \in E(\bH_2)$, where $\phi(e) := \{ \phi(i): i \in e \}$. Moreover, for $(e,i) \in \evpair(\bH_1)$, we require that $\xi_{\bH_1}(e, i) = \xi_{\bH_2}(\phi(e), \phi(i))$. 

  We say two edge-vertex rooted marked hypergraphs $(\bH_1, e_1, i_1)$ and $(\bH_2, e_2, i_2)$ are isomorphic, and write $(\bH_1, e_1, i_1) \equiv (\bH_2, e_2, i_2)$, if there is a bijection $\phi: V(\bH_1) \rightarrow V(\bH_2)$, such that 
$\phi(i_1) = i_2$, $\phi(e_1) = e_2$, and such that $e \in E(\bH_1)$ iff $\phi(e) \in E(\bH_2)$. Moreover, we must have $\xi_{\bH_1}(e, i) = \xi_{\bH_2}(\phi(e), \phi(i))$ for $(e, i) \in \evpair(\bH_1)$.
\end{definition}

\begin{definition}
  \label{def:bH*}
  Let $\mbH_*(\Xi)$ be the space of equivalence classes of  connected vertex rooted marked hypergraphs with marks taking values in $\Xi$, with the 
  equivalence class of $(\bH,i)$ being denoted $[\bH,i]$.
  Similarly, let $\mbH_{**}(\Xi)$ be the space of equivalence classes of edge-vertex rooted marked  hypergraphs with marks taking values in $\Xi$, with the 
  equivalence class of $(\bH,e, i)$ being denoted $[\bH,e,i]$.
\end{definition}


We endow $\mbH_*(\Xi)$ with the metric  $\bar{d}_*$ where the distance between $[\bH_1, i_1]$ and $[\bH_2, i_2]$ is defined in the following way: take  arbitrary representatives $(\bH'_1, i'_1) \in [\bH_1, i_1]$ and $(\bH'_2, i'_2) \in [\bH_2, i_2]$, then let $m^*$ be the supremum over all $m$ such that $(H'_1, i'_1) \equiv_m (H'_2, i'_2)$, and the $\Xi$--distance between the corresponding marks up to level $m$ is at most $1/m$, i.e. if $\phi$ is the level $m$ isomorphism, then
\begin{equation*}
  d_\Xi(\xi_{\bH_1}(\tilde{e}, \tilde{i}), \xi_{\bH_2}(\phi(\tilde{e}), \phi(\tilde{i}))) \leq \frac{1}{m}~, \qquad \forall \tilde{e} \in E_{H'_1}\left (V^{H'_1}_{i'_1, m} \right ),
\end{equation*}
where $d_\Xi$ denotes the metric on $\Xi$.
If there is no $m$ satisfying the above conditions, 
we set $m^*$ to be 0.
Then, $\bar{d}_*([\bH_1, i_1], [\bH_2, i_2])$ is defined to be $1/(1+m^*)$. 
Since all the members in $[\bH, i]$ are isomorphic as vertex rooted marked hypergraphs, $\bar{d}_*$ can be easily checked to be well defined. One can also check that it is a metric; in particular it satisfies the triangle inequality.


Similarly, we endow $\mbH_{**}(\Xi)$ with the metric $\bar{d}_{**}$ where  the distance between $[\bH_1, e_1, i_1]$ and $[\bH_2, e_2, i_2]$ is defined in the following way: take  arbitrary representatives $(\bH'_1, e'_1, i'_1) \in [\bH_1, e_1, i_1]$ and $(\bH'_2, e'_2, i'_2) \in [\bH_2, e_2, i_2]$, then let $m^{**}$ be the supremum over all $m$ such that $(H'_1, e'_1, i'_1) \equiv_m (H'_2, i'_2, e'_2)$ and the $\Xi$--distance between the corresponding marks up to level $m$ is at most $1/m$. 
If there is no $m$ satisfying these conditions, we set $m^*$ to be 0.
Finally, define $\bar{d}_{**}([\bH_1, e_1, i_1], [\bH_2, e_2, i_2])$ to be $1/(1+m^{**})$.


In Appendix~\ref{sec:mh_-mh_-are} we prove that $\mbH_*(\Xi)$ and $\mbH_{**}(\Xi)$ with their respective metrics are Polish spaces, see Proposition~\ref{prop:H*-H**-Polish}.

Similar to what we did for $\mH_*$ and $\mH_{**}$, we can define $\partial$ and $\nabla$ operators as follows, where
we  use a simplified notation, whose validity can be justified as in Remarks~\ref{rem:partial-simplified-notation} and \ref{rem:cap-del-well-defined}:
\begin{equation*}
  \partial f(\bH, i) := \sum_{e \in E(\bH), e \ni i} f(\bH, e, i),
\end{equation*}
and 
\begin{equation*}
  \nabla f(\bH, e, i) := \frac{1}{|e|} \sum_{j \in e} f(\bH, e, j).
\end{equation*}

\begin{rem}
\label{rem:notation-abuse}
The preceding notation, strictly speaking, applies only to real valued 
functions on $\mH_{**}$. However, if we
consider the function $f: \mbH_{**}(\Xi) \rightarrow \Xi$ defined as $f(\bH, e, i) = \xi_{\bH}(e, i)$, then, when $\Xi$ has an additive structure, we may define $\partial f : \mbH_*(\Xi) \rightarrow \Xi$ as 
\begin{equation*}
  \partial f (\bH, i) := \sum_{e\in E(H), e \ni i} f(\bH, e, i) = \sum_{e\in E(\bH), e \ni i} \xi_{\bH}(e, i).
\end{equation*}
By abuse of notation, we may use the notation $\partial \xi_{\bH}(i)$ instead of $\partial f(\bH, i)$ with $f$ defined above, and similarly for
$\nabla \xi_{\bH}(i)$. We will use such 
notation in this document because the marks we are interested in will
be real valued.
\end{rem}

For a probability measure $\mu \in \mP(\mbH_*(\Xi))$, we define $\vmu \in \mM(\mbH_{**}(\Xi))$ in a manner similar to what was done in Section~\ref{sec:vmu-}. Namely, $\vmu \in \mM(\mbH_{**}(\Xi))$ is defined by requiring that for every nonnegative Borel function $f: \mbH_{**}(\Xi) \rightarrow [0,\infty)$ we have
\begin{equation*}
  \int f d \vmu = \int \partial f d \mu.
\end{equation*}

A probability measure $\mu \in \mP(\mbH_*(\Xi))$ is called unimodular if for every nonnegative Borel function $f: \mbH_{**}(\Xi) \rightarrow [0,\infty)$ we have 
\begin{equation*}
  \int f d \vmu = \int \nabla f d \vmu.
\end{equation*}

By removing marks, we get a natural projection from $\Pi: \mbH_*(\Xi) \rightarrow \mH_*$ defined as 
\begin{equation}
  \label{eq:projection-mbH*-mH*}
  \Pi([\bH, i])  = [H, i]~.
\end{equation}
This can easily be checked to be continuous.




As was done in Section~\ref{sec:unimodularity}, choosing a vertex uniformly at random from a finite marked hypergraph results in a unimodular measure. More precisely, if $H$ is a finite hypergraph together with a mark $\xi$ taking values in $\Xi$, 
\begin{equation*}
  u_{\bH} := \frac{1}{|V(\bH)|} \sum_{i \in V(\bH)} \delta_{[\bH(i), i]},
\end{equation*}
is unimodular.
Here, $\bH( i)$ denotes the connected component of $i$ in $\bH$.
 Moreover, the local weak limit of finite marked hypergraphs, i.e. the weak limit of the measures $u_{H_n}$, is unimodular, if it exists. 
See Appendix~\ref{sec:everything-shows-at} for a proof.






\section{Main Results}
\label{sec:main-results}


We now summarize the main results of the paper.
We first prove some properties of balanced Borel allocations, as defined in Definition~\ref{def:borel-balanced-allocation}.

\begin{thm}
  \label{thm:balanced--properties}
  Let $\mu \in \mP(\mH_*)$ be a unimodular probability measure such that $\deg(\mu) < \infty$. The following are true.
  \begin{enumerate}
  \item \label{thm:prop-exitence}
    There exists a Borel allocation $\Theta: \mH_{**} \rightarrow [0,1]$ which is balanced with respect to $\mu$.
  \item \label{thm:prop-var-char}
    Let $\Theta$ be a balanced  Borel allocation with respect to $\mu$. Then  we have the following variational characterization of the mean excess load under $\Theta$ above the load level $t$. For any $t \in \reals$:
    \begin{equation*}
      \int (\partial \Theta - t)^+ d \mu = \max_{\stackrel{f: \mH_* \rightarrow [0,1]}{\text{Borel}}} \int \tilde{f}_\text{min} d \vmu - t \int f d\mu,
    \end{equation*}
    where $\tilde{f}_\text{min}$ is defined as 
  \begin{equation*}
    \tilde{f}_\text{min} (H, e, i) := \frac{1}{|e|} \min_{j \in e} f(H, j).
  \end{equation*}
\item \label{thm:prop-optimality}
  The following are equivalent for a Borel allocation $\Theta: \mH_{**} \rightarrow [0,1]$:
  \begin{enumerate}
    \item $\Theta$ is balanced with respect to $\mu$.
    \item $\Theta$ minimizes $\int f \circ \partial \Theta d \mu$ among all Borel allocations, for some strictly convex function $f: [0,\infty) \rightarrow [0,\infty)$.
    \item $\Theta$ minimizes $\int f \circ \partial \Theta d \mu$ among all Borel allocations for every convex function $f: [0,\infty) \rightarrow [0,\infty)$.
  \end{enumerate}
  \item \label{thm:prop-uniqueness}
    Assume $\Theta_0$ is a balanced allocation with respect to $\mu$  and $\Theta$ is any other allocation on $\mH_{**}$. Then $\Theta$ is balanced if and only if $\partial \Theta_0 = \partial \Theta$, $\mu$--almost surely.
  \item   Let $\{H_n\}_{n \geq 1}$ be a sequence of finite hypergraphs with local weak 
limit $\mu$. Let $\mL_{H_n}$ denote the distribution of the total load at a vertex in $H_n$ chosen uniformly at random. Namely,
 $\mL_{H_n} = \frac{1}{|V(H_n)|} \sum_{i \in V(H_n)} \delta_{\partial \theta_n(i)}$, where $(\partial \theta_n(i), i \in V(H_n))$ denotes the load vector corresponding to any balanced allocation $\theta_n$, 
which we recall exists and is unique, due to Proposition~\ref{prop:total-load-unique-finite-hypergraphs}.
Let $\mL$ denote the law of the total load at the root of the balanced allocation on $\mu$, i.e. the pushforward of $\mu$ under the mapping $\partial \Theta$, which we have just shown is well defined and unique, due to parts 1 and 4 of this theorem. Then $\mL_{H_{n}}$ converges weakly to $\mL$. 
  \end{enumerate}
\end{thm}


The proof of the above theorem is given in Section~\ref{sec:prop-balanc-alloc-1}.
Before that, we first investigate,
in Section~\ref{sec:epsilon-balancing},
the properties of $\epsilon$--balanced allocations for a specific hypergraph, as introduced in Definition~\ref{def:epsilon-balanced-specific-hypergraph}.
We then investigate,
in Section~\ref{sec:epsil-balanc-alloc-1},
the properties of $\epsilon$--balanced allocations
on $\mH_{**}$ for a given $\mu \in \mP(\mH_*)$, as defined in Definition~\ref{def:epsilon-balanced-allocation-H*}.
Then, by sending $\epsilon$ to zero, we prove part 1 of the above theorem in Section~\ref{sec:existence}. The proofs of other parts of the theorem are given in Sections~\ref{sec:vari-char} through \ref{sec:cont-with-resp} respectively. 

Note that, as a result of part 4 of the theorem, for a given $t \in \reals$, the value of the integral $\int (\partial \Theta - t)^+ d \mu$ for a balanced allocation $\Theta$ does not depend on the particular choice of $\Theta$.
It only depends on $\mu$ and $t$, so it can be written as
\begin{equation}
  \label{eq:mean-excess-definition}
  \Phi_\mu(t) := \int (\partial \Theta- t)^+ d \mu.
\end{equation}
The function $\Phi_\mu$ is called the mean--excess function. 
Knowledge of $\Phi_\mu$ is equivalent to determining the 
distribution $\mL$ of the load at the root associated to a balanced allocation
$\Theta$. 
We now describe $\Phi_\mu$ for the class of unimodular Galton--Watson process defined in Section~\ref{sec:unim-galt-wats}.


Recall the notation $\natszf$ defined in \eqref{eq:def-natsz}.
Assume $P \in \mP(\natszf)$ and $t \in \reals$ are fixed. For a sequence of Borel probability measures $(Q_l, l \geq 2)$ on real numbers, let $F_{P, t}^{(k)}(\{ Q_l \}_{l \geq 2})$ be the distribution of the random variable 
\begin{equation}
\label{eq:generalized-fixed-point}
  t - \sum_{k'\ge 2} \sum_{i=1}^{\Gamma(k')} \left [ 1 - X_{k', i, 1}^+ - \dots - X_{k', i, k'-1}^+ \right ]_0^1,
\end{equation}
where $\Gamma$ has law $\hat{P}_k$, and $X_{k', i, j}$ are random variables which are mutually independent and independent of $\Gamma$, with $X_{k', i, j}$ having law $Q_{k'}$. Note that $\hat{P}_k$ is the size biased version  of $P$ defined in \eqref{eq:size-biased-distirbution}. 
Also note that the first sum on the right hand side of 
\eqref{eq:generalized-fixed-point} is a finite sum, because
$\Gamma$ has finite support, pointwise. 

Let $\mQ$ be the set of sequences $\{ Q_l \}_{l \geq 2}$ such that, for all $k \geq 2$, we have:
\begin{equation}
  \label{eq:general-fixed-points}
  Q_k = F_{P,t}^{(k)} (\{ Q_{l} \}_{l \geq 2}).
\end{equation}

Now, we are ready to provide a characterization of the mean excess function. We give the proof of the following result in Section~\ref{sec:char-mean-excess}.

\begin{thm}
\label{thm:fixed-point-generalized-geeneralization-v1}
  Let $P$ be a distribution on $\natszf$ such that 
 $\ev{\norm{\Gamma}_1}<\infty$
where $\Gamma$ has law $P$. Then, with $\mu := \ugwt(P)$, for any $t \in \reals$, we have 
  \begin{equation}
\label{eq:mean-excecss-characterization-general-main-tatement}
    \Phi_\mu(t) = \max_{\{ Q_k \}_{k \geq 2} \in \mQ} \left ( \sum_{k=2}^\infty \frac{\ev{\Gamma(k)}}{k} \pr{ \sum_{i=1}^k X_{k,i}^+ < 1}\right ) - t \pr{\sum_{k=2}^{h(\Gamma)} \sum_{i=1}^{\Gamma(k)} Y_{k,i} > t},
  \end{equation}
  where, in the first expression, $\Gamma$ is a random variable on $\natszf$ with law $P$ and $\{X_{k,i}\}_{k,i}$ are i.i.d. such that $X_{k,i}$ has law $Q_k$. Also, in the second expression, $\Gamma$ has law $P$ and $\{ Y_{k,i} \}_{k,i}$ are independent from each other and from $\Gamma$, with $Y_{k,i}$ having the  law of the random variable $ [ 1 - (Z_{1}^+ + \dots + Z_{k-1}^+)]_0^1$, where $Z_j$ are i.i.d. with law $Q_k$.
\end{thm}

For a finite hypergraph $H$, we define $\varrho(H)$ to be the maximum load corresponding to a balanced allocation on $H$, i.e. if $\theta$ is a balanced allocation on $H$, 
\begin{equation}
  \label{eq:max-load-def}
  \varrho(H) := \max_{v \in V(H)} \partial \theta(v),
\end{equation}
which is well defined due to Proposition~\ref{prop:total-load-unique-finite-hypergraphs}.
From \cite[Corollary 7]{hajek1990performance} we know that there is a duality between this parameter and the subgraph of maximum edge density, i.e. 
\begin{equation}
  \label{eq:max-load-duality}
  \varrho(H) = \max_{S \subseteq V(H), S \neq \emptyset} \frac{|E_H(S)|}{|S|},
\end{equation}
where $E_H(S)$ denotes the set of edges of $H$ with all endpoints in $S$.

For a unimodular probability distribution $\mu$ on $\mH_*$ with finite $\deg (\mu)$, 
we define 
\begin{equation}
  \label{eq:max-load--mu}
  \varrho(\mu) := \sup \{ t \in \reals: \Phi_\mu(t) > 0 \},
\end{equation}
where $\Phi_\mu(.)$ is the mean excess function defined above. In other words, if $\Theta$ is the balanced allocation corresponding to $\mu$ introduced in Theorem~\ref{thm:balanced--properties} and $\mL_\mu$ is the law of $\partial \Theta$ under $\mu$, then
\begin{equation*}
  \varrho(\mu) = \sup \{ t \in \reals: \mL_\mu([t, \infty)) > 0 \}.
\end{equation*}
One question is whether local weak convergence implies convergence of maximum load, i.e., if $H_n$ is a sequence of graphs with local weak limit $\mu$, does $\varrho(H_n)$ converge to $\varrho(\mu)$? Similar to the graph case, this is not true in general, since we can always add an arbitrary but bounded clique to boost $\varrho(H_n)$ without changing the local weak limit. 
We prove convergence, under some conditions, for the special case where the limit $\mu$ is the UGWT model defined in Section~\ref{sec:unim-galt-wats} and, for each $n$, $H_n$ is a random hypergraph obtained from a generalized hypergraph configuration model defined in Section \ref{sec:conf-model-hypergr}.

\edit
\begin{thm}
  \label{thm:max-load}
  Let $P$ be a probability distribution on $\natszf$ such that, if $\Gamma$ is a random variable with law $P$, $P(\Gamma(k) > 0) > 0$ for finitely many $k$ and $\ev{\Gamma(k)} < \infty$ for all $k  \geq 2$. Moreover, let $\mu := \ugwt(P)$. Then, if $\{H_n\}_{n=1}^\infty$ is a sequence of random hypergraphs obtained from a configuration model, under some conditions stated in Proposition~\ref{prop:max-load-detail}, $\varrho(H_n)$ converges in probability to $\varrho(\mu)$. 
\end{thm}

This theorem is proved in Section~\ref{sec:conv-maxim-load}.


\section{\texorpdfstring{$\epsilon$}{e}--balancing with baseloads}
\label{sec:epsilon-balancing}

In this section, we analyze the properties of $\epsilon$--balanced allocations with respect to a baseload, which were introduced in Definition~\ref{def:epsilon-balanced-with-baseload}. 
Note that throughout this section, we are dealing with a given hypergraph, not a distribution on $\mH_*$. By setting the baseload function $b$ to zero, our results here reduce to those for $\epsilon$--balanced allocations as introduced in Definition~\ref{def:epsilon-balanced-specific-hypergraph}.

\subsection{Existence}
\label{sec:epsil-balanc-with}

The existence of $\epsilon$--balanced allocations with respect to a baseload $b$ on hypergraphs 
is a consequence of the Schauder--Tychonoff fixed point theorem (see, for instance, \cite{agarwal2009fixed}).
Here we give the details. 
Fix a 
hypergraph $H = \langle V, E \rangle$ 
and define the topological vector space $W$ to be:
\begin{equation*}
  W := \{ \theta: \Psi(H) \rightarrow \reals \} = \reals^{\Psi(H)},
\end{equation*}
with the product topology of $\reals$. 
Here, we recall that $\Psi(H)$ is the set of all edge--vertex pairs of the hypergraph
\eqref{eq:evpair-define}, and that this is a countable set. Define the following convex subset of functions with values in $[0,1]$:
\begin{equation*}
  A := \{ \theta : \Psi(H) \rightarrow [0,1] \}.
\end{equation*}
Since we have employed the product topology, 
Tychonoff's theorem tells us that $A$ is a compact set (see, for instance, \cite{munkres2000topology}).
Define the mapping $T: A\rightarrow A$ via:
\begin{equation*}
  (T\theta) (e ,i) := \frac{\exp \left ( - \frac{\partial_b \theta(i)}{\epsilon}\right ) }{\sum_{j \in e}  \exp \left ( - \frac{\partial_b \theta(j)}{\epsilon}\right ) }.
\end{equation*}
We want to show that $T$ has a fixed point. In order to do so, we need to show that $T$ is continuous. Since we have employed the product topology, we need to show that, for all $(e, i) \in \Psi(H)$, the projected version $T_{e,i}$ defined as:
\begin{equation*}
  T_{e,i} (\theta) := (T\theta)(e,i),
\end{equation*}
which is a mapping from $A$ to $[0,1]$, is continuous. In order to show this, note that $T_{e, i}$ is the concatenation of a projection $\Pi_e: A \rightarrow \reals^{|U_e|}$, where 
\begin{equation*}
  U_e := \{ (e',j) : e' \in E, e' \cap e \neq \emptyset, j \in e \cap e'\},
\end{equation*}
and an addition function from $\reals^{|U_e|}$ to $\reals^{|e|}$, which gives us the vector $[\partial \theta(j) ]_{j \in e}$, and then a function $f: \reals^{|e|} \rightarrow \reals$ defined as:
\begin{equation*}
  f([x_j]_{j \in e}):= \frac{e^{-(x_i+b(i)) / \epsilon}}{\sum e^{-(x_j+b(j))/\epsilon}}.
\end{equation*}
Since all these three functions are continuous (note that $U_e$ is a finite set since we have assumed that the graph is locally finite and all edges have finite size), $T$ is also continuous. Therefore, since $W$ is Hausdorff and locally convex, $A$ is compact, and $T$ is continuous, the Schauder--Tychonoff fixed point theorem implies that $T$ has a fixed point  (see, for instance, \cite[Theorem 8.2]{agarwal2009fixed}).
Note that $\theta' := T(\theta)$  satisfies $\sum_{i \in e} \theta'(e ,i) = 1$ for any $\theta \in W$. Therefore this fixed point is an allocation 
in the sense of Definition \ref{def:allocation},
and is also $\epsilon$--balanced.

\subsection{Monotony and Uniqueness}
\label{sec:monotony-uniqueness}


Intuitively, we expect that when we add more edges to a hypergraph and increase baseloads, the total load for an $\epsilon$--balanced allocation would increase. 
We also expect that the effect of an increase in baseload at any vertex
tends to dissipate as one moves away from the vertex, when comparing 
the respective balanced allocations. Lemmas~\ref{lem:contraction-depth-1-baseload} and \ref{lem:contraction-base-d-baseload} below 
quantify these phenomena. They are formulated in the language of 
vertex rooted hypergraph embedding from Definition~\ref{def:hypergraph-local-embedding}.

\begin{lem}[depth $1$ local contraction]
  \label{lem:contraction-depth-1-baseload}
  Assume the vertex rooted hypergraph $(H, i)$ can be embedded up to depth $1$ into the vertex rooted hypergraph $(H', i')$, i.e. $(H, i) \hookrightarrow_1 (H', i')$, with embedding $\phi:V^H_{i, 1} \hookrightarrow V^{H'}_{i', 1}$. Let $\theta_\epsilon$ and $\theta'_\epsilon$ be $\epsilon$--balanced allocations on $H$ and $H'$ respectively, with respective baseload functions $b$ and $b'$, with $b(i) \leq b'(i')$. If 
  \begin{equation*}
    M := \max_{j: d_H(i, j)=1} \partial_b \theta_\epsilon(j) - \partial_{b'} \theta'_\epsilon(\phi(j)),
  \end{equation*}
then we have 
\begin{equation*}
  \partial_{b'} \theta'_\epsilon(i')  \geq \partial_b \theta_\epsilon(i)  - \frac{|D^H_{i, 1}|}{|D^H_{i, 1} | + 4 \epsilon} M^+,
\end{equation*}
where $D^H_{i,1}$ is the set of nodes at distance one from node $i$ as was defined in Definition~\ref{def:E-H_VHid_DHid}.
\end{lem}

Note that, in this lemma, $\theta_\epsilon$ and $\theta'_\epsilon$ are two arbitrary $\epsilon$--balanced allocations on $H$ and $H'$ respectively,
with respective baseload functions $b$ and $b'$.
We know from Section~\ref{sec:epsil-balanc-with} that such allocations
exist, but they might a priori not be unique. We will later prove uniqueness for the special case of bounded hypergraphs, which were introduced in Definition
\ref{def:bounded-hypergraph}.

Before proving this result, we need the following tool, whose proof is given after the proof of Lemma~\ref{lem:contraction-depth-1-baseload}.


\begin{lem}
  \label{lem:fx-fx'}
  Assume that for $\epsilon>0$, the function $f_\epsilon: \reals^k \rightarrow \reals$ is defined in the following way:
  \begin{equation*}
    f_\epsilon(x_1, \dots, x_k) = \frac{1}{1 + \sum_{i = 1}^k e^{-\frac{x_i}{\epsilon}}}.
  \end{equation*}
  Then, for arbitrary real valued sequences $(x_1, \dots, x_k)$ and $(x'_1, \dots, x'_k)$,
  we have,
  \begin{equation*}
    f_\epsilon(x_1, \dots, x_k) - f_\epsilon(x'_1, \dots, x'_k) \leq \frac{1}{4\epsilon} \sum_{i = 1}^k [x_i - x'_i]^+.
  \end{equation*}
\end{lem}

\begin{proof}[Proof of Lemma~\ref{lem:contraction-depth-1-baseload}]
  Since $\phi(e) \in E(H')$ for $e \ni i$, and $\theta'_\epsilon$ is a nonnegative function, we have 
\begin{equation}
  \label{eq:forget-extra-edges-depth-1}
  \sum_{e' \ni i'} \theta'_\epsilon (e', i') \geq \sum_{e \ni i} \theta'_\epsilon (\phi(e), i').
\end{equation}
On the other hand, we have 
\begin{equation}
\label{eq:monotony}
  \begin{split}
    \partial_b \theta_\epsilon(i)  - \partial_{b'} \theta'_\epsilon(i') &
    \stackrel{(a)}{\leq} 
    \sum_{e \ni i} \theta_\epsilon(e, i)
    - \sum_{e' \ni i'} \theta'_\epsilon (e', i') \\
    &\stackrel{(b)}{\leq} \sum_{e \ni i} \theta_\epsilon(e, i) - \theta'_\epsilon(\phi(e), i') \\
    &\stackrel{(c)}{=} \sum_{e \ni i} \Bigg ( \frac{1}{1+ \sum_{\stackrel{j\in e}{j \neq i}} \exp \left ( - \frac{\partial_b \theta_\epsilon(j) - \partial_b \theta_\epsilon(i) }{\epsilon} \right )} \\
    & \qquad - \frac{1}{1 + \sum_{\stackrel{j \in e}{j \neq i}} \exp \left ( - \frac{\partial_{b'} \theta'_\epsilon(\phi(j))  - \partial_{b'} \theta'_\epsilon(i')}{\epsilon}\right ) } \Bigg ) \\
    &\stackrel{(d)}{\leq} \frac{1}{4\epsilon} \sum_{e \ni i} \sum_{\stackrel{j \in e}{j \neq i}} \left [ \left (  \partial_b \theta_\epsilon(j)  - \partial_b \theta_\epsilon(i) \right ) - \left ( \partial_{b'} \theta'_\epsilon(\phi(j))  - \partial_{b'} \theta'_\epsilon(i')  \right ) \right ]^+,
  \end{split}
\end{equation}
where $(a)$ results from $b(i) \leq b'(i')$, $(b)$ uses \eqref{eq:forget-extra-edges-depth-1}, $(c)$ is a substitution from Definition~\ref{def:balanced-allocation-with-baseload}, and $(d)$ uses Lemma~\ref{lem:fx-fx'}.

Now, let
\begin{equation*}
  I := \{(e, j) : e \ni i, j \in e, j \neq i, \partial_b \theta_\epsilon(j) - \partial_{b'} \theta'_\epsilon(\phi(j)) \geq \partial_b \theta_\epsilon(i) - \partial_{b'}\theta'_\epsilon(i') \}.
\end{equation*}
Then the inequality in \eqref{eq:monotony} together with the definition of $M$ implies 
\begin{equation*}
  \partial_b \theta_\epsilon(i) - \partial_{b'} \theta'_\epsilon(i') \leq \frac{1}{4 \epsilon} |I| ( M - (\partial_b \theta_\epsilon(i) - \partial_{b'} \theta'_\epsilon(i') ) ).
\end{equation*}
Rearranging the terms, we get
\begin{equation*}
  \partial_b \theta_\epsilon(i) - \partial_{b'} \theta'_\epsilon(i') \leq \frac{|I|}{|I| + 4\epsilon} M \leq \frac{|D^H_{i,1}|}{|D^H_{i,1}| + 4 \epsilon} M^+,
\end{equation*}
which is exactly what we wanted to prove.
\end{proof}

\begin{proof}[Proof of Lemma~\ref{lem:fx-fx'}]
  First we prove the statement for $k=1$. In this case, the function $f_\epsilon(x) = \frac{1}{1 + e^{-\frac{x}{\epsilon}}}$ is $\frac{1}{4\epsilon}$--Lipschitz and increasing in $x$, hence the statement holds for $k=1$. 

Now assume $k > 1$ is arbitrary. We have
\begin{equation}
  \label{eq:fx-fx'-expand}
  f_\epsilon(x_1, \dots, x_k) - f_\epsilon(x'_1, \dots, x'_k) = \sum_{i=1}^k \frac{e^{-\frac{x'_i}{\epsilon}} - e^{-\frac{x_i}{\epsilon}}}{\left ( 1 + \sum_{r =1}^k e^{-\frac{x_r}{\epsilon}} \right )\left ( 1 + \sum_{s =1}^k e^{-\frac{x'_s}{\epsilon}} \right ) }.
\end{equation}
Now, for each $1 \leq i \leq k$, if $x'_i \geq x_i$ then
\begin{equation*}
  \frac{e^{-\frac{x'_i}{\epsilon}} - e^{-\frac{x_i}{\epsilon}}}{\left ( 1 + \sum_{r =1}^k e^{-\frac{x_r}{\epsilon}} \right )\left ( 1 + \sum_{s =1}^k e^{-\frac{x'_s}{\epsilon}} \right ) } \leq 0 = \frac{1}{4\epsilon} [x_i - x'_i]^+.
\end{equation*}
On the other hand, if $x'_i < x_i$, then we have
\begin{equation*}
  \begin{split}
  \frac{e^{-\frac{x'_i}{\epsilon}} - e^{-\frac{x_i}{\epsilon}}}{\left ( 1 + \sum_{r =1}^k e^{-\frac{x_r}{\epsilon}} \right )\left ( 1 + \sum_{s =1}^k e^{-\frac{x'_s}{\epsilon}} \right ) } &\leq \frac{e^{-\frac{x'_i}{\epsilon}} - e^{-\frac{x_i}{\epsilon}}}{\left ( 1 + e^{-\frac{x_i}{\epsilon}} \right ) \left ( 1 + e^{-\frac{x'_i}{\epsilon}} \right )} \\
  &= \frac{1}{1 + e^{-\frac{x_i}{\epsilon}}} - \frac{1}{1 + e^{-\frac{x'_i}{\epsilon}}} \\
  &\leq \frac{1}{4\epsilon}[x_i - x'_i]^+,
  \end{split}
\end{equation*}
where the last step uses the statement for $k=1$. Therefore, in either case, we have proved that for all $1 \leq i \leq k$ we have:
\begin{equation*}
  \frac{e^{-\frac{x'_i}{\epsilon}} - e^{-\frac{x_i}{\epsilon}}}{\left ( 1 + \sum_{r =1}^k e^{-\frac{x_r}{\epsilon}} \right )\left ( 1 + \sum_{s =1}^k e^{-\frac{x'_s}{\epsilon}} \right ) } \leq \frac{1}{4\epsilon}[x_i - x'_i]^+.
\end{equation*}
Substituting this into (\ref{eq:fx-fx'-expand}) we get the desired result.
\end{proof}

Now we generalize Lemma~\ref{lem:contraction-depth-1-baseload} to depth $d$ local embeddings. 

\begin{lem}[depth $d$ local contraction]
  \label{lem:contraction-base-d-baseload}
  Assume $(H, i) \hookrightarrow_d (H', i')$ with embedding $\phi:V^H_{i,d} \hookrightarrow V^{H'}_{i',d}$. Also let $\theta_\epsilon$ and $\theta'_\epsilon$ be $\epsilon$--balanced allocations on $H$ and $H'$ respectively,
  with respective baseload functions $b$ and $b'$, where
  $b(j) \leq b'(\phi(j))$ for all $j$ such that $d_H(i, j) \leq d-1$. Then we have:
  \begin{equation*}
    \partial_{b'} \theta'_\epsilon(i') \geq \partial_b \theta_\epsilon(i) - \left ( \frac{L \Delta}{4 \epsilon + L \Delta} \right )^d M_d^+,
  \end{equation*}
where
\begin{equation*}
  M_d := \max_{j \in D^H_{i,d}} \partial_b \theta_\epsilon(j) - \partial_{b'} \theta'_\epsilon(\phi(j)),
\end{equation*}
and
\begin{equation*}
  \begin{split}
    L &:= \max_{e \in E_H(V^H_{i, d})} |e|~, \\
    \Delta &:= \max_{j \in V^H_{i, d-1}} \deg_H(j)~.
  \end{split}
\end{equation*}
\end{lem}

\begin{proof}
  Take some $j \in V(H)$ with $d_H(i, j) = k$. If $j' \in V(H)$ is such that $d_H(j, j') = 1$, the triangle inequality implies that
  \begin{equation*}
    |d_H(i, j') - d_H(i, j)| \leq d_H(j, j') = 1.
  \end{equation*}
Thus
  \begin{equation}
    \label{eq:D-1-k-1-k-k+1}
    D^H_{j, 1} \subset D^H_{i, k-1} \cup D^H_{i,k} \cup D^H_{i,k+1}.
  \end{equation}
Hence, if $d_H(i, j) \leq d-1$, using the same embedding map $\phi$ we have $(H,j) \hookrightarrow_1 (H', \phi(j))$. Hence, if we define
  \begin{equation*}
    M_k := \max_{j: d_H(i,j) =k} \partial_b \theta_\epsilon(j) - \partial_{b'} \theta'_\epsilon(\phi(j)) \qquad 0 \leq k \leq d,
  \end{equation*}
  and
  \begin{equation*}
    \alpha := \frac{L \Delta}{4 \epsilon + L \Delta},
  \end{equation*}
  then using Lemma~\ref{lem:contraction-depth-1-baseload} and \eqref{eq:D-1-k-1-k-k+1} we have
  \begin{equation}
    \label{eq:Mk-k-1-k+1}
    M_k \leq \alpha ( M_{k-1}^+ \vee M_k^+ \vee M_{k+1}^+) \qquad 1 \leq k \leq d-1,
  \end{equation}
  and
  \begin{equation}
    \label{eq:M0-1}
    M_0 \leq \alpha M_1^+.
  \end{equation}
We show by induction that $M_k^+ \leq \alpha M_{k+1}^+$ for $0 \leq k \leq d-1$. For $k=0$, this follows from \eqref{eq:M0-1} and  the fact that $x \mapsto x^+$ is increasing. For $k \geq 1$ we have from \eqref{eq:Mk-k-1-k+1} that 
  \begin{equation*}
    M_k^+ \leq \alpha ( M_{k-1}^+ \vee M_k^+ \vee M_{k+1}^+) \leq \alpha ((\alpha M_k^+ \vee M_k^+) \vee M_{k+1}^+) = \alpha (M_k^+ \vee M_{k+1}^+).
  \end{equation*}
We claim that $M_k^+ \leq \alpha M_{k+1}^+$. 
The above inequality means that either $M_k^+ \leq \alpha M_k^+$ or $M_k^+ \leq \alpha M_{k+1}^+$. The latter case is precisely what we have claimed, while in the former case, as $\alpha < 1$, we have $M_k= 0$ and the inequality $M_k^+ \leq \alpha M_{k+1}^+$ is automatic. 
This implies that $M_0 \leq \alpha^d M_d^+$ and completes the proof. 
\end{proof}



Using the above local results, we
now show that if we add edges to a hypergraph and/or increase the
baseload, the total load should increase.

\begin{prop}
  \label{prop:bounded-partial-b-increasing}
  Let $H$ and $H'$ be two hypergraphs defined on the same vertex set $V = V(H) = V(H')$, such that $E(H) \subset E(H')$, and $H$ is bounded (but $H'$ is not necessarily bounded). Suppose two bounded baseload functions $b, b': V \rightarrow \reals$ are given, such that $b(i) \leq b'(i)$ for all $i \in V$.
If $\theta_\epsilon$ and $\theta'_\epsilon$ are two $\epsilon$--balanced allocations on $H$ and $H'$, respectively, for the respective baseload functions $b$ and $b'$, then we have $\partial_b \theta_\epsilon(i) \leq \partial_{b'} \theta'_\epsilon(i)$ for all $i \in V$. 
\end{prop}

\begin{proof}
  Since $H$ is bounded, there are constants $\Delta$ and $L$ such that, for all $i \in V$, $\deg_H(i) \leq \Delta$ and $|e| \leq L$ for all $e \in E(H)$. Also, as $b$ is bounded,  for some $K>0$ and all $i \in V$, we have $|b(i)| \leq K$. 

Now, fix some $i \in V$. Since $E(H) \subset E(H')$, we have $(H, i) \hookrightarrow_1 (H', i)$ with the identity map as the embedding function. With this, define
\begin{equation*}
  M := \sup_{j \in V} \partial_b \theta_\epsilon(j) - \partial_{b'} \theta'_\epsilon(j).
\end{equation*}
Note that, since $\partial_b \theta_\epsilon^b$ is  bounded to $\Delta + K$, the above quantity is finite and well defined. Now, using Lemma~\ref{lem:contraction-depth-1-baseload}, we have 
\begin{equation*}
  \partial_b \theta_\epsilon^b(i) - \partial_{b'} \theta_\epsilon^{b'}(i)
  \leq \frac{\Delta L}{4 \epsilon + \Delta L} M^+.
\end{equation*}
Taking the supremum over $i$ on the left hand side, we get
\begin{equation*}
  M \leq  \frac{\Delta L}{4 \epsilon + \Delta L} M^+,
\end{equation*}
which, since $ \frac{\Delta L}{4 \epsilon + \Delta L} < 1$, implies $M \leq 0$ and completes the proof.
\end{proof}

If we have a fixed bounded hypergraph $H$ with a baseload function, and $\theta_\epsilon, \theta_\epsilon'$ are two $\epsilon$--balanced allocations on $H$, repeating the above proposition twice, we get $\partial \theta_\epsilon \leq \partial \theta'_\epsilon$ and $\partial \theta'_\epsilon \leq \partial \theta_\epsilon$, which implies uniqueness. To sum up, we have:
\begin{cor}
  \label{cor:bounded-epsilon-balanced-unique}
  If a hypergraph $H$ is bounded,
  there is a unique $\epsilon$--balanced allocation with respect to any given baseload on it.
\end{cor}

\subsection{\texorpdfstring{$\epsilon$}{e}--balanced allocations for unbounded hypergraphs with respect to a baseload: canonical allocations}
\label{sec:epsil-balanc-alloc}

We now produce $\epsilon$--balanced allocations on a hypergraph $H$ 
with respect to a given baseload even when the hypergraph is not necessarily bounded. Note that we do not make any claims about uniqueness.

For a given hypergraph $H$ and $\Delta \in \nats$, define $H^\Delta$ to be the hypergraph with vertex set $V(H)$ and edge set $E^\Delta$, where
\begin{equation}
  \label{eq:E-delta}
  E^\Delta := \{ e \in E(H): |e| \leq \Delta, \deg_H(i) \leq \Delta \quad \forall i \in e \}.
\end{equation}
Given the baseload function $b$,
define $\theta_\epsilon^{\Delta}$ to be the unique $\epsilon$--balanced allocation on $H^\Delta$ with respect to the baseload $b$, where the existence and uniqueness of $\theta_\epsilon^{\Delta}$ is a consequence of
Corollary~\ref{cor:bounded-epsilon-balanced-unique} above. Since $E^\Delta$
increases to $E(H)$ as $\Delta$ increases, Proposition~\ref{prop:bounded-partial-b-increasing} implies that $\pbt_\epsilon^{\Delta}$ is pointwise increasing in $\Delta$.
Since it is also pointwise bounded, it is convergent. For a node $i \in V(H)$, define 
\begin{equation*}
  l_i := \lim_{\Delta \rightarrow \infty} \pbt_\epsilon^{\Delta}(i),
\end{equation*}
Now, for  $e \in E(H)$ and $i\in e$, define 
\begin{equation}
  \label{eq:def-epsilon-ba-canonical}
  \theta_\epsilon(e, i) := \frac{ \exp \left ( - \frac{l_i}{\epsilon} \right ) }{\sum_{j \in e} \exp \left ( - \frac{l_j}{\epsilon} \right )}.
\end{equation}

Now we prove that $\theta_\epsilon$ is an $\epsilon$--balanced allocation with respect to the baseload $b$. First, observe that, because of the normalizing term in the denominator, $\sum_{i \in e} \theta_\epsilon(e, i) = 1$ for all $e \in E(H)$. We now show that $l_i = \partial_b \theta_\epsilon(i)$. Note that 
\begin{equation*}
  \begin{split}
  \sum_{e \ni i}\theta_\epsilon(e, i) 
    &= \sum_{e \ni i} \frac{\exp \left ( - \frac{l_i}{\epsilon} \right )}{\sum_{j \in e} \exp \left ( - \frac{l_j}{\epsilon} \right ) } \\
    &= \lim_{\Delta \rightarrow \infty} \sum_{e \ni i} \frac{ \exp \left ( - \frac{\pbt_\epsilon^{\Delta}(i)}{\epsilon} \right ) }{\sum_{j \in e} \exp \left ( - \frac{\pbt_\epsilon^{\Delta}(j)}{\epsilon} \right )}. \\
  \end{split}
\end{equation*}
Observe that, for $\Delta > \max_{e \ni i} \left ( |e| \vee \max_{j \in e} \deg_H(j) \right )$, we have $e \in E^\Delta$ for all $e \ni i$. Hence,  the term inside the summation is 
$\theta_\epsilon^\Delta(e, i)$, because $\theta_\epsilon^\Delta$ is the unique
$\epsilon$-balanced allocation on $H^\Delta$ with respect to the baseload $b$.
Since we are taking $\Delta \rightarrow \infty$, we can assume it is big enough to get 
\begin{equation*}
  \partial_b \theta_\epsilon(i) = b(i) + \lim_{\Delta \rightarrow \infty} \sum_{e \ni i} \theta_\epsilon^{\Delta}(e, i) = \lim_{\Delta \rightarrow \infty} \partial_b \theta_\epsilon^{\Delta}(i) = l_i.
\end{equation*}
Substituting this into \eqref{eq:def-epsilon-ba-canonical}, we conclude that $\theta_\epsilon$ is $\epsilon$--balanced, with respect to the baseload $b$.

\begin{rem}
  \label{rem:non-uniqueness-of-epsilon--canonical}
  Note that the above procedure gives rise to an $\epsilon$--balanced allocation with respect to any given baseload, but we do not know if it is the only possible $\epsilon$--balanced allocation or not. In fact, we have proved uniqueness for bounded hypergraphs only. To emphasize this and avoid confusion, we call the $\epsilon$--balanced allocation resulting from the  procedure above the ``canonical'' allocation with respect to the given baseload. A special case of the procedure yields the canonical $\epsilon$--balanced allocation when there is no baseload.
\end{rem}

Now, we generalize the monotonicity property of Proposition~\ref{prop:bounded-partial-b-increasing} to the case of
not necessarily bounded hypergraphs,
for these canonical allocations.

\begin{prop}
  \label{prop:monotonicity-canonical-eba}
  Given hypergraphs $H$ and $H'$ on the same vertex set $V$, with $E(H) \subseteq E(H')$, and baseload functions $b, b': V \rightarrow \reals$ such that $b(i) \leq b'(i)$ for all $i \in V$, if $\theta_\epsilon$ and $\theta'_\epsilon$ are the canonical $\epsilon$--balanced allocations on $H$ and $H'$, with respect to baseloads $b$ and $b'$, respectively, then $\pbt_\epsilon(i) \leq \pbpt'_\epsilon(i)$ for all $i \in V$.
\end{prop}

\begin{proof}
Set $E := E(H)$ and $E' := E(H')$.
Note that since $E \subseteq E'$ and the vertex sets are the same for $H$ and $H'$, $E^\Delta \subseteq E'$ for any $\Delta$. Thus, if $\theta_\epsilon^{\Delta}$ is the unique  $\epsilon$--balanced allocation on $\langle V, E^\Delta \rangle$ with respect to the baseload $b$ and $\theta'_\epsilon$ is the canonical $\epsilon$--balanced allocation on $\langle V, E' \rangle$ with respect to the baseload $b'$,  Proposition~\ref{prop:bounded-partial-b-increasing} implies that $\pbt_\epsilon^{\Delta}(i) \leq \pbpt'_\epsilon(i)$ for all $i \in V$ (note that in Proposition~\ref{prop:bounded-partial-b-increasing} only the smaller hypergraph needs to be bounded, hence we only need to truncate $E$). By sending $\Delta$ to infinity, we get the desired result.  
\end{proof}

\subsection{Nonexpansivity}
\label{sec:non-expansivity}

\begin{prop}
  \label{prop:non-expansivitiy}
  Let $H$ be a given
   hypergraph and $b, b'$ be two baseload functions. Let $\theta$ and $\theta'$ be the canonical $\epsilon$--balanced allocations on $H$ with respect to $b$ and $b'$, respectively. Then, we have
  \begin{equation*}
    \norm{\pbt_\epsilon - \pbpt'_\epsilon}_{l^1(V(H))} \leq \norm{b - b'}_{l^1(V(H))}.
  \end{equation*}
\end{prop}
\begin{proof}
  To start with, assume that $H$ is bounded. Later, we will relax this assumption.
  
  Consider first the special case $b(i) \geq b'(i)$ for all $i \in V(H)$.
  Then, using the monotonicity  property of Proposition~\ref{prop:bounded-partial-b-increasing}, we have $\pbt_\epsilon(i) \geq \pbpt'_\epsilon(i)$ for all $i \in V(H)$. In particular,
  \begin{equation*}
    \norm{\pbt_\epsilon - \pbpt'_\epsilon}_{l^1(V(H))} = \sum_{i \in V(H)} \pbt_\epsilon(i) - \pbpt'_\epsilon(i).
  \end{equation*}
  Now, let $\{V_n\}$ be a nested sequence of finite subsets of $V(H)$ converging to $V(H)$, i.e. $V_n \uparrow V(H)$. Let $\theta_{n,\epsilon}$ and $\theta'_{n,\epsilon}$ be the $\epsilon$--balanced allocations 
  on $\langle V_n, E_H(V_n) \rangle$ with respect to the restrictions of $b$ and $b'$ to $V_n$, respectively.
  The monotonicity property of Proposition~\ref{prop:bounded-partial-b-increasing} and an argument similar to the one given in Section~\ref{sec:epsil-balanc-alloc} above, implies that $\pbt_{n,\epsilon}$ and $\pbpt'_{n,\epsilon}$ converge to $\epsilon$--balanced allocations on $H$ and $H'$ with respect to $b$ and $b'$, respectively. Since $H$ is bounded, there is a unique $\epsilon$--balanced allocations with respect to any baseload function. Therefore,  we have $\pbt_{n,\epsilon} \uparrow \pbt_\epsilon$ and $\pbpt'_{n,\epsilon} \uparrow \pbpt'_\epsilon$ as $n \to \infty$. Using the conservation of mass, we have 
  \begin{equation*}
    \sum_{i \in V_n} \pbt_{n,\epsilon}(i) - \pbpt'_{n,\epsilon}(i) = \sum_{i \in V_n} b(i) - b'(i) \leq \sum_{i \in V(H)} b(i) - b'(i) = \norm{b - b'}_{l^1(V(H))},
  \end{equation*}
  where we have used the assumption $b \geq b'$. In fact, for any finite subset of vertices $K \subset V(H)$, we have $K \subset V_n$ for $n$ large enough and so using $\pbt_{n,\epsilon} \uparrow \pbt_\epsilon$ and $\pbpt'_{n,\epsilon} \uparrow \pbpt'_\epsilon$ and also the monotonicity property of Proposition~\ref{prop:bounded-partial-b-increasing}, we have
  \begin{equation*}
    \begin{split}
    \sum_{i \in K} \pbt_\epsilon(i) - \pbpt'_\epsilon(i) &= \lim_{n \rightarrow \infty} \sum_{i \in K} \pbt_{n,\epsilon}(i) - \pbpt'_{n,\epsilon}(i) \\
    &\leq \lim_{n \rightarrow \infty} 
    \sum_{i \in V_n} \pbt_{n,\epsilon}(i) - \pbpt'_{n,\epsilon}(i) \\
    &\le \norm{b - b'}_{l^1(V(H))}.
    \end{split}
\end{equation*}
  Since this holds for every finite subset $K \subset V(H)$,
  we conclude that 
  \begin{equation*}
\norm{\pbt_\epsilon - \pbpt'_\epsilon}_{l^1(V(H))} \leq \norm{b - b'}_{l^1(V(H))}.
  \end{equation*}

  Now, continuing to assume that $H$ is bounded, suppose the condition $b \geq b'$ does not necessarily hold. Define $b'' := b \wedge b'$ with $\theta''_\epsilon$ being the unique $\epsilon$--balanced allocation on $H$ with respect to $b''$. Due to monotonicity property of Proposition~\ref{prop:bounded-partial-b-increasing}, we have $\pbt_\epsilon \geq \partial_{b''} \theta''_\epsilon$ and $\pbpt'_\epsilon \geq \partial_{b''} \theta''_\epsilon$. Thus, using the above argument and the triangle inequality, we have 
  \begin{equation*}
    \begin{split}
    \norm{\pbt_\epsilon - \pbpt'_\epsilon}_{l^!(V(H))} &\leq \norm{\pbt_\epsilon - \partial_{b''} \theta''_\epsilon}_{l^1(V(H))} + \norm{\pbpt'_\epsilon - \partial_{b''} \theta''_\epsilon}_{l^1(V(H))} \\
    & \leq \norm{b - b''}_{l^1(V(H))} + \norm{b' - b''}_{l^1(V(H))} \\
    &= \sum_{i \in V(H)} b(i) - b''(i) + b'(i) - b''(i) \\
    &= \sum_{i \in V(H)} |b(i) - b'(i)| = \norm{b - b'}_{l^1(V(H))}.
    \end{split}
  \end{equation*}

  Finally, we relax the boundedness assumption on $H$. 

We  take a not necessarily bounded hypergraph $H$ and let $H^\Delta$ be the truncation of $H$, as defined in Section~\ref{sec:epsil-balanc-alloc}. Let
$\theta_\epsilon^\Delta$ and 
$\theta_\epsilon^{\prime,\Delta}$ 
be the unique $\epsilon$--balanced allocations on $H^\Delta$, with respect to the baseloads $b$ and $b'$, respectively. Since $H^\Delta$ is bounded, using the above argument, we have
\begin{equation*}
  \norm{\pbt_\epsilon^{\Delta} - \partial_{b'} 
  \theta_\epsilon^{\prime, \Delta}}_{l^1(V(H))} \leq \norm{b - b'}_{l^1(V(H))}.
\end{equation*}
Hence, for any finite subset $K \subseteq V(H)$, we have 
\begin{equation*}
  \sum_{i \in K} | \partial_b \theta_\epsilon^{\Delta}(i) - \partial_{b'} \theta_\epsilon^{\prime,\Delta}(i)| \leq \norm{\partial_b \theta_\epsilon^{\Delta} - \partial_{b'} \theta_\epsilon^{\prime,\Delta}}_{l^1(V(H))} \leq \norm{b - b'}_{l^1(V(H))}.
\end{equation*}
Sending $\Delta$ to infinity and using the facts that $\partial \theta_\epsilon^\Delta$ and $\partial \theta_\epsilon^{\prime,\Delta}$ converge to $\partial \theta_\epsilon$ and $\partial \theta'_\epsilon$, respectively, and also the fact that $K$ is finite, we have 
\begin{equation*}
  \sum_{i \in K} | \partial_b \theta_\epsilon(i) - \partial_{b'} \theta'_\epsilon(i) | \leq \norm{b - b'}_{l^1(V(H))}.
\end{equation*}
Since the above is true for all finite $K$, we have 
\begin{equation*}
  \norm{\partial_b \theta_\epsilon - \partial_{b'} \theta'_\epsilon}_{l^1(V(H))} \leq \norm{b - b'}_{l^1(V(H))},
\end{equation*}
and the proof is complete. 
\end{proof}

\subsection{Regularity property for canonical 
$\epsilon$-balanced allocations 
with respect to a baseload}
\label{sec:regul-perop-stand}


In this section, we give a regularity property of canonical allocations which is crucial in our analysis.

Let $T$ be a hypertree. For a node $i \in V(T)$ and $e \in E(T)$, 
$e \ni i$, define $T_{e \rightarrow i}$ to the connected subtree with root $i$ that does not contain the part of $T$ directed from $e$. 
To be more precise, the vertex set of $T_{e \rightarrow i}$ is the set of vertices $j \in V(T)$ such that the shortest path from $i$ to $j$ does not contain $e$. 
The edge set of $T_{e \rightarrow i}$ contains all the edges with all their end points in this subset, i.e.  $E_T(V(T_{e \rightarrow i}))$ in the notation of Section~\ref{sec:hypergraphs}. 

\begin{prop}
  \label{prop:regularity-canonical-e-ba}
  Let $T$ be a hypertree, $b$ a baseload function on $T$, and $\theta$ the canonical $\epsilon$--balanced allocation on $T$ with respect to the baseload $b$. Let $e \in E(T)$ and $i \in e$, and let $\theta_{T_{e \rightarrow i}}$ denote the restriction of $\theta$ to $T_{e \rightarrow i}$, i.e. 
  \begin{equation*}
    \theta_{T_{e \rightarrow i}} (e' ,i') = \theta(e' ,i')~, \qquad e' \in E(T_{e \rightarrow i}), i' \in e'.
  \end{equation*}
  Then, $\theta_{T_{e \rightarrow i}}$ is the canonical $\epsilon$--balanced allocation on $T_{e \rightarrow i}$ with respect to the baseload function $\tilde{b}$ defined as 
  \begin{equation*}
    \tilde{b}(i) = b(i) + \theta(e ,i), 
  \end{equation*}
  and $\tilde{b}(j) = b(j)$ for $j \in V(T_{e \rightarrow i}) \setminus \{ i \}$.
\end{prop}

\begin{proof}

It is straightforward to check that 
$\theta_{T_{e \rightarrow i}}$ is an $\epsilon$--balanced allocation on $T_{e \rightarrow i}$ with baseload function $\tilde{b}$. 
Thus, the content of the theorem is the statement about this 
$\epsilon$--balanced allocation being the canonical 
$\epsilon$--balanced allocation on $T_{e \rightarrow i}$ with the baseload function $\tilde{b}$. 

Let $\tilde{\theta}$ denote the canonical $\epsilon$-balanced allocation on $T_{e \rightarrow i}$ with respect to the baseload $\tilde{b}$. Moreover, let $\tilde{\theta}^\Delta$ denote the unique $\epsilon$-balanced allocation on the bounded tree $(T_{e\rightarrow i})^\Delta$ with respect to the baseload $\tilde{b}$. Throughout this proof, we assume that $\Delta > \Delta_0$ where
\begin{equation*}
  \Delta_0 := |e| \vee \max \{ \deg_H(j): j \in e \}.
\end{equation*}
If $\Delta$ satisfies this property, then $e \in E(T^\Delta)$, and one can also check that 
\begin{equation}
  \label{eq:T-i-j-Delta-T-Delta-i-j}
  (T_{e\rightarrow i})^\Delta \equiv (T^\Delta)_{e \rightarrow i},
\end{equation}
so we can unambiguously write $T^\Delta_{e \rightarrow i}$ for this tree.
(Note that $T^\Delta$ need not be connected, but this is not
relevant.)
If $\theta^\Delta$ denotes the unique $\epsilon$--balanced allocation on $T^\Delta$ with respect to the baseload $b$ and $\theta^\Delta_{T_{e \rightarrow i}}$ is its restriction to $T^\Delta_{e \rightarrow i}$, then, since $\Delta \geq \Delta_0$, $\theta^\Delta_{T_{e \rightarrow i}}$ is the unique $\epsilon$-balanced allocation on $T^\Delta_{e \rightarrow i}$ with respect to the baseload $b^\Delta$ defined as $b^\Delta(i) = b(i) + \theta^\Delta(e ,i)$ and $b^\Delta(j) = b(j)$ for $j \neq i$. Now, 
$\tilde{\theta}^\Delta$ and $\theta^\Delta_{T_{e \rightarrow i}}$ are canonical $\epsilon$-balanced allocations on $T_{e \rightarrow i}^\Delta$ with respect to the baseloads $\tilde{b}$ and $b^\Delta$, respectively. Using Proposition~\ref{prop:non-expansivitiy}, we conclude that for an arbitrary vertex $k \in V(T_{e \rightarrow i})$ we have
\begin{equation*}
  | \partial_{b^\Delta} \theta^\Delta_{T_{e \rightarrow i}}(k) - \partial_{\tilde{b}} \tilde{\theta}^\Delta(k) | \leq \norm{ \partial_{b^\Delta} \theta^\Delta_{T_{e \rightarrow i}} - \partial_{\tilde{b}} \tilde{\theta}^\Delta}_{l^1(V(T_{e \rightarrow i}))} \leq | \theta^\Delta(e ,i) - \theta(e,i)|.
\end{equation*}
Now, sending $\Delta$ to infinity and noting the fact that $\theta^\Delta(e ,i) \rightarrow \theta(e ,i)$, we have $| \partial_{b^\Delta} \theta^\Delta_{T_{e \rightarrow i}}(k) - \partial_{\tilde{b}} \tilde{\theta}^\Delta(k)| \rightarrow 0$. 
Since $\theta^\Delta_{T_{ e \rightarrow i}}$ is the restriction to
$T^\Delta_{ e \rightarrow i}$ of the $\epsilon$-balanced allocation
$\theta^\Delta$ on $T^\Delta$ with respect to the baseload $b$, 
we have $\partial_{b^\Delta}\theta^\Delta_{T_{e \rightarrow i}}(k) = \partial_b \theta^\Delta(k)$ for all $k \in V(T^\Delta_{e \rightarrow i})$.
Since $\partial_b \theta^\Delta(k) \rightarrow \partial_b \theta(k)$ 
and $\partial_{\tilde{b}} \tilde{\theta}^\Delta(k) \rightarrow \partial_{\tilde{b}} \tilde{\theta}(k)$ for all $k \in V(T)$ as $\Delta \rightarrow \infty$, we conclude that 
\begin{equation*}
  \partial_{\tilde{b}} \tilde{\theta}(k) = \partial_b \theta(k)~, \qquad \forall k \in V(T_{e\rightarrow i}).
\end{equation*}
But it is straightforward to check that $\partial_b \theta(k) = 
\partial_{\tilde{b}} \theta_{T_{e \rightarrow i}} (k)$ for all
$k \in V(T_{e\rightarrow i})$. From the definition of $\epsilon$-balanced
allocations, we conclude that $\theta_{T_{e \rightarrow i}} (e',i')$
equals $\tilde{\theta}(e',i')$ for all $e' \in E(T_{e \rightarrow i})$
and all $i' \in e'$, i.e. that $\theta_{T_{e \rightarrow i}}$ is the 
canonical $\epsilon$-balanced allocation on $T_{e \rightarrow i}$
with respect to the baseload $\tilde{b}$, which was what was to be shown.
\end{proof}


\section{\texorpdfstring{$\epsilon$}{e}--balanced allocations on \texorpdfstring{$\mH_{**}$}{H**}}
\label{sec:epsil-balanc-alloc-1}



In this section, we discuss how to find an $\epsilon$--balanced allocation on $\mH_{**}$,
in the sense of Definition~\ref{def:epsilon-balanced-allocation-H*}, 
with respect to a unimodular measure $\mu \in \mP(\mH_*)$.
Recall that this is a Borel allocation $\Theta_\epsilon: \mH_{**} \rightarrow [0,1]$ such that 
\begin{equation*}
  \Theta_\epsilon(H,e, i) = \frac{\exp (- \partial \Theta_\epsilon(H, i) / \epsilon)}{\sum_{j \in e} \exp (- \partial \Theta_\epsilon(H, j) / \epsilon)}, \quad \vmu\text{--a.e.}.
\end{equation*}
What we do here is in fact stronger, in the sense that we introduce a Borel allocation $\Theta_\epsilon$ such that the above condition is satisfied pointwise, i.e.
\begin{equation}
  \label{eq:Theta-epsilon-pointwise-epsilon-balanced}
  \Theta_\epsilon(H,e, i) = \frac{\exp (- \partial \Theta_\epsilon(H, i) / \epsilon)}{\sum_{j \in e} \exp (- \partial \Theta_\epsilon(H, j) / \epsilon)} \quad \forall [H, e, i] \in \mH_{**}.
\end{equation}
In fact, we can define $\Theta_\epsilon(H, e, i)$ to be $\theta^H_\epsilon(e, i)$ where $\theta^H_\epsilon$ is the canonical $\epsilon$-balanced allocation for $H$, as was  introduced in Section~\ref{sec:epsil-balanc-alloc}. Defining $\Theta_\epsilon$ in this way guarantees that \eqref{eq:Theta-epsilon-pointwise-epsilon-balanced} is satisfied. Therefore, it remains to show that $\Theta_\epsilon$ is a Borel allocation. 
In the following, we construct $\Theta_\epsilon$ differently, but as we will see later, the $\Theta_\epsilon(H, e, i)$ we construct will be equal to $\theta_\epsilon^H(e, i)$ (see Remark~\ref{rem:Theta-epsilon-H**-canonical} below).

For the construction, given $\Delta \in \nats$,  define $F^\Delta_\epsilon : \mH_* \rightarrow \reals$ via $F_\epsilon^\Delta(H, i) := \partial \theta^{H^\Delta}_\epsilon(i)$, where $H^\Delta$ is the truncated hypergraph introduced in Section~\ref{sec:epsil-balanc-alloc}. The uniqueness property of $\epsilon$--balanced allocations for bounded hypergraphs (Corollary~\ref{cor:bounded-epsilon-balanced-unique}) implies that the above definition does not depend on the specific choice of $(H,i)$ in the equivalence class. 
Therefore, $F_\epsilon^\Delta$ is well defined. We claim that $F_\epsilon^\Delta$ is a continuous function on $\mH_*$. 
Indeed, if $(H_1, i_1) \equiv_d (H_2, i_2)$, then $(H_1^\Delta, i_1) \equiv_{d-1} (H_2^\Delta, i_2)$, and using Lemma~\ref{lem:contraction-base-d-baseload} we have
\begin{equation*}
  \left | \partial \theta^{H_1^\Delta}_\epsilon(i_1) - \partial \theta^{H_2^\Delta}_\epsilon(i_2) \right | \leq \left ( \frac{\Delta^2}{4\epsilon + \Delta^2} \right)^{d-1} \Delta~.
\end{equation*}
This implies that  $F_\epsilon^\Delta$ is (uniformly) continuous. Moreover, 
Proposition~\ref{prop:bounded-partial-b-increasing} implies that $F_\epsilon^\Delta$ is pointwise increasing in $\Delta$.
On the other hand, $F_\epsilon^\Delta(H, i) \leq \deg_H(i)$. Hence, 
there is a pointwise limit
\begin{equation}
  \label{eq:F-epsilon-lim-F-epsilon-Delta}
  F_\epsilon(H, i) := \lim_{\Delta \rightarrow \infty} F_\epsilon^\Delta(H, i).
\end{equation}
We now define $\Theta_\epsilon: \mH_{**} \rightarrow [0,1]$ in the following way:
\begin{equation}
  \label{eq:Theta-epsilon-definition-F}
  \Theta_\epsilon(H, e, i) := \frac{\exp ( - F_\epsilon(H, i) / \epsilon)}{\sum_{j \in e} \exp (- F_\epsilon(H, j) / \epsilon)}.
\end{equation}
We want to show that $\Theta_\epsilon$ is an $\epsilon$--balanced allocation on $\mH_{**}$.
Since $F_\epsilon^\Delta$ is continuous, $F_\epsilon$ is Borel, and hence $\Theta_\epsilon$ is a Borel map. Also, the normalization factor in the denominator guarantees that $\sum_{i \in e} \Theta_\epsilon(H, e, i) = 1$, and hence $\Theta_\epsilon$ is a Borel allocation. Comparing \eqref{eq:Theta-epsilon-pointwise-epsilon-balanced} 
with \eqref{eq:Theta-epsilon-definition-F}, it suffices to show that
\begin{equation}
  \label{eq:partial-Theta-epsilon=F-epsilon}
  \partial \Theta_\epsilon(H, i) = F_\epsilon(H, i).
\end{equation}
In order to show this, note that 
\begin{equation*}
  \begin{split}
    \partial \Theta_\epsilon(H, i) &= \sum_{e \ni i} \Theta_\epsilon(H, e, i) \\
    &= \sum_{e \ni i} \frac{\exp(- F_\epsilon(H,i) / \epsilon)}{\sum_{j \in e} \exp ( - F_\epsilon(H, j) / \epsilon)} \\
    &= \lim_{\Delta \rightarrow \infty} \sum_{e \ni i} \frac{\exp(- F_\epsilon^\Delta(H,i) / \epsilon)}{\sum_{j \in e} \exp ( - F_\epsilon^\Delta(H, j) / \epsilon)}.
  \end{split}
\end{equation*}
Now, when $\Delta \geq \max_{e \ni i} |e|$ and also $\Delta \geq \max_{j \in e, e \ni i} \deg_H(j))$,  we have  $e \in E(H^\Delta)$, and we have
\begin{equation*}
  \frac{\exp(- F_\epsilon^\Delta(H,i) / \epsilon)}{\sum_{j \in e} \exp ( - F_\epsilon^\Delta(H, j) / \epsilon)}  =   \frac{\exp(- \partial \theta_\epsilon^{H^\Delta}(H,i) / \epsilon)}{\sum_{j \in e} \exp ( - \partial \theta_\epsilon^{H^\Delta}(H, j) / \epsilon)}= \theta_\epsilon^{H^\Delta}(H,e , i).
\end{equation*}
Consequently,
\begin{equation*}
  \begin{split}
    \partial \Theta_\epsilon(H, i) &= \lim_{\Delta \rightarrow \infty} \sum_{e \ni i} \theta_\epsilon^{H^\Delta}(H, e, i) \\
    &= \lim_{\Delta \rightarrow \infty} \partial \theta_\epsilon^{H^\Delta}(H, i) \\
    &= \lim_{\Delta \rightarrow \infty} F_\epsilon^\Delta(H, i) \\
    &= F_\epsilon(H, i).
  \end{split}
\end{equation*}
Therefore, we have shown \eqref{eq:partial-Theta-epsilon=F-epsilon}, which shows that $\Theta_\epsilon$ satisfies \eqref{eq:Theta-epsilon-pointwise-epsilon-balanced} and hence is an $\epsilon$--balanced allocation (both pointwise and $\vmu$--almost everywhere).

\begin{rem}
  \label{rem:Theta-epsilon-H**-canonical}
  Note that \eqref{eq:F-epsilon-lim-F-epsilon-Delta} means that $F_\epsilon(H, i) = \partial \theta^H_\epsilon(i)$, where $\theta^H_\epsilon$ is the canonical $\epsilon$-balanced allocation in $H$. Hence, \eqref{eq:Theta-epsilon-definition-F} implies that $\Theta_\epsilon(H, e, i)$ is pointwise equal to  $\theta_\epsilon^H(e, i)$. 
\end{rem}

The following Proposition proves an almost sure uniqueness property for Borel $\epsilon$--balanced allocations which is similar in flavor to part 4 of Theorem~\ref{thm:balanced--properties}.
The proof of this statement is given in Appendix~\ref{sec:proof-proposition-eba-uniqueness}.

\begin{prop}
  \label{prop:epsilon-balanced-uniqueness}
  Assume $\mu$ is a unimodular measure on $\mH_*$  such that $\deg(\mu) < \infty$. Given $\epsilon > 0$, let $\Theta_\epsilon$ be the $\epsilon$--balanced allocation defined in this section, and let $\Theta'_\epsilon$ be any other $\epsilon$--balanced allocation, both with respect to $\mu$. Then, we  have $\partial \Theta_\epsilon = \partial \Theta'_\epsilon$, $\mu$--a.s.\ and $\Theta_\epsilon = \Theta'_\epsilon$, $\vmu$--a.e..
\end{prop}


\section{Properties of balanced allocations}
\label{sec:prop-balanc-alloc-1}

\subsection{Existence}
\label{sec:existence}


In this section we prove the existence of balanced allocations (part 1 of Theorem~\ref{thm:balanced--properties}):

\begin{prop}
  \label{prop:epsilon-to-zero}
  Assume $\mu \in \mP(\mH_*)$ is unimodular with $\deg(\mu) < \infty$. Then,  there is a sequence $\epsilon_k \downarrow 0$ such that  $\Theta_{\epsilon_k}$ converges to a balanced allocation $\Theta_0$, with the convergence being both in $L^2(\vmu)$ and $\vmu$--almost everywhere.
\end{prop}

\begin{proof}

First, we show that $\Theta_{\epsilon}$ is Cauchy in $L^2(\vmu)$.  To do so, we take $\epsilon, \epsilon' > 0$ and try to bound $\norm{\Theta_\epsilon - \Theta_{\epsilon'}}_{L^2(\vmu)}$. 
For an integer $\Delta > 0$, define the function $\Theta_\epsilon^\Delta$ on $\mH_{**}$ as follows:
\begin{equation}
\label{eq:Theta-epsilon-Delta-proof-existence}
  \Theta_\epsilon^\Delta(H, e, i) =
  \begin{cases}
    \theta^{H^\Delta}_\epsilon(e, i) & \mbox{ if $e \in E(H^\Delta)$}, \\
    0 & \text{otherwise},
  \end{cases}
\end{equation}
where $H^\Delta$ is the truncated hypergraph introduced in Section~\ref{sec:epsil-balanc-alloc} and $\theta_\epsilon^{H^\Delta}$ is the unique $\epsilon$--balanced allocation associated to it.  

Take a locally finite hypergraph $H$ and  an edge $e$ in $E(H^\Delta)$. As $\theta_\epsilon^{H^\Delta}$ is $\epsilon$--balanced on $H^\Delta$ and $e \in E(H^\Delta)$, for $i, j \in e$ we have 
\begin{equation*}
  \frac{\theta_\epsilon^{H^\Delta}(e, i)}{\theta_\epsilon^{H^\Delta}(e, j)} = \frac{\exp (-\partial \theta_\epsilon^{H^\Delta}(i) / \epsilon)}{\exp (-\partial \theta_\epsilon^{H^\Delta}(j) / \epsilon)}.
\end{equation*}
By the definition of $\Theta_\epsilon^\Delta$, if $e \in E(H^\Delta)$, then for all $j' \in e$, $\Theta_\epsilon^\Delta(H,e, j') = \theta_\epsilon^{H^\Delta}(e, j')$ and $\partial \Theta_\epsilon^\Delta(H, j') = \partial \theta_\epsilon^{H^\Delta}(j')$. Taking logarithms on both sides of the above equation, we have 
\begin{equation*}
  \partial \Theta_{\epsilon}^\Delta(H, i)  + \epsilon \log \Theta_{\epsilon}^\Delta(H, e, i) = \partial \Theta^\Delta_{\epsilon}(H, j) + \epsilon \log \Theta^\Delta_{\epsilon}(H, e, j), \qquad \forall i, j \in e,
\end{equation*}
with this equation holding pointwise whenever $e \in E(H^\Delta)$. 
As this equality holds for all $i, j \in e$, any two convex combinations of the values of $\partial \Theta^\Delta_\epsilon + \epsilon \log \Theta^\Delta_\epsilon$ evaluated at nodes in $e$ are equal, 
whenever $e \in E(H^\Delta)$. In particular, if $\Theta^\Delta_{\epsilon'}$ denotes the $\epsilon'$--balanced allocation on $\mH_{**}$ defined similarly to the above, then,
whenever $e \in E(H^\Delta)$, 
we have $\sum_{i \in e} \Theta_\epsilon^\Delta(H, e, i) = \sum_{i \in e} \Theta_{\epsilon'}^{\Delta}(H, e, i) = 1$. Hence we have:
\begin{equation}
\label{eq:proof-existence-epsilon-epsilon'}
\begin{split}
\sum_{i \in e} \Theta^\Delta_\epsilon(H, e, i) &\left (   \partial \Theta^\Delta_{\epsilon}(H, i)  + \epsilon \log \Theta^\Delta_{\epsilon}(H, e, i) \right ) \\
&= \sum_{i \in e} \Theta^\Delta_{\epsilon'}(H, e, i) \left (   \partial \Theta^\Delta_{\epsilon}(H, i)  + \epsilon \log \Theta^\Delta_{\epsilon}(H, e, i) \right ),
\end{split}
\end{equation}
which holds pointwise, whenever $e \in E(H^\Delta)$. On the other hand, if $e \notin E(H^\Delta)$, then, by definition, $\Theta^\Delta_\epsilon(H, e, j)$ as well as $\Theta^\Delta_{\epsilon'}(H, e, j)$ are zero for all $j \in e$. In this case the above equality again holds, with  $0 \log 0$ being interpreted as 0. 
In other words, pointwise on $\mH_{**}$, we have
\begin{equation*}
  \nabla ( \Theta_\epsilon^\Delta(\partial \Theta^\Delta_\epsilon + \epsilon \log \Theta^\Delta_\epsilon)) = \nabla (\Theta^\Delta_{\epsilon'}(\partial \Theta^\Delta_\epsilon + \epsilon \log \Theta^\Delta_\epsilon)),
\end{equation*}
where, by abuse of notation, we have treated $\partial \Theta_\epsilon^\Delta$ as a function on $\mH_{**}$ rather than on $\mH_*$, via $\partial \Theta_\epsilon^\Delta(H, e, i) := \partial \Theta_\epsilon^\Delta(H, i)$, and likewise for $\partial \Theta_{\epsilon'}^\Delta$.
Rewriting the above identity, we have:
\begin{equation}
  \label{eq:epsilon-zero-nabla-3-1}
  \nabla \left ( \Theta^\Delta_\epsilon \partial \Theta^\Delta_\epsilon + \epsilon \Theta^\Delta_\epsilon \log \Theta^\Delta_\epsilon - \Theta^\Delta_{\epsilon'} \partial \Theta^\Delta_{\epsilon} \right ) = \epsilon \nabla (\Theta^\Delta_{\epsilon'} \log \Theta^\Delta_\epsilon).
\end{equation}
Note that $\partial \Theta_\epsilon^\Delta$ is pointwise bounded by $\Delta$ by definition. Moreover, $\deg (\mu) < \infty$, which implies that $\vmu$ has finite total measure. Hence, all terms in the above equation have finite integral. 
On the other hand, from the definition of $\vmu$, we have: 
\begin{equation*}
  \int (\Theta^\Delta_\epsilon - \Theta^\Delta_{\epsilon'}) \partial \Theta^\Delta_\epsilon d \vmu = \int \partial ( (\Theta^\Delta_\epsilon - \Theta^\Delta_{\epsilon'}) \partial \Theta^\Delta_\epsilon) d \mu = \int (\partial \Theta^\Delta_\epsilon - \partial \Theta^\Delta_{\epsilon'}) \partial \Theta^\Delta_\epsilon d \mu.
\end{equation*}
Substituting this into \eqref{eq:epsilon-zero-nabla-3-1} and using unimodularity, we have 
\begin{equation}
  \label{eq:entropy-kl-epsilon-epsilon'}
  \langle \partial \Theta^\Delta_\epsilon - \partial \Theta^\Delta_{\epsilon'} , \partial \Theta^\Delta_\epsilon \rangle + \epsilon \int  \Theta^\Delta_\epsilon \log \Theta^\Delta_\epsilon d \vmu = \epsilon \int \Theta^\Delta_{\epsilon'} \log \Theta^\Delta_{\epsilon} d \vmu,
\end{equation}
where $\langle .,. \rangle$ denotes the inner product of two functions in $L^2(\mu)$. 
Now, changing the order of $\epsilon$ and $\epsilon'$, we have
\begin{equation}
  \label{eq:entropy-kl-epsilon'-epsilon}
  \langle \partial \Theta^\Delta_{\epsilon'} - \partial \Theta^\Delta_{\epsilon} , \partial \Theta^\Delta_{\epsilon'} \rangle + \epsilon' \int  \Theta^\Delta_{\epsilon'} \log \Theta^\Delta_{\epsilon'} d \vmu = \epsilon' \int \Theta^\Delta_{\epsilon} \log \Theta^\Delta_{\epsilon'} d \vmu.
\end{equation}
Summing up these two equalities, we have:
\begin{equation}
\label{eq:epsilon-epsilon'-norm2-4}
\begin{split}
  \norm{\partial \Theta^\Delta_\epsilon - \partial \Theta^\Delta_{\epsilon'} }_2^2 &= \epsilon \int \Theta^\Delta_{\epsilon'} \log \Theta^\Delta_\epsilon d\vmu- \epsilon \int \Theta^\Delta_\epsilon \log \Theta^\Delta_\epsilon d \vmu \\
&+ \epsilon' \int \Theta^\Delta_\epsilon \log \Theta^\Delta_{\epsilon'} d\vmu - \epsilon' \int \Theta^\Delta_{\epsilon'} \log \Theta^\Delta_{\epsilon'} d \vmu.
\end{split}
\end{equation}
We now use the following information theoretic notation. For functions $\Theta, \Theta': \mH_{**} \rightarrow [0,1]$, we define
\begin{equation*}
  \begin{split}
  H(\Theta) & := - \int \Theta \log \Theta d \vmu, \mbox{ and }\\
  D(\Theta \Vert \Theta') & := \int \Theta \log \frac{\Theta}{\Theta'} 
  d \vmu,
  \end{split}
\end{equation*}
where $0 \log 0$ is interpreted as 0, and $\Theta(H, e, i)$ is assumed to be zero whenever $\Theta'(H, e, i)$ is zero (by definition, $\Theta_\epsilon^\Delta$ and $\Theta_{\epsilon'}^\Delta$ have this property). Rearranging the terms in \eqref{eq:epsilon-epsilon'-norm2-4},
\begin{equation}
  \label{eq:epsilon-epsilon'-D-H}
  \norm{\partial \Theta^\Delta_\epsilon - \partial \Theta^\Delta_{\epsilon'}}_2^2 + \epsilon D(\Theta^\Delta_{\epsilon'} \Vert \Theta^\Delta_{\epsilon}) + \epsilon' D(\Theta^\Delta_\epsilon \Vert \Theta^\Delta_{\epsilon'}) = (\epsilon - \epsilon') (H(\Theta^\Delta_\epsilon) - H(\Theta^\Delta_{\epsilon'}) ).
\end{equation}
Note that, since $\deg(\mu) < \infty$ and $\partial \Theta_\epsilon^\Delta$ and $\partial \Theta_{\epsilon'}^\Delta$ are pointwise bounded to $\Delta$, all the terms are finite. 
With, $\mA := \{ [H, e, i] \in \mH_{**}: e \in E(H^\Delta) \}$, we have 
\begin{equation}
\label{eq:D-theta-epsilon'-epsilon-nonnegative} 
  \begin{split}
    D(\Theta_{\epsilon'}^\Delta \Vert \Theta_\epsilon^\Delta) &= \int_\mA \Theta^\Delta_{\epsilon'} \log \frac{\Theta_{\epsilon'}^\Delta}{\Theta_\epsilon^\Delta} d \vmu \\
    &= \int_\mA \nabla \left ( \Theta_{\epsilon'}^\Delta \log \frac{\Theta_{\epsilon'}^\Delta}{\Theta_{\epsilon}^\Delta} \right ) d \vmu \\
    &= \int_\mA \frac{1}{|e|} D_{\text{KL}}((\Theta_{\epsilon'}^\Delta(H, e, j))_{j \in e} \Vert (\Theta_\epsilon^\Delta(H, e, j))_{j \in e}) d \vmu (H, e, i)\\
    &\geq 0,
  \end{split}
\end{equation}
where $D_{\text{KL}}$ denotes the standard KL divergence, and for $[H, e, i] \in \mA$, by definition, $(\Theta_{\epsilon'}^\Delta(H, e, j))_{j \in e}$ and $(\Theta_{\epsilon'}^\Delta(H, e, j))_{j \in e}$ are vectors of nonnegative values summing up to one. Similarly, one can show that  $D(\Theta^\Delta_\epsilon \Vert \Theta^\Delta_{\epsilon'}) \geq 0$. Therefore, all the terms on the left hand side of \eqref{eq:epsilon-epsilon'-D-H} are nonnegative. As a result
\begin{equation}
  \label{eq:entropy-decreasing-Delta}
  \epsilon > \epsilon' \qquad \Rightarrow \qquad H(\Theta^\Delta_\epsilon) \geq H(\Theta^\Delta_{\epsilon'}).
\end{equation}
As $H(\Theta_\epsilon^\Delta) \geq 0$ for all $\epsilon>0$, the above inequality implies that $H(\Theta_\epsilon^\Delta)$ converges to some value as $\epsilon \downarrow 0$.
Another result of \eqref{eq:epsilon-epsilon'-D-H} is that  for $\epsilon > \epsilon'$,
\begin{equation*}
  \epsilon D(\Theta^\Delta_{\epsilon'} \Vert \Theta^\Delta_{\epsilon}) \leq (\epsilon - \epsilon') (H(\Theta^\Delta_\epsilon) - H(\Theta^\Delta_{\epsilon'})).
\end{equation*}
Dividing by $\epsilon$ and using $\epsilon > \epsilon'$ and $H(\Theta^\Delta_\epsilon) \geq H(\Theta^\Delta_{\epsilon'})$, we get 
\begin{equation}
\label{eq:D-theta-epsilon'-epsilon-H}
  D(\Theta^\Delta_{\epsilon'} \Vert \Theta^\Delta_{\epsilon}) \leq H(\Theta^\Delta_\epsilon) - H(\Theta^\Delta_{\epsilon'}).
\end{equation}
Using \eqref{eq:D-theta-epsilon'-epsilon-nonnegative} and Pinsker's inequality (see, for instance, \cite{csiszar2011information}), 
\begin{equation*}
  \begin{split}
    D(\Theta^\Delta_{\epsilon'} \Vert \Theta^\Delta_\epsilon) 
    &= \int_\mA \frac{1}{|e|} D_{\text{KL}} ( (\Theta^\Delta_{\epsilon'}(H, e,j))_{j \in e} \Vert (\Theta^\Delta_\epsilon(H, e,j))_{j \in e}) d \vmu \\
    &\stackrel{(a)}{\geq} \int_\mA \frac{1}{|e|} \frac{1}{2} \sum_{j \in e} | \Theta^\Delta_{\epsilon'}(H, e, j) - \Theta^\Delta_\epsilon(H, e, j) |^2 d \vmu \\
    &\stackrel{(b)}{=} \frac{1}{2} \int_\mA | \Theta^\Delta_{\epsilon'} - \Theta^\Delta_\epsilon |^2 d \vmu \\
    &= \frac{1}{2} \norm{\Theta^\Delta_\epsilon - \Theta^\Delta_{\epsilon'}}_2^2,
  \end{split}
\end{equation*}
where $(a)$ uses the Pinsker's inequality, and $(b)$ uses unimodularity of $\mu$.
 Combining this with \eqref{eq:D-theta-epsilon'-epsilon-H}, for $\epsilon, \epsilon' > 0$ we have 
\begin{equation*}
  \norm{\Theta^\Delta_{\epsilon} - \Theta^\Delta_{\epsilon'}}_2^2 \leq 2 | H(\Theta^\Delta_\epsilon) - H(\Theta^\Delta_{\epsilon'}) |.
\end{equation*}
Now, as we send $\Delta$ to infinity, $\Theta_\epsilon^\Delta$ converges pointwise to the function $\Theta_\epsilon$ defined in Section~\ref{sec:epsil-balanc-alloc-1}. Moreover, $\deg (\mu) < \infty$ and the function $x \log x$ is bounded for $x \in [0,1]$. Hence, using dominated convergence theorem, we have 
\begin{equation}
\label{eq:theta-epsilon-theta-epsilon-prime-H}
  \norm{\Theta_{\epsilon} - \Theta_{\epsilon'}}_2^2 \leq 2 | H(\Theta_\epsilon) - H(\Theta_{\epsilon'}) |.
\end{equation}
Similarly, sending $\Delta \rightarrow \infty$, \eqref{eq:entropy-decreasing-Delta} implies that
\begin{equation}
  \label{eq:entropy-decreasing}
  \epsilon > \epsilon'>0 \qquad \Rightarrow \qquad H(\Theta_\epsilon) \geq H(\Theta_{\epsilon'}).
\end{equation}
This means that, as $H(\Theta_\epsilon) \geq 0$ for all $\epsilon>0$,  $H(\Theta_\epsilon)$ is convergent as $\epsilon \downarrow 0$. In particular, this together with \eqref{eq:theta-epsilon-theta-epsilon-prime-H} implies that there is a sequence of positive values $\epsilon_k$ converging to zero such that $\Theta_{\epsilon_k}$ converges to some $\Theta_0: \mH_{**} \rightarrow [0,1]$ in $L^2(\vmu)$. Therefore, there is a subsequence of this sequence converging to $\Theta_0$ $\vmu$--almost everywhere. Without loss of generality, we may assume this subsequence is the whole sequence. With this, $\Theta_{\epsilon_k}$ converges to $\Theta_0$ both in $L^2(\vmu)$ and $\vmu$--almost everywhere. Using Lemma~\ref{lem:vmu-a.e.-del-mu-a.s.}, we have $\partial \Theta_{\epsilon_k} \rightarrow \partial \Theta_0$ $\mu$--almost surely.  
Lemma~\ref{prop:everyting-shows-at-the-root} then implies that for $\mu$--almost all $[H, i] \in \mH_*$, $\partial \Theta_{\epsilon_k}(H, j) \rightarrow \partial \Theta_0(H, j)$  for all $j \in e$. Following Remark~\ref{rem:partial-theta-H**-unify}, we can treat $\partial \Theta_0$ and the $\partial \Theta_{\epsilon_k}$ as functions on $\mH_{**}$ instead of $\mH_*$. Then, Lemma~\ref{lem:a.s.in-mu--a.e.in-vmu} implies that 
\begin{equation}
  \label{eq:for-all-j-in-e-vmu-a.e}
  \partial \Theta_{\epsilon_k}(H, e, j) \rightarrow \partial \Theta_0 (H, e, j) \,\, \forall j \in e, \qquad \qquad \vmu\text{--a.e.}.
\end{equation}
Now, we are ready to show that $\Theta_0$ is actually $\vmu$--balanced. To do so, assume that $\partial \Theta_0(H, i) >  \partial \Theta_0(H, j)$ for some $i, j \in e$. Equivalently, $\partial \Theta_0(H, e, i) > \partial \Theta_0(H, e, j)$. Then \eqref{eq:for-all-j-in-e-vmu-a.e} implies that, outside a measure zero set, for some fixed $\delta$, and for $k$ large enough,
\begin{equation*}
  \partial \Theta_{\epsilon_k}(H, e, i) - \partial \Theta_{\epsilon_k}(H, e, j) > \delta \qquad \vmu\text{--a.e.}.
\end{equation*}
On the other hand, using the definition of an $\epsilon$--balanced allocation, we have
\begin{equation*}
  \Theta_{\epsilon_k}(H, e, i) \leq \frac{1}{1 + \exp \left ( - \frac{\partial \Theta_{\epsilon_k}(H,e, j) - \partial \Theta_{\epsilon_k}(H, e, i)}{\epsilon_k} \right )} \leq \frac{1}{1+ \exp (\delta/ \epsilon_k)}.
\end{equation*}
Sending $k$ to infinity, since $\delta$ is fixed, the above inequality implies that  $\Theta_{\epsilon_k}(H, e, i)$ converges to zero. Also, we know that $\Theta_{\epsilon_k} \rightarrow \Theta_0$ $\vmu$--almost everywhere. Thus, we have shown that
\begin{equation*}
  \partial \Theta_0(H,i) > \partial \Theta_0(H, j) \,\,\text{for}\,\, i, j \in e \qquad \Rightarrow \qquad \Theta_0(H, e, i) =0, \qquad \qquad \vmu\text{--a.e.},
\end{equation*}
which shows that $\Theta_0$ is balanced and the proof is complete.
\end{proof}

\subsection{Variational characterization}
\label{sec:vari-char}


In this section, we prove the variational characterization (part 2) of Theorem~\ref{thm:balanced--properties}.

\begin{prop}
  \label{prop:variational}
  Assume $\Theta$ is a Borel allocation on $\mH_{**}$ and $\mu \in \mP(\mH_*)$ is unimodular with $\deg (\mu) < \infty$. Then, for all $t \in \reals$ we have
  \begin{equation}
    \label{eq:variational-inequality}
    \int (\partial \Theta - t)^+ d \mu \geq \sup_{\stackrel{f: \mH_* \rightarrow [0,1]}{\text{Borel}}} \int \tilde{f}_\text{min} d \vmu - t \int f d\mu,
  \end{equation}
  where $\tilde{f}_\text{min}$ is defined as 
  \begin{equation*}
    \tilde{f}_\text{min} (H, e, i) = \frac{1}{|e|} \min_{j \in e} f(H, j).
  \end{equation*}
  Furthermore, equality happens for all $t \in \reals$ if and only if $\Theta$ is balanced. Moreover, when $\Theta$ is balanced, the function $f = \oneu{\partial \Theta > t}$ achieves the supremum.
\end{prop}

\begin{proof}
  Note that for any real number $x$, we have $x^+ \geq x y$ for $y \in [0,1]$. Therefore, for any Borel function $f: \mH_* \rightarrow [0,1]$, we have
  \begin{equation}
    \label{eq:variational-1}
    \int (\partial \Theta - t)^+ d \mu \geq \int (\partial \Theta - t) f d\mu = \int f \partial \Theta d\mu - t \int f d\mu.
  \end{equation}
  The assumption that $\deg(\mu) < \infty$ guarantees that both integrals on the RHS  are finite, and hence $\int (\partial \theta - t) f d \mu$ exists.
  It is easy to see that  $f \partial \Theta = \partial (\tilde f \Theta)$, where $\tilde f: \mH_{**} \rightarrow \reals$ is defined as $\tilde{f}(H,e ,i) := f(H, i)$. 
  Therefore, using the unimodularity of $\mu$, we have
  \begin{equation}
    \label{eq:f-del--Del-tilde-f}
    \int f \partial \Theta d \mu = \int \tilde{f} \Theta d \vmu = \int \nabla (\tilde{f} \Theta) d \vmu.
  \end{equation}
Now, we have
  \begin{equation}
    \label{eq:Del-tilde-f-tilde-f-min}
    \begin{split}
      \nabla (\tilde{f} \Theta) (H, e, i) &= \frac{1}{|e|} \sum_{j \in e} \tilde{f}(H, e, j) \Theta(H, e, j) \\
      &= \frac{1}{|e|} \sum_{j \in e} f(H, j) \Theta(H, e, j) \\
      &\stackrel{(a)}{\geq} \frac{1}{|e|} \min_{j \in e} f(H, j) \\
      &= \tilde{f}_\text{min} (H, e, i),
    \end{split}
  \end{equation}
  where $(a)$ holds since $\sum_{j \in e} \Theta(H, e, j) = 1$ and $\Theta(H, e, j) \geq 0$ for all $j \in e$. 
  This, together with \eqref{eq:variational-1} and \eqref{eq:f-del--Del-tilde-f}, proves \eqref{eq:variational-inequality}. 

Now, we show that equality holds if and only if $\Theta$ is balanced.
   First, assume $\Theta$ is balanced. We show that equality holds in \eqref{eq:variational-inequality} for all $t \in \reals$. Take $f = \oneu{\partial \Theta > t}$. Then, \eqref{eq:variational-1} becomes an equality. Therefore, it remains to show that \eqref{eq:Del-tilde-f-tilde-f-min} also becomes an equality, $\vmu$--almost everywhere. Note that if $f(H, j) = 0$ for all $j \in e$ or $f(H, j) = 1$ for all $j \in e$, equality holds in \eqref{eq:Del-tilde-f-tilde-f-min}. Thereby, it suffices to  consider the case that for some $j \in e$, $f(H, j) = 0$ and for some other $j' \in e$, $f(H, j') = 1$. This implies $\partial \Theta(H, j') \geq t > \Theta(H, j)$. As $\Theta$ is balanced, outside a measure zero set, we can conclude from the above that $\Theta(H, e, j') = 0$, and so its contribution in \eqref{eq:Del-tilde-f-tilde-f-min} vanishes. Since this is true for any $j'$ with $f(H, j') = 1$, both sides of the inequality $(a)$ in \eqref{eq:Del-tilde-f-tilde-f-min} become equal to zero. As this argument holds outside a measure zero set, the above discussion shows that  \eqref{eq:variational-inequality} becomes an equality when $\Theta$ is balanced. Moreover, the function $f = \oneu{\partial \Theta > t}$ achieves the supremum. 

Now, we show that if \eqref{eq:variational-inequality} is an equality for all $t \in \reals$, then $\Theta$ is balanced. 
Fix some $t \in \reals$. First, we claim that the supremum on the RHS of \eqref{eq:variational-inequality} is a maximum. To see this, note that we have already shown in Section~\ref{sec:existence} that a balanced allocation $\Theta_0$ with respect to  $\mu$ exists, and the above discussion then implies that $f = \oneu{\partial \Theta_0 > t}$ achieves the equality in \eqref{eq:variational-inequality} with $\Theta$ being replaced with $\Theta_0$. But the RHS of \eqref{eq:variational-inequality} has no dependence on $\Theta$. Thereby, the RHS of \eqref{eq:variational-inequality}  always has 
$f := \oneu{\partial \Theta_0 > t}$ as a maximizer,
whenever $\mu$ is unimodular with bounded mean.
Since $f \in [0,1]$, we have $(\partial \Theta - t)^+ \geq (\partial \Theta - t) f$ pointwise. However, equality holds in \eqref{eq:variational-inequality}, so by \eqref{eq:Del-tilde-f-tilde-f-min}, we must have
\begin{equation}
  \label{eq:f=+-as}
  (\partial \Theta - t) f = (\partial \Theta - t)^+ \qquad \mu \text{--a.s.}.
\end{equation}
Using Lemma~\ref{lem:A-as-Atilde-ae}, this means that
\begin{equation}
  \label{eq:f=+-ae}
  (\partial \Theta - t) \tilde{f} = (\partial \Theta - t)^+ \qquad \vmu \text{--a.e.},
\end{equation}
where, by abuse of notation,  we have treated 
$\partial \Theta$ as being defined on $\mH_{**}$ via 
$\partial \Theta(H, e, i) = \partial \Theta(H, i)$. 
On the other hand, \eqref{eq:Del-tilde-f-tilde-f-min} must be an equality and 
\begin{equation}
  \label{eq:tilde-f-min-Del-tilde-f}
  \nabla (\tilde{f} \Theta) = \tilde{f}_\text{min} \qquad \vmu \text{--a.e.}.
\end{equation}
Now, we show that $\Theta$ must be balanced by checking the conditions in Defintion~\ref{def:borel-balanced-allocation}.
For the above fixed $t$, if for some $[H, e, i] \in \mH_{**}$ and some $j \in e$ we have
\begin{equation*}
  \partial \Theta(H, e, i) > t > \partial \Theta(H, e, j),
\end{equation*}
then \eqref{eq:f=+-ae} implies that outside a measure zero set, we can conclude that $f(H, i) = 1$ and $f(H, j) = 0$. This can be seen by comparing the 
left hand side and right hand side of \eqref{eq:f=+-ae} in each of the
two cases, and then recalling that $\tilde{f}(H,e,i) = f(H,i)$.

Hence, by definition, $\tilde{f}_\text{min}(H,e , i) = 0$. Therefore, \eqref{eq:tilde-f-min-Del-tilde-f} implies that outside a measure zero set, we have 
\begin{equation*}
  \begin{split}
    0 = \nabla(\tilde{f} \Theta)(H, e, i) &= \frac{1}{|e|} \sum_{k \in e} f(H, k) \Theta(H, e, k) \\
    &\geq \frac{1}{|e|} f(H, i) \Theta(H, e, i) \\
    &= \frac{1}{|e|} \Theta(H, e, i),
  \end{split}
\end{equation*}
which implies that $\Theta(H, e, i) = 0$.

So far we have shown that, for a fixed $t$, for almost all $[H, e, i] \in \mH_{**}$, if for some $j \in e$ we have $\Theta(H, e, i) > t > \Theta(H, e, j)$, then $\Theta(H, e, i) =0$. Since this holds for all $t \in \mathbb{Q}$, and $\mathbb{Q}$ is countable and dense in $\reals$, we can conclude that for $\vmu$--almost every $[H, e, i] \in \mH_{**}$, $\partial \Theta(H, i) > \partial \Theta(H, j)$ for some $j \in e$ implies that $\Theta(H, e, i) = 0$. Using Proposition~\ref{prop:everyting-shows-at-the-root}, we can conclude that for $\vmu$--almost all $[H,e, i] \in \mH_{**}$, $\partial \Theta(H, e, j_1) \geq \partial \Theta(H, e, j_2)$ for some $j_1, j_2 \in e$ implies that $\partial \Theta(H, j_1) = 0$. Thus, $\Theta$ is balanced. 
\end{proof}

\subsection{Optimality}
\label{sec:optimality}


In this section, we prove part 3 of Theorem~\ref{thm:balanced--properties}. We do this in three steps:

\underline{$(a) \Rightarrow (c)$} Let $\Theta$ be a  balanced allocation with respect to $\mu$. Let $\Theta'$ be any other allocation, and $f:[0,\infty) \rightarrow [0,\infty)$ be convex. We will show that
\begin{equation*}
  \int f \circ \partial \Theta d \mu \leq \int f \circ \partial \Theta' d \mu.
\end{equation*}
Using Proposition~\ref{prop:variational}, since $\Theta$ is balanced and $\Theta'$ is a Borel allocation we have, for any $t \in \reals$,
\begin{equation}
  \label{eq:theta-theta'-t+<=}
  \int (\partial \Theta - t)^+ d \mu = \sup_{\stackrel{f: \mH_* \rightarrow [0,1]}{\text{Borel}}} \int \tilde{f}_\text{min} d \vmu - t \int f d \mu \leq \int (\partial \Theta' - t)^+ d \mu.
\end{equation}
On the other hand, 
\begin{equation*}
  \int \partial \Theta d \mu = \int \Theta d \vmu = \int \nabla \Theta d \vmu = \int \frac{1}{|e|} d \vmu.
\end{equation*}
The same chain of equalities holds for $\Theta'$. Hence, $\int \partial \Theta d \vmu = \int \partial \Theta' d \vmu$. Standard results in the convex ordering of random variables (see Theorem~3.A.1 of \cite{shaked2007stochastic} for instance) show that  \eqref{eq:theta-theta'-t+<=} implies that $\Theta \preceq_{\text{cx}} \Theta'$, where $\preceq_{\text{cx}} $ denotes the partial order defined by the property that $\int f \circ \Theta d \mu\leq \int f \circ \Theta' d \mu$ for any convex function $f$. 

\underline{$(c) \Rightarrow (b)$} Just restrict to the given strictly convex function.

\underline{$(b) \Rightarrow (a)$} Assume $\Theta_0$ is a balanced allocation w.r.t. $\mu$ (from Section~\ref{sec:existence} we know it exists). As we showed above, $\int f \circ \partial \Theta_0 d \mu \leq \int f \circ \partial \Theta d \mu$ for the given strictly convex function $f:[0,\infty) \rightarrow [0,\infty)$. But $\Theta$ is a minimizer. Therefore, 
\begin{equation*}
  \int f \circ \partial \Theta d \mu = \int f \circ \partial \Theta_0 d \mu.
\end{equation*}
Now, define $\Theta': \mH_{**} \rightarrow [0,1]$ as $\Theta' = (\Theta + \Theta_0) /2$. Note that $\Theta'$ is a Borel allocation. Also, using the convexity of $f$, we have 
\begin{equation}
  \label{eq:f-Theta'<f-Theta0-f-Theta}
  \int f \circ \partial \Theta' d \mu \leq \frac{1}{2} \left ( \int f \circ \partial \Theta d \mu + \int f \circ \partial \Theta_0 d \mu \right ) = \int f \circ \partial \Theta_0 d \mu.
\end{equation}
On the other hand, since $\Theta'$ is a Borel allocation and $\Theta_0$ is balanced, the $(a) \Rightarrow (c)$ part implies that 
\begin{equation*}
  \int f \circ \partial \Theta_0 d \mu \leq \int f \circ \partial \Theta' d \mu.
\end{equation*}
Hence, we have equality in \eqref{eq:f-Theta'<f-Theta0-f-Theta}. As $f$ is nonnegative, it must be the case that 
\begin{equation*}
  f \circ \partial \Theta' = \frac{f \circ \partial \Theta_0 + f \circ \partial \Theta}{2} \qquad \mu\text{--a.s.}.
\end{equation*}
Now, since $f$ is strictly convex, this implies that $\partial \Theta = \partial \Theta_0$, $\mu$--almost surely. Therefore, for any value of $t$, we have 
\begin{equation*}
  \int (\partial \Theta - t )^+ d \mu = \int (\partial \Theta_0 - t)^+ d \mu.
\end{equation*}
Comparing this with \eqref{eq:variational-inequality}, we realize that $\Theta$ achieves the equality therein, which means, from Proposition~\ref{prop:variational}, that $\Theta$ is balanced.

\subsection{Uniqueness of $\partial \Theta$}
\label{sec:uniq-part-theta}


In this section we prove part~4 of Theorem~\ref{thm:balanced--properties}. 
If $\Theta$ is balanced, part~3 of the theorem, which was proved in the preceding section, implies that for a strictly convex function $f$, $\int f \circ \partial \Theta d \mu$ is minimized. Then, in the proof of the $(b) \Rightarrow (a)$ part of Section~\ref{sec:optimality} it was shown that $\partial \Theta = \partial \Theta_0$ $\mu$--almost surely. 
This is precisely what we want to show. 

For the converse, assume that $\partial \Theta = \partial \Theta_0$, $\mu$-almost surely. Then, for all $t$, we have
\begin{equation*}
  \int (\partial \Theta - t )^+ d \mu = \int (\partial \Theta_0 - t)^+ d \mu, 
\end{equation*}
so Proposition~\ref{prop:variational} implies that $\Theta$ is balanced,
by the same logic that was used at the end of the preceding section.

\subsection{Continuity with respect to the local weak limit}
\label{sec:cont-with-resp}


In this section we prove Part~5 of Theorem~\ref{thm:balanced--properties}. Prior to that, we need some notation and tools. 

Recall the notion of marked hypergraphs from Section~\ref{sec:mbh-mbh}.
Let $\Pi: \mbH_*(\Xi) \rightarrow \mH_*$ be the function that removes marks, i.e.
\begin{equation}
  \label{eq:projection-mbH*-mH*-2}
  \Pi([\bH, i])  = [H, i],
\end{equation}
where $H$ is the underlying hypergraph associated to $\bH$. 
It can be easily checked that
$\Pi$ is a continuous map.

Note that allocations could be considered as marks with values in $[0,1]$. 
Hence, we can capture the notion of a balanced allocation using the formalism in Section~\ref{sec:mbh-mbh}, via the following definition. 

Consider the function $f: \mbH_{**}(\Xi) \rightarrow \Xi$ defined as $f(\bH, e, i) = \xi_{\bH}(e, i)$. When $\Xi$ has an additive structure  (e.g. $\Xi = [0,1]$) we may consider $\partial f : \mbH_*(\Xi) \rightarrow \Xi$, defined as  
\begin{equation}
\label{eq:partial-zeta-function-def}
  \partial f (\bH, i) = \sum_{e\in E(\bH), e \ni i} f(\bH, e, i) = \sum_{e\in E(\bH), e \ni i} \xi_{\bH}(e, i).
\end{equation}
By abuse of notation, we may write $\partial \xi_{\bH}(i)$ instead of $\partial f(\bH, i)$ with the above $f$.

\begin{definition}
\label{def:bmu-balancedness}
A measure $\bmu \in \mP(\mbH_*([0,1]))$ is called balanced if for $\vbmu$--almost every $[\bH, e, i] \in \mbH_*([0,1])$, we have 
\begin{equation}
  \label{eq:bmu-balanced-sum-1}
  \sum_{j \in e} \xi_{\bH}(e, j) = 1,
\end{equation}
and
\begin{equation}
\label{eq:bmu-balanced-load-condition}
  \partial \xi_{\bH}(j) > \partial \xi_{\bH}(j') \quad \text{ for some } j,j' \in e \quad \Longrightarrow \quad \xi_{\bH}(e, j) = 0.
\end{equation}
\end{definition}

Before proving the result of this section, we state the following lemmas and postpone their proofs until the end of this section. 

\begin{lem}
  \label{lem:weak-limit-of-balanced-network-is-balanced}
  If $\bmu_n$ is a sequence of balanced probability measures on $\mbH_*([0,1])$ with weak limit $\bmu$, then $\bmu$ is also balanced.
\end{lem}

\begin{lem}
  \label{lem:theta-theta'-balanced-sum-theta-theta'-loads-non-positive}
  Assume $\theta_1, \dots, \theta_n$ and $\theta'_1, \dots, \theta'_n$ are non--negative real numbers such that $\sum_i \theta_i =\sum_i \theta'_i = 1$. Also, assume that nonnegative real numbers $l_1, \dots, l_n$ and $l'_1, \dots, l'_n$ are given such that for $1 \leq i,j \leq n$, $l_i > l_j$ implies $\theta_i = 0$. Similarly, assume that $l'_i > l'_j$ implies $\theta'_i = 0$. Then, we have 
  \begin{equation*}
    \sum_{i=1}^n ( \theta_i - \theta'_i) \oneu{l_i > l'_i} \leq 0.
  \end{equation*}
\end{lem}

\begin{lem}
  \label{lem:C-compact-Cbar-compact}
  Assume $K$ is a compact subset of $\mH_*$ and $\Xi$ is a compact metric space. Define $\bar{K} \subset \mbH_*(\Xi)$ as 
  \begin{equation*}
    \bar{K} := \{ [\bH, i] \in \mbH_*(\Xi): [H, i] \in K\}.
  \end{equation*}
Then, $\bar{K}$ is compact in $\mbH_*(\Xi)$.
\end{lem}

\begin{prop}
  \label{prop:law-convergence}
  Let $\{H_n\}_{n \geq 1}$ be a sequence of finite hypergraphs with local weak limit $\mu$. Then, if $\mL_{H_n}$ denotes the distribution of balanced load on $H_n$ with a vertex chosen uniformly at random, and $\mL$ is the law of $\partial \Theta$ corresponding to the balanced allocation $\Theta$ on $\mu$, we have 
  \begin{equation*}
    \mL_{H_{n}} \Rightarrow \mL.
  \end{equation*}
\end{prop}

\begin{proof}
  Define $\theta_n$ to be a balanced allocation on $H_n$ and  $\bH_n$ to be the marked hypergraph obtained by adding $\theta_n$ to $H_n$ as edge marks, i.e. $\xi_{\bH_n}(e, i) = \theta_n(e, i)$ for $(e, i) \in \evpair(H)$.
Note that $H_n$ is a finite hypergraph; hence, $\theta_n$ is well defined and $\partial \theta_n(i)$ is unique for $ i \in V(H_n)$. 
Now, define $\bmu_n \in \mP(\mbH_*([0,1]))$ to be the distribution of $[\bH_n, v]$ where $v \in V(H_n)$ is chosen uniformly at random. We claim that $\{\bmu_n\}_{n=1}^\infty$ is a tight sequence, which means that it has a convergent subsequence. Since $\mu_n$ converges weakly to $\mu$, Prokhorov's theorem implies that $\mu_n$ is tight in $\mP(\mH_*)$ (see, for instance, \cite[Theorems 5.1 and 5.2]{billingsley2013convergence}). 
 Consequently, for $\epsilon> 0$, there is a compact set $K \subset \mH_*$ such that $\mu_n(K^c) \leq \epsilon$ for all $n$. Define
  \begin{equation*}
    \bar{K} := \{ [\bH, i] \in \mbH_*([0,1]): [H, i] \in K\}.
  \end{equation*}
From Lemma~\ref{lem:C-compact-Cbar-compact}, $\bar{K}$ is compact in $\mbH_*([0,1])$. 
It is easy to see that $\bmu_n(\bar{K}^c) = \mu_n(K^c)$ which means that $\bmu_n$ is a tight sequence. 
Hence, it has a subsequence converging weakly to some $\bmu \in \mP( \mbH_*)$. In order to simplify the notation, assume that this subsequence is the whole sequence.

With the projection map $\Pi$ defined in \eqref{eq:projection-mbH*-mH*}, define $\nu_n$ and $\nu$ to be the pushforward measures on $\mH_*$ corresponding to $\bmu_n$ and $\bmu$, respectively. As $\Pi$ removes marks, we have  $\nu_n = \mu_n$ and $\nu = \mu$. 
Now, note that $\bmu_n \Rightarrow \bmu$ implies $(\Pi)_* \bmu_n \Rightarrow (\Pi)_* \bmu$ which means that $\mu_n \Rightarrow (\Pi)_* \bmu$. On the other hand, we know $\mu_n \Rightarrow \mu$. Thereby, $(\Pi)_* \bmu = \mu$.

Note that, with the above construction, $\mL_{H_n}$ is the pushforward of $\bar{\mu}_n$ under the mapping $[\bH, i] \mapsto \partial \xi_{\bH}(i)$ defined in \eqref{eq:partial-zeta-function-def}. Therefore, as $\bar{\mu}_n \Rightarrow \bar{\mu}$, $\mL_{H_n}$ converges weakly to the law of $\partial \xi_{\bH}(i)$ under $\bar{\mu}$. Consequently, to show that $\mL_{H_n} \Rightarrow \mL$, it suffices to show that 


\begin{equation}
  \label{eq:partial-zeta-i=partial-Theta0-i}
  \partial \xi_{\bH}(i) = \partial \Theta_0([H, i]) \qquad \bmu\text{--a.s.}.
\end{equation}
From Proposition~\ref{prop:uimoularity-of-networks-weak-limit} in Appendix~\ref{sec:everything-shows-at} we know that $\bmu$ is unimodular.  Hence, we have  
\begin{equation*}
  \begin{split}
  &\int \left ( \partial \xi_{\bH}(i) - \partial \Theta_0([H, i]) \right )^+ d \bmu(\bH, i) \\
&\qquad= \int \left ( \partial_{\bH} \xi(i) - \partial \Theta_0([H, i]) \right ) \oneu{\partial \xi_{\bH}(i) > \partial \Theta_0([H, i])} d \bmu([\bH, i]) \\
  &\qquad= \int \left ( \xi_{\bH}(e, i) - \Theta_0([H,e, i]) \right ) \oneu{\partial \xi_{\bH}(i) > \partial \Theta_0([H, i])} d \vbmu([\bH,e, i]) \\
  &\qquad= \int \frac{1}{|e|} \sum_{j \in e} \left ( \xi_{\bH}(e, j) - \Theta_0([H,e,j]) \right ) \oneu{\partial \xi_{\bH}(j) > \partial \Theta_0([H, j])} d \vbmu([\bH, e, i),
  \end{split}
\end{equation*}
where the last equality follows from unimodularity of $\bmu$. 
From Lemma~\ref{lem:weak-limit-of-balanced-network-is-balanced}, $\bmu$ is balanced in the sense of Definition~\ref{def:bmu-balancedness} above.
Also, Lemma~\ref{lem:theta-theta'-balanced-sum-theta-theta'-loads-non-positive} together with the balancedness of $\bar{\mu}$ and $\Theta_0$, implies that the integrand is non--positive almost everywhere, which means that  $\partial \xi_{\bH}(i) \leq \partial \Theta_0([H, i])$ $\bmu$-almost surely. It could be proved in a similar fashion that $\partial \Theta_0(H, i) \leq \partial \xi_{\bH}(i)$ $\bmu$-almost surely. This proves \eqref{eq:partial-zeta-i=partial-Theta0-i}.

So far, we have shown that $\mL_{H_n}$ has a subsequence $\mL_{H_{n_k}}$ such that $ \mL_{H_{n_k}} \Rightarrow \mL$. This argument could be repeated for any subsequence of $H_n$, i.e. any subsequence $H_{n_k}$ has a further subsequence $H_{n_{k_l}}$ such that $\mL_{H_{n_{k_l}}} \Rightarrow \mL$. This implies that $\mL_{H_n} \Rightarrow \mL$ (see, for instance, Theorem~2.2 in \cite{billingsley1971weak}).
\end{proof}

\begin{proof}[Proof of Lemma~\ref{lem:weak-limit-of-balanced-network-is-balanced}]
First, we show that $\bmu$ satisfies \eqref{eq:bmu-balanced-load-condition}. Define 
\begin{equation*}
  A_{**} := \{ [\bH, e, i] \in \mbH_{**}([0,1]): \forall j,j' \in e, \partial \xi_{\bH}(j) > \partial \xi_{\bH}(j') \Rightarrow \xi_{\bH}(e, j) = 0 \},
\end{equation*}
and
\begin{equation*}
  A_* := \{ [\bH, i] \in \mbH_*([0,1]): \forall e \ni i, j,j' \in e, \partial \xi_{\bH}(j) > \partial \xi_{\bH}(j') \Rightarrow \xi_{\bH}(e, j) = 0 \}.
\end{equation*}
We claim balancedness of $\bmu \in \mP(\mH_*([0,1]))$ is equivalent to $\bmu(A_*) = 1$. By definition, $\bmu$ being balanced means $\bmu(A_{**}) = \bmu(\mbH_{**}([0,1]))$. This is equivalent to $\int \oneu{A_{**}} d \vbmu = \int 1 d \vbmu$, or $\int \partial \oneu{A_{**}} d \bmu = \int \deg_H(i) d \bmu$. But we have  $\partial \oneu{A_{**}} \leq \deg_H(i)$ pointwise. Therefore, $\partial \oneu{A_{**}}(H, i, \xi) = \deg_H(i)$, $\bmu$--almost surely. This is equivalent to $[\bH, e, i] \in A_{**}$ for all $e \ni i$, $\bmu$--almost surely, or $[\bH, i] \in A_*$ $\bmu$--almost surely which is in turn equivalent to $\bmu(A_*) = 1$.

Now, we claim that $A_*$ is closed in $\mbH_*([0,1])$ (with respect to the topology induced by the distance $\bar{d}_*$ defined in Section~\ref{sec:mbh-mbh}). Equivalently, we show that $A_*^c$ is open. Take an arbitrary $[\bH, i] \in A_*^c$. Since $[\bH, i] \notin A_*$, there exists $e \ni i$ and $j,j' \in e$ such that $\partial \xi_{\bH}(j) > \partial \xi_{\bH}(j')$ but $\xi_{\bH}(e, j) > 0$. If  $[\bH', i'] \in \mbH_*([0,1])$ is such that $d := \bar{d}_*([\bH, i], [\bH', i'])$ satisfies $d < 1/4$, then $[H, i] \equiv_3 [H', i']$ (recall that $H$ and $H'$ are the underlying unmarked hypergraphs associated to $\bH$ and $\bH'$, respectively). Moreover, if we define 
\begin{equation*}
  \epsilon = \min \left \{ \frac{\partial \xi_{\bH}(j) - \partial \xi_{\bH}(j')}{3} , \frac{\xi_{\bH}(e, j)}{2} \right \},
\end{equation*}
and
\begin{equation*}
\Delta = \max_{k: d_H(i, k) \leq 3} \deg_H(k),
\end{equation*}
then if
\begin{equation*}
  d < \frac{1}{1+ \frac{\Delta}{\epsilon}},
\end{equation*}
one can check that $\partial \xi_{\bH'}(\phi(j)) > \partial \xi_{\bH'}(\phi(j'))$ while $\xi_{\bH'}(\phi(e), \phi(j)) > 0$, where $\phi$ is the local isomorphism between $H$ and $H'$ following from $[H, i] \equiv_3 [H', i']$. This in particular means that $[\bH', i'] \notin A_*$. Consequently, the ball with radius $\min \{ 1/4, 1/ (1+(\Delta / \epsilon)) \}$ around $[\bH, i]$ is in $A_*^c$. This means that $A_*$ is closed.

Notice that since $\bmu_n$ is balanced, $\bmu_n(A_*) = 1$. On the other hand, as  $\bmu_n \Rightarrow \bmu$ and $A_*$ is closed, we have
\begin{equation*}
  \bmu(A_*) \geq \limsup \bmu_n(A_*) = 1,
\end{equation*}
which means that $\bmu(A_*) = 1$ and $\bmu$ is balanced.

Now we turn to showing that $\bmu$ also satisfies \eqref{eq:bmu-balanced-sum-1}. Similar to the above approach, if we define
\begin{equation*}
B_{**} := \{ [\bH, e, i] \in \mH_{**}([0,1]): \sum_{j \in e} \xi_{\bH}(e, j) = 1 \},
\end{equation*}
and 
\begin{equation*}
  B_* := \{ [\bH, i] \in \mH_*([0,1]): [\bH, e, i] \in B_{**} \, \forall e \ni i \},
\end{equation*}
we can show that for $\vbmu$--almost every $[\bH, e, i] \in \mH_{**}([0,1])$ we have  $\sum_{j \in e} \xi_{\bH}(e, j) =1$. Hence, all the conditions in Definition~\ref{def:bmu-balancedness} are satisfied and $\bmu$ is balanced. 
\end{proof}

\begin{proof}[Proof of Lemma~\ref{lem:theta-theta'-balanced-sum-theta-theta'-loads-non-positive}]
Define $L:= \min_i l_i$ and  $A := \{ 1 \leq i \leq n: l_i = L\}$. Likewise, let $L':= \min_i l'_i$ and  $A' := \{ 1 \leq i \leq n: l'_i = L'\}$. The given condition implies that $\theta_j = 0$ for $j \notin A$ and similarly $\theta'_j = 0$ for $j \notin A'$.

First assume that $L > L'$. In this case, we have
\begin{equation*}
  \begin{split}
    \sum_i ( \theta_i - \theta'_i) \oneu{l_i > l'_i} &\stackrel{(a)}{=} \sum_{i \in A \cup A'} (\theta_i - \theta'_i) \oneu{l_i > l'_i} \\
    &= \sum_{i \in A \setminus A'} ( \theta_i - \theta'_i) \oneu{l_i > l'_i} + \sum_{i \in A'} (\theta_i - \theta'_i) \oneu{l_i > l'_i} \\
    &\stackrel{(b)}{=} \sum_{i \in A \setminus A'} \theta_i \oneu{L > l'_i} + \sum_{i \in A'} (\theta_i - \theta'_i) \oneu{l_i > l'_i} \\
    &\stackrel{(c)}{=} \sum_{i \in A \setminus A'} \theta_i \oneu{L > l'_i} + \sum_{i \in A'} (\theta_i - \theta'_i) \\
    &\leq \sum_{i \in A \setminus A'} \theta_i + \sum_{i \in A'} \theta_i - \theta'_i \\
    &= 1 - 1 = 0,
  \end{split}
\end{equation*}
where $(a)$ holds since $\theta$ and $\theta'$ are zero outside $A \cup A'$, $(b)$ uses $\theta'_i = 0$ for $i \notin A'$, and $(c)$ uses the fact that for $i \in A'$, $l_i \geq L > L' = l'_i$.

Now, for the case $L \leq L'$ we have 
\begin{equation*}
  \begin{split}
    \sum_{i=1}^n (\theta_i - \theta'_i) \oneu{l_i > l'_i} &= \sum_{i \in A} (\theta_i - \theta'_i) \oneu{l_i > l'_i} + \sum_{i \in A' \setminus A} (\theta_i - \theta'_i) \oneu{l_i > l'_i} \\
    &\stackrel{(a)}{=} \sum_{i \in A' \setminus A} (\theta_i - \theta'_i) \oneu{l_i > l'_i} \\
    &\stackrel{(b)}{=} \sum_{i \in A' \setminus A} - \theta'_i \oneu{l_i > l'_i} \\
    & \leq 0,
  \end{split}
\end{equation*}
where $(a)$ follows from the fact that for $i \in A$, $l_i = L \leq L' \leq l'_i$ and $(b)$ employs $\theta_i = 0$ for $i \notin A$. This completes the proof.
\end{proof}


\begin{proof}[Proof of Lemma~\ref{lem:C-compact-Cbar-compact}]
We take a sequence $[\bH_n, i_n]$ in $\bar{K}$ and try to find a converging subsequence. Notice that $[H_n, i_n]$ is a sequence in $K$. Therefore, it  has  a convergent subsequence. Without loss of generality, assume this subsequence is the whole sequence, converging to some $[H, i] \in \mH_*$. By reducing to a further subsequence, we may assume that $d_{\mH_*}([H_n, i_n], [H, i]) \leq 1/(n+1)$. This means that,  for all $m\geq n$, we have  $(H_m, i_m) \equiv_n (H,i)$. Since $(H,i)$ is locally finite,  there are finitely many marks up to level  $n$ in $(H, i)$ which all have values in the compact space $\Xi$. Hence, there is a subsequence where the marks in the first $n$ levels in $\bH_n$ are convergent. With this, we may associate marks to $(H, i)_n$ using the limiting values. Since this is true for all $n$, via a diagonalization argument we may construct a marked rooted hypergraph $(\bH, i)$ together with a subsequence $(\bH_{m_l}, i_{m_l})$ such that the underlying unmarked hypergraph is identical to $H$ and also, for any integer $k$, the marks up to depth $k$ in $(\bH_{m_l}, i_{m_l})$ converge to those in $(\bH, i)$ as $m\rightarrow \infty$. This means that $[\bH_{m_l}, i_{m_l}] \rightarrow [\bH, i]$ and completes the proof. 
\end{proof}


\section{Response Functions}
\label{sec:response-functions}

Let $H$ be a fixed hypergraph (not necessary bounded) and $i \in V(H)$ be a vertex in $H$. Fix $\epsilon > 0$. The response function $\rf_{(H, i)}^\epsilon: \reals \rightarrow \reals$ is defined as follows: $\rf_{(H, i)}^\epsilon(t)$
is the total load at node $i$ corresponding to the canonical $\epsilon$--balanced allocation with respect to an external load with value $t$ at node $i$ (recall the definition of canonical $\epsilon$--balanced allocations from Section~\ref{sec:epsil-balanc-alloc}). 
More precisely, given $t \in \reals$, let $b_{t,i}: V(H) \rightarrow \reals$ be the baseload function such that $b_{t,i}(i) = t$ and $b_{t,i}(j) = 0$ for $j \neq i$. Moreover, let $\theta_\epsilon^{b_{t,i}}$ be the canonical $\epsilon$--balanced allocation on $H$ with respect to the baseload $b_{t,i}$, as was defined in Section~\ref{sec:epsil-balanc-alloc}. We then define
\begin{equation*}
  \rf_{(H, i)}^\epsilon(t) := \partial_{b_{t,i}} \theta_\epsilon^{b_{t,i}}(i).
\end{equation*}

It turns out that this function has the following properties:

\begin{prop}
  \label{prop:response-non-expansion}
  Given a  vertex rooted hypergraph $(H,i)$, for any $\epsilon>0$ and $x < y$, we have 
  \begin{equation}
    \label{eq:f-non-expansion}
    0 \leq \rf^\epsilon_{(H, i)}(y) - \rf^\epsilon_{(H, i)}(x) \leq y - x.
  \end{equation}
Also, 
  \begin{equation}
    \label{eq:f-deg}
    0 \leq \rf^\epsilon_{(H, i)}(x) - x \leq \deg_H(i).
  \end{equation}
\end{prop}

\begin{proof}
  Let 
  $\theta_\epsilon^{b_{x,i}}$
  be the canonical $\epsilon$--balanced allocation with respect to the baseload $b_{x,i}$, as defined above.
  Then, by definition, $\rf^\epsilon_{(H, i)}(x) = \partial_{b_{x,i}}\theta_\epsilon^{b_{x,i}}(i)$.
Let $b_{y,i}$ and $\theta_\epsilon^{b_{y,i}}$ be defined similarly, with $x$ being replaced by $y$.
As $x < y$, Proposition~\ref{prop:monotonicity-canonical-eba} implies that $\partial_{b_{x,i}}\theta_\epsilon^{b_{x,i}}(i) \leq \partial_{b_{y,i}}\theta_\epsilon^{b_{y,i}}(i)$, which means  $\rf^\epsilon_{(H, i)}(x) \leq \rf^\epsilon_{(H, i)}(y)$. 

In order to show that $\rf^\epsilon_{(H, i)}(y) - \rf^\epsilon_{(H, i)}(x) \leq y - x$, let $d_{x,i}$ be the baseload function such that $d_{x,i}(i) = 0$ and $d_{x,i}(j) = -x$ for $j\in V(H)$, $j \neq i$. Moreover, let $d_{y,i}$ be the baseload function such that $d_{y,i}(i) = 0$ and $d_{y,i}(j) = -y$ for $j \in V(H)$, $j\neq i$. 
Let 
$\theta_\epsilon^{b_{x,i},\Delta}$
and 
$\theta_\epsilon^{b_{y,i},\Delta}$
be the canonical $\epsilon$--balanced allocations on $H^\Delta$ with respect to $b_{x,i}$ and $b_{y,i}$, respectively. It is easy to see that, on $H^\Delta$, $\theta_\epsilon^{b_{x,i},\Delta}$ and $\theta_\epsilon^{b_{y,i},\Delta}$ are $\epsilon$--balanced with respect to $d_{x,i}$ and $d_{y,i}$ as well, respectively. This is because $b_{x,i}(j) = d_{x,i}(j) + x$ and $b_{y,i}(j) = d_{y,i}(j) + y$, for all $j \in V(H)$. As $d_{x,i}(i) = 0$ and $b_{x,i}(i) = x$, we have 
\begin{equation*}
  \partial_{d_{x,i}} \theta_\epsilon^{b_{x,i},\Delta}(i) = \partial_{b_{x,i}} \theta_\epsilon^{b_{x,i},\Delta}(i) - x.
\end{equation*}
Likewise, $  \partial_{d_{y,i}} \theta_\epsilon^{b_{y,i},\Delta}(i) = \partial_{b_{y,i}} \theta_\epsilon^{b_{y,i},\Delta}(i) - y$. 
On the other hand, as $d_{x,i}(j) \geq d_{y,i}(j)$ for all $j \in V(H)$, using Proposition~\ref{prop:bounded-partial-b-increasing}, $\partial_{d_{x,i}} \theta_\epsilon^{b_{x,i},\Delta}(i) \geq \partial_{d_{y,i}} \theta_\epsilon^{b_{y,i},\Delta}(i)$. Comparing with the above, this means that 
\begin{equation*}
  \partial_{b_{x,i}} \theta_\epsilon^{b_{x,i},\Delta}(i) - x \geq \partial_{b_{y,i}} \theta_\epsilon^{b_{y,i},\Delta}(i) - y.
\end{equation*}
Sending $\Delta \rightarrow \infty$, we have $\partial_{b_{y,i}}\theta_\epsilon^{b_{y,i}}(i) - \partial_{b_{x,i}} \theta_\epsilon^{b_{x,i}} \leq y - x$. This means that $\rho_{(H, i)}^\epsilon(y) - \rho_{(H,i)}^\epsilon(x) \leq y- x$.


To show \eqref{eq:f-deg}, simply note that $\rf^\epsilon_{(H, i)}(x) - x$ is $\partial \theta_\epsilon^{b_{x,i}}(i)$, which is the sum of $\deg_H(i)$ many numbers in the interval $[0,1]$, hence is in the interval $[0,\deg_H(i)]$.
\end{proof}


In view of Proposition \ref{prop:response-non-expansion}, it is
convenient to define the following class of functions:

\begin{definition}
\label{def:Upsilon-functions}
  For $C > 0$, let $\Upsilon(C)$ denote the class of functions $g: \reals \rightarrow \reals$ satisfying the following conditions:
  \begin{enumerate}[label=$(\roman{*})$]
  \item $g$ is nondecreasing.
  \item $g$ is nonexpansive, i.e. for $x \leq y$ we have $g(y) - g(x) \leq y - x$.
  \item $0 \leq g(x) - x \leq C$.
  \end{enumerate}
\end{definition}

In this notation, Proposition~\ref{prop:response-non-expansion} implies that for any vertex rooted hypergraph $(H, i)$ and $\epsilon>0$, 
\begin{equation*}
 \rf_{(H, i)}^\epsilon(.) \in \Upsilon(\deg_H(i)).
\end{equation*}
It is easy to check the following properties of this class of functions:

\begin{lem}
\label{lem:Upsilon-properties}
For any function $g \in \Upsilon(C)$ we have
\begin{enumerate}[label=$(\roman{*})$]
\item $g$ is continuous.
\item $\lim_{x \rightarrow \pm \infty} g(x) = \pm \infty$.
\item For any $t \in \reals$, the set $\{ x \in \reals: g(x) =t \}$ is a nonempty bounded closed interval in $\reals$.
\item The function $g^{-1}: \reals \rightarrow \reals$ defined as
  \begin{equation}
    \label{eq:non-decreasing-non-expansive-inverse}
    g^{-1}(t) := \max \{ x \in \reals: g(x) = t \},
  \end{equation}
  is well defined,  nondecreasing and right--continuous. Moreover, if $D$ denotes its set of discontinuities of $g^{-1}$, $D$ is countable. Furthermore,  the set $\{x: g(x) = t\}$ has exactly one element iff $t \notin D$.
\item Let $D$ be the set of discontinuities of $g^{-1}$ and $t \notin D$. If for some $a$,  $a < g^{-1}(t)$, then we have $g(a) < t$. Moreover, if for some $a$ we have $a > g^{-1}(t)$, then we have $g(a) > t$. 
\end{enumerate}
\end{lem}

\subsection{Recursion for Response functions on Hypertrees}
\label{sec:recurs-resp-funct}

Recall from Section~\ref{sec:regul-perop-stand} that if $T$ is a hypertree, then for $(e, i) \in \evpair(T)$, $T_{e \rightarrow i}$ is obtained by removing $e$ from $T$ and looking at the connected component of the resulting hypertree rooted at $i$.


Given $x \in \reals$, let $b_x$ be the baseload function on $T_{e \rightarrow i}$ with $b_x(i) = x$ and $b_x(j) = 0$, $j \neq i$. With this, let $\theta_\epsilon^{b_x}$ be the canonical $\epsilon$--balanced allocation on $T_{e\rightarrow i}$ with respect to $b_x$. As was discussed above, $\rf_{T_{e \rightarrow i}}^\epsilon(x) = x + \partial \theta_\epsilon^{b_x}(i)$. 

Now we attempt to find some recursive expressions for such response functions.
Before that, we need the following general lemma, whose proof is postponed to the end of this section. 

\begin{lem}
  \label{lem:venkat-Gibbs-epsilon-uniqueness}
Given $C >0$ and a collection of nondecreasing functions $g_i : [0,1] \rightarrow \reals$, $1 \leq i \leq n$,  the set of fixed point equations 
  \begin{equation}
    \label{eq:theta-i-g-i-fixed-point}
    \theta_i = \frac{e^{- g_i(\theta_i)}}{\sum_{j=1}^n e^{-g_j(\theta_j)} + C} \qquad 1 \leq i \leq n,
  \end{equation}
has a unique solution $(\theta_1, \dots, \theta_n) \in [0,1]^n$.
\end{lem}

\edit
\begin{prop}
\label{prop:T-unbounded-f-epsilon-inverse}
  Assume $T$ is a locally finite hypertree (not necessarily bounded), $\epsilon> 0$, and $(e, i) \in \evpair(T)$. Then, $\rft{e}{i}^\epsilon(.)$ is an invertible function, and for $t \in \reals$, we have
  \begin{equation}
    \label{eq:epsilon-rft-inverse-t-sum}
    \left ( \rft{e}{i}^\epsilon \right )^{-1}(t) = t - \sum_{e' \ni i: e' \neq e} \left ( 1 - \sum_{j \in e', j \neq i} \zeta^\epsilon_{e', j} \right ),
  \end{equation}
  where $\{ \zeta^\epsilon_{e', j} \}$ for $e' \ni i$, $e' \neq e$, $j \in e'$, $j \neq i$ are the unique solutions to the set of equations 
  \begin{equation}
    \label{eq:theta-epsilon-e'-j-fixed-point}
    \zeta^\epsilon_{e',j} = \frac{ \exp \left ( - \frac{\rft{e'}{j}^\epsilon(\zeta^\epsilon_{e', j})}{\epsilon} \right )}{e^{- t / \epsilon} + \sum_{l \in e', l \neq i} \exp \left ( - \frac{\rft{e'}{l}^\epsilon(\zeta^\epsilon_{e', l})}{\epsilon} \right ) }.
  \end{equation}
\end{prop}

\edit
\begin{proof}
  Note that Lemma~\ref{lem:Upsilon-properties} implies that the set
$    A(t):= \{ x \in \reals: \rft{e}{i}^\epsilon(x) = t \}$
is not empty. Hence, in order to show that $\rft{e}{i}^\epsilon(.)$ is invertible, we should show that $A(t)$ is a singleton for all $t \in \reals$. By definition,  $x \in A(t)$ means that 
we have 
\begin{equation}
\label{eq:t=x+sum-f-inverse}
  t = x + \sum_{e' \ni i, e' \neq e} \left ( 1 - \sum_{j \in e', j \neq i} \theta_\epsilon^{b_x}(e' ,j) \right ).
\end{equation}
For each $(e',j)$ in the above summation, as $T_{e' \rightarrow j}$ is a subtree of $T_{e\rightarrow i}$, we can treat $\theta_\epsilon^{b_x}$ as an allocation on $T_{e'\rightarrow j}$ via the identity projection. With this, Proposition~\ref{prop:regularity-canonical-e-ba} implies that $\theta_\epsilon^{b_x}$ is the canonical $\epsilon$--balanced allocation on $T_{e' \rightarrow j}$ with respect to the baseload function that evaluates to $\theta^\epsilon(e' ,j)$ at $j$ and zero elsewhere. Thereby, by the definition of the response function, for each such pair $(e', j)$, we have $\partial \theta_\epsilon^{b_x}(j) = \rft{e'}{j}^\epsilon(\theta^\epsilon(e' ,j))$. Consequently,
\begin{equation*}
  \theta_\epsilon^{b_x}(e',j) = \frac{ \exp \left ( - \frac{\rft{e'}{j}^\epsilon(\theta_\epsilon^{b_x}(e',j))}{\epsilon} \right )}{e^{-t/\epsilon} + \sum_{l \in e', l \neq i} \exp \left ( - \frac{\rft{e'}{l}^\epsilon(\theta_\epsilon^{b_x}(e' ,l))}{\epsilon} \right ) }.
\end{equation*}
Lemma~\ref{lem:venkat-Gibbs-epsilon-uniqueness} guarantees that for each $e' \ni i$, $e' \neq e$, there is a unique set of solutions to these equations. 
From \eqref{eq:t=x+sum-f-inverse}, we realize that for any $t \in \reals$, there is a unique solution for $x$. 
This implies the invertibility of $\rft{e}{i}^\epsilon$. 
Rearranging  \eqref{eq:t=x+sum-f-inverse}, we get \eqref{eq:epsilon-rft-inverse-t-sum}, with $\zeta_{e', j} = \theta_\epsilon^{b_x}(e' ,j)$.
\end{proof}

Now, we will send $\epsilon$ to zero in the above Proposition and show that, under some conditions, the sequence of response functions converges pointwise to a limit which satisfies a certain fixed point equation.
To do so, we need the following lemma, whose proof is given at the end of this section.
\begin{lem}
  \label{lem:limit-function-inverse}
  Assume $\{ g_n \}_{n = 1}^\infty$ is a sequence of functions in $\Upsilon(C)$ that converge pointwise to a function $g$, i.e. $g(x) = \lim_{n \rightarrow \infty} g_n(x)$ for all $x \in \reals$. Then, $g$ is also in $\Upsilon(C)$. Furthermore, if $D_n$ denotes the set of discontinuities of $g_n$ and $D$ denotes the set of discontinuities of $g$, then for $t \notin (\bigcup_{n} D_n ) \cup D$ we have
  \begin{equation}
\label{eq:lim-gn-1--g-1}
    \lim_{n \rightarrow \infty} g_n^{-1}(t) = g^{-1}(t).
  \end{equation}
\end{lem}

Now, we can write recursive equations for the limit of $\rft{e}{i}^\epsilon$, if it exists. 

\begin{prop}
  \label{prop:f-epsilon-converges-f-recursive-hypertree}
  Assume $T$ is a locally finite hypertree (not necessarily bounded) and 
$\epsilon_n$ is a sequence of positive numbers converging to zero, with $\theta_{\epsilon_n}$ being the canonical  $\epsilon_n$--balanced allocation on 
$T$.
If $l_i := \lim_{n \rightarrow \infty} \partial \theta_{\epsilon_n}(i)$ exists for all $i \in V(T)$, then, for all $(e, i) \in \evpair(T)$, $\rft{e}{i}^{\epsilon_n}(.)$ converges pointwise to some $\rft{e}{i}(.) \in \Upsilon(\deg_T(i) - 1)$. Moreover, for all $t \in \reals$, we have 
  \begin{equation}
    \label{eq:f-inverse-recursive}
    \rft{e}{i}^{-1}(t) = t - \sum_{e' \ni i: e' \neq e} \left [1 - \sum_{j \in e', j \neq i} \left ( \rft{e'}{j}^{-1}(t) \right)^+ \right ]_0^1,
  \end{equation}
  where the inverse functions are defined as in \eqref{eq:non-decreasing-non-expansive-inverse}. Furthermore, for a node $i \in V(T)$ we have
  \begin{equation}
    \label{eq:l_i>t-iff-sum-f}
    l_i > t \qquad \Longleftrightarrow \qquad \sum_{e \ni i} \left [ 1 - \sum_{j \in e, j \neq i} \left ( \rf^{-1}_{T_{e \rightarrow j}}(t) \right )^+ \right ]_0^1 > t.
  \end{equation}
\end{prop}

\begin{proof}
First we fix $x \in \reals$ and $(e, i) \in \evpair(T)$ and show that, as $n\rightarrow \infty$, $\rf_{T_{e \rightarrow i}}^{\epsilon_n}(x)$ is convergent. We call the limit $\rf_{T_{e \rightarrow i}}(x)$. Let $\theta^{b_x}_{\epsilon_n}$ denote the canonical $\epsilon_n$--balanced allocation on $T_{e \rightarrow i}$ with baseload $b_x$, which equals
$x$ at $i$ and is zero elsewhere. By definition,  $\rf^{\epsilon_n}_{T_{e \rightarrow i}}(x) = x + \partial \theta^{b_x}_{\epsilon_n}(i)$. Thus, it suffices to show that $\partial \theta^{b_x}_{\epsilon_n}(i)$ is convergent. Note that $\theta^{b_x}_{\epsilon_n}$ is a sequence in the compact space $[0,1]^{\evpair(T_{e\rightarrow i})}$ equipped with the product topology. On the other hand, $\partial \theta^{b_x}_{\epsilon_n}(i)$ is a bounded sequence depending only on a finite number of  coordinates, namely $\deg_T(i) - 1$. As a result, in order to show that $\partial \theta^{b_x}_{\epsilon_n}(i)$ is convergent, it suffices to show that if $\theta_1$ and $\theta_2$ are two subsequential limits of $\theta^{b_x}_{\epsilon_n}$ in $[0,1]^{\evpair(T_{e \rightarrow i})}$, then $\pt_1(i) = \pt_2(i)$. Passing to the limit in \eqref{eq:epsilon-baseload-definition}, we realize that both $\theta_1$ and $\theta_2$ are balanced allocations on $T_{e \rightarrow i}$ with respect to the baseload $b_x$.
Using Proposition~\ref{prop:weak-uniqueness}, it suffices to show that $\norm{\pt_1 - \pt_2}_{l^1(V(T_{e \rightarrow i}))} < \infty$. 
From Proposition~\ref{prop:regularity-canonical-e-ba}, we know that the restriction of $\theta_{\epsilon_n}$ to $T_{e \rightarrow i}$ is the canonical $\epsilon_n$--balanced allocation with baseload $\theta_{\epsilon_n}(e ,i)$ at $i$ and zero elsewhere. Hence, if $K$ is a finite subset of $V(T_{e \rightarrow i}) \setminus \{i\}$, using Proposition~\ref{prop:non-expansivitiy}, we have 
\begin{equation*}
  \sum_{j \in K} | \partial \theta^{b_x}_{\epsilon_n}(j) - \partial \theta_{\epsilon_n}(j)| \leq |x| + \theta_{\epsilon_n}(e ,i) \leq |x| + 1.
\end{equation*}
Using the triangle inequality, for integers $n$ and $m$,
\begin{equation*}
  \sum_{j \in K} | \partial \theta^{b_x}_{\epsilon_n}(j) - \partial \theta^{b_x}_{\epsilon_m}(j) |\leq 2(|x| + 1) + \sum_{j \in K} | \partial \theta_{\epsilon_n}(j) - \partial \theta_{\epsilon_m}(j)|.
\end{equation*}
Now, send $n$ to infinity along the subsequence of $\theta^{b_x}_{\epsilon_n}$ that converges to $\theta_1$. Likewise, send $m$ to infinity along the subsequence of $\theta^{b_x}_{\epsilon_n}$ that converges to $\theta_2$. Using the assumption
that $\partial \theta_{\epsilon_n}(j)$ is convergent,
and since $K$ is finite, we get 
\begin{equation*}
  \sum_{j \in K} | \pt_1(j) - \pt_2(j) | \leq 2 (|x| + 1).
\end{equation*}
Since $K$ is arbitrary, sending $K$ to $V(T_{e\rightarrow i}) \setminus \{ i\}$ we get
\begin{equation*}
\begin{split}
  \sum_{j \in V(T_{e \rightarrow i})} | \pt_1(j) - \pt_2(j) | &= |\pt_1(i) - \pt_2(i)|+ \\
&\qquad \sum_{j \in V(T_{e \rightarrow i}) \setminus \{i\}} | \pt_1(j) - \pt_2(j) |  \\
&\leq 2( |x| + 1 ) + 2 \deg_T(i) < \infty,
\end{split}
\end{equation*}
which means $\norm{\pt_1 - \pt_2}_{l^1(V(T_{e \rightarrow i}))} < \infty$. This means that for all $x \in \reals$ and $(e, i) \in \evpair(T)$, $\rf_{T_{e \rightarrow i}}^{\epsilon_n}(x) \rightarrow \rf_{T_{e \rightarrow i}}(x)$.
As $\rf_{T_{e \rightarrow i}}^{\epsilon_n}(.) \in \Upsilon(\deg_T(i) - 1)$, this in particular implies that $\rf_{T_{e \rightarrow i}}(.) \in \Upsilon(\deg_T(i) - 1)$.

Now, we prove \eqref{eq:f-inverse-recursive}.
Using part $(iv)$ of Lemma~\ref{lem:Upsilon-properties}, both sides of \eqref{eq:f-inverse-recursive} are right continuous functions of $t$. Hence, if $D$ denotes the union of the discontinuity sets of $\rft{e}{i}^{-1}$ and $\rft{e'}{j}^{-1}$ for $e' \ni i$ and $j \in e'$, $D$ is countable and hence $D^c$ is dense in $\reals$. Thus, due to the right continuity, it suffices to show \eqref{eq:f-inverse-recursive} for $t \notin D$.

Now, take some $t \in D^c$. Using Lemma~\ref{lem:limit-function-inverse} we have 
\begin{equation*}
  \rft{e}{i}^{-1}(t) = \lim_{n \rightarrow \infty} \left ( \rft{e}{i}^{\epsilon_n} \right)^{-1} (t),
\end{equation*}
for $t$ in a dense subset of $D^c$. By abuse of notation, we continue
to denote this subset by $D^c$.
Using Proposition~\ref{prop:T-unbounded-f-epsilon-inverse}, $(\rft{e}{i}^{\epsilon_n})^{-1}$ is expressed in terms of $\zeta^{\epsilon_n}_{e', j}$ which are solutions to \eqref{eq:theta-epsilon-e'-j-fixed-point}. Comparing this to \eqref{eq:f-inverse-recursive}, it suffices to prove that for each $e' \ni i, e' \neq e$, we have
\begin{equation}
  \label{eq:theta-epsilon->zero-trunaction-response}
   \lim_{n \rightarrow \infty} 1 - \sum_{j \in e', j \neq i} \zeta^{\epsilon_n}_{e', j} = \left [ 1 - \sum_{j \in e', j\neq i} \left ( \rft{e'}{j}^{-1}(t) \right )^+ \right ]_0^1.
\end{equation}
Fixing such an $e'$, since for each $j \in e'$, $j \neq i$, the sequence $\{\zeta^{\epsilon_n}_{e' , j}\}_{n=1}^\infty$ is in the compact set $[0,1]$, it suffices to show that if for all $j \in e'$, $j \neq i$, there is a subsequence  $\zeta^{\epsilon_{n_k}}_{e', j} $ converging to some $\zeta^*_{e', j}$, then
\begin{equation}
  \label{eq:theta-epsilon-*-truncation-response}
  1 - \sum_{j \in e', j \neq i} \zeta^*_{e', j} = \left [ 1 - \sum_{j \in e', j\neq i} \left ( \rft{e'}{j}^{-1}(t) \right )^+ \right ]_0^1.
\end{equation}
Without loss of generality and in order to simplify the notation, we may assume that the subsequence is the whole sequence, i.e. for all $j \in e'$, $j \neq i$,  $\zeta^{\epsilon_n}_{e' , j} \rightarrow \zeta^*_{e' , j}$. We show \eqref{eq:theta-epsilon-*-truncation-response} in different cases:

\textbf{Case I}: $\sum_{j \in e', j\neq i} \left ( \rft{e'}{j}^{-1}(t) \right )^+ \leq 1$: since $\sum_{j \in e', j\neq i} \left ( \rft{e'}{j}^{-1}(t) \right )^+$ is nonnegative, it suffices to show that $\zeta^*_{e', j} = \left ( \rft{e'}{j}^{-1}(t) \right)^+$ for each $j \in e'$, $j \neq i$. For such a $j$, we do this in two subcases:

\textbf{Case Ia}: First, assume $\rft{e'}{j}^{-1}(t) \leq 0$, in which case we should show $\zeta^*_{e', j} = 0$. If this is not the case, as $\zeta^*_{e',j} \in [0,1]$, there must be the case that $\zeta^*_{e',j} > 0 \geq \rft{e'}{j}^{-1}(t)$. Then, part $(v)$ of Lemma~\ref{lem:Upsilon-properties} implies that $\rft{e'}{j}(\zeta^*_{e', j}) \geq t + \delta$ for some $\delta > 0$. With this,
\begin{align*}
    \rft{e'}{j}^{\epsilon_n}(\zeta^{\epsilon_n}_{e', j}) - t &= \rft{e'}{j}^{\epsilon_n}(\zeta^{\epsilon_n}_{e', j}) - \rft{e'}{j}^{\epsilon_n}(\zeta^*_{e', j}) \\
    & \qquad + \rft{e'}{j}^{\epsilon_n}(\zeta^*_{e', j}) - \rft{e'}{j}(\zeta^*_{e', j}) \\
    & \qquad + \rft{e'}{j}(\zeta^*_{e', j}) - t \\
    & \geq - |\zeta^{\epsilon_n}_{e',j} - \zeta^*_{e', j} | \\
    & \qquad + \rft{e'}{j}^{\epsilon_n}(\zeta^*_{e', j}) - \rft{e'}{j}(\zeta^*_{e', j}) \\ 
    & \qquad + \rft{e'}{j}(\zeta^*_{e' ,j}) - t.
\end{align*}
Note that the first two terms converge to zero as $n\rightarrow \infty$ and hence they could be made smaller than $\delta/3$ by choosing $n$ large enough. Thus $\rft{e'}{j}^{\epsilon_n}(\zeta^{\epsilon_n}_{e', j}) - t \geq \delta / 3$ for $n$ large enough. On the other hand
\begin{equation*}
  \begin{split}
    \zeta^{\epsilon_n}_{e',j} &\leq \frac{1}{1 + \exp \left ( - \frac{t - \rft{e'}{j}^{\epsilon_n}(\zeta^{\epsilon_n}_{e', j})}{\epsilon_n} \right )} \\
    &\leq \frac{1}{1 + e^{ \delta / (3 \epsilon_n)}}.
  \end{split}
\end{equation*}
Sending $n$ to infinity, since $\delta$ is fixed,  we realize that $\zeta^{\epsilon_n}_{e' , j}$ converges to zero, which is a contradiction with the assumption that $\zeta^*_{e', j} > 0$.

\textbf{Case Ib}: Now consider the case $\rft{e'}{j}^{-1}(t) > 0$. If $\zeta^*_{e', j} > \rft{e'}{j}^{-1}(t)$, following a similar argument as  in case Ia, we can conclude that $\zeta^*_{e', j} = 0$, which is a contradiction. Hence, assume $\zeta^*_{e' , j} < \rft{e'}{j}^{-1}(t)$. Since $t \in D^c$ is a continuity point of $\rft{e'}{j}^{-1}$, using Lemma~\ref{lem:Upsilon-properties} part $(v)$, we have $\rft{e'}{j}(\zeta^*_{e' , j}) \leq t - \delta $ for some $\delta > 0$. An argument similar to that in case Ia implies $\rft{e'}{j}^{\epsilon_n}(\zeta^{\epsilon_n}_{e', j}) \leq t - \delta / 3$ for $n$ large enough. Now
\begin{align*}
    1 - \sum_{l \in e', l \neq i} \zeta^*_{e' , l} &= \lim_{n \rightarrow \infty} 1 - \sum_{l \in e', l\neq i} \zeta^{\epsilon_n}_{e' , l} \\
    &= \lim_{n \rightarrow \infty} \frac{1}{1 + \sum_{l \in e', l \neq i} \exp \left ( - \frac{ \rft{e'}{l}^{\epsilon_n}(\zeta^{\epsilon_n}_{e' , l}) - t}{\epsilon_n} \right ) } \\
    &\leq \lim_{n \rightarrow \infty} \frac{1}{1 + \exp \left ( - \frac{ \rft{e'}{j}^{\epsilon_n}(\zeta^{\epsilon_n}_{e' , j}) - t}{\epsilon_n} \right ) } \\
    &\leq \frac{1}{1 + e^{\delta / (3\epsilon_n)}}.
\end{align*}
Since $\delta$ is fixed, sending $n$ to infinity we realize that $\sum_{l \in e', l \neq i} \zeta^*_{e' , l} = 1$. This, together with our earlier assumption of Case I, means that 
\begin{equation*}
  \sum_{l \in e', l \neq i} \zeta^*_{e' , l}  = 1 \geq \sum_{l \in e', l \neq i} \left ( \rft{e'}{l}^{-1}(t) \right )^+ = \sum_{l \in e', l \neq i, \rft{e'}{l}^{-1}(t) > 0} \rft{e'}{l}^{-1}(t).
\end{equation*}
Since we have assumed that $\zeta^*_{e', j} < \rft{e'}{j}^{-1}(t)$,
this means there exists some $j' \in e'$, $j' \neq i$ such that $\rft{e'}{j'}^{-1}(t) > 0$ and $\zeta^{*}_{e' , j'} > \rft{e'}{j}^{-1}(t) > 0$. This, as we discussed above, results in $\zeta^*_{e' , j'} = 0$, which is a contradiction. Hence, $\zeta^*_{e' , j}$ should be equal to $\rft{e'}{j}^{-1}(t)$ and the proof of this case is complete. 

\textbf{Case II}: $\sum_{j \in e', j\neq i} \left ( \rft{e'}{j}^{-1}(t) \right )^+ > 1$, in which case we need to show that $1-\sum_{j \in e', j \neq i} \zeta^*_{e' , j} = 0$ to conclude \eqref{eq:theta-epsilon-*-truncation-response}.  Since 
\begin{align*}
\sum_{j \in e', j \neq i} \zeta^*_{e' , j} &= \lim_{n \rightarrow \infty} \sum_{j \in e', j \neq i} \zeta^{\epsilon_n}_{e',j} \\
&= \lim_{n \rightarrow \infty} \frac{\sum_{l \in e', l \neq i} \exp  ( - \rft{e'}{l}^{\epsilon_n}(\zeta^{\epsilon_n}_{e', l})/\epsilon  )}{e^{-t/\epsilon} + \sum_{l \in e', l \neq i} \exp  ( - \rft{e'}{l}^{\epsilon_n}(\zeta^{\epsilon_n}_{e', l})/\epsilon  )} \\
&\leq 1 < \sum_{j \in e', j\neq i} \left ( \rft{e'}{j}^{-1}(t) \right )^+,
\end{align*}
there should exist some $j^* \in e', j^* \neq i$ such that $\zeta^*_{e', j^*} < \rft{e'}{j^*}^{-1}(t)$ and $\rft{e'}{j^*}^{-1} > 0$. Since $t \notin D$, using Lemma~\ref{lem:Upsilon-properties} we have $\rft{e'}{j^*}(\zeta^*_{e', j^*}) \leq t - \delta$ for some $\delta > 0$. Using calculations similar to those in case Ia above, for $n$ large enough we have $\rft{e'}{j^*}^{\epsilon_n}(\zeta^{\epsilon_n}_{e' , j^*}) \leq t - \delta / 3$. Now, 
\begin{equation*}
  \begin{split}
    1 - \sum_{j \in e', j \neq i} \zeta^{\epsilon_n}_{e' , j} &= \frac{1}{1 + \sum_{j \in e', j \neq i} \exp \left ( - \frac{\rft{e'}{j}^{\epsilon_n}(\zeta^{\epsilon_n}_{e' , j}) - t}{\epsilon_n} \right ) } \\
    &\leq \frac{1}{1 +  \exp \left ( - \frac{\rft{e'}{j^*}^{\epsilon_n}(\zeta^{\epsilon_n}_{e' , j^*}) - t}{\epsilon_n} \right ) } \\
    &\leq \frac{1}{1 + e^{\delta / (3 \epsilon_n)}}.
  \end{split}
\end{equation*}
Since $\delta$ is fixed, sending $n$ to infinity we conclude that $1 - \sum_{j \in e', j \neq i} \zeta^*_{e' , j} = 0$, which completes the argument of this case.  

Having verified \eqref{eq:theta-epsilon-*-truncation-response} in all cases, we conclude \eqref{eq:f-inverse-recursive}.
Now, we prove \eqref{eq:l_i>t-iff-sum-f}. Note that in the above discussion we started with the rooted tree $T_{e \rightarrow i}$. However, it can be verified that all the arguments are valid if we start with the tree $T$ rooted at an arbitrary vertex $i \in V(T)$, i.e. $(T, i)$. Convergence of $\rf_{(T, i)}$ is also similar. In this case, repeating the above argument, \eqref{eq:f-inverse-recursive} becomes
 \begin{equation}
    \label{eq:f-inverse-recursive-T-i}
    \rf_{(T, i)}^{-1}(t) = t - \sum_{e \ni i} \left [1 - \sum_{j \in e, j \neq i} \left ( \rft{e}{j}^{-1}(t) \right)^+ \right ]_0^1.
  \end{equation}
On the other hand, note that
\begin{equation*}
  l_i = \lim_{n \rightarrow \infty} \pt_n(i) = \lim_{n \rightarrow \infty} \rf_{(T, i)}^{\epsilon_n}(0) = \rf_{(T, i)}(0).
\end{equation*}
Hence, $l_i > t$ iff $\rf_{(T, i)}(0) > t$, or equivalently, using part $(v)$ of Lemma~\ref{lem:Upsilon-properties}, $\rf_{(T, i)}^{-1}(t) < 0$. Substituting into \eqref{eq:f-inverse-recursive-T-i}, we conclude \eqref{eq:l_i>t-iff-sum-f}, and the proof is complete. 
\end{proof}

\edit
\begin{proof}[Proof of Lemma~\ref{lem:venkat-Gibbs-epsilon-uniqueness}]
Assume $(\theta_1, \dots, \theta_n)$ and $(\theta'_1, \dots, \theta'_n)$ are two distinct solutions to this set of equations. We claim that it can not be the case that $\theta'_i > \theta_i$ for some $i$ and $\theta'_j \leq \theta_j$ for some other $j \neq i$. Assume this holds. Since the right hand side of \eqref{eq:theta-i-g-i-fixed-point} is positive, all $\theta_i$'s and $\theta'_i$'s are positive. On the other hand, we have 
\begin{equation*}
  \frac{e^{-g_i(\theta'_i)}}{e^{-g_j(\theta'_j)}} = \frac{\theta'_i}{\theta'_j} > \frac{\theta_i}{\theta_j} = \frac{e^{-g_i(\theta_i)}}{e^{-g_j(\theta_j)}},
\end{equation*}
which means 
\begin{equation*}
  g_i(\theta'_i) + g_j(\theta_j) < g_j(\theta'_j) + g_i(\theta_i).
\end{equation*}
But $\theta'_i > \theta_i$ implies $g_i(\theta'_i) \geq g_i(\theta_i)$ since $g_i$ is nondecreasing. On the other hand, $\theta'_j \leq \theta_j$ implies $g_j(\theta'_j) \geq g_j(\theta_j)$ which is a contradiction with the above inequality. 

Hence, without loss of generality, we may assume that $\theta'_i \geq \theta_i$ for all $1 \leq i \leq n$. If it is not the case that $\theta_i = \theta'_i$ for all $1 \leq i \leq n$, then
\begin{equation*}
  \frac{\sum_{i=1}^n e^{-g_i(\theta_i)}}{\sum_{i=1}^n e^{-g_i(\theta_i)} + C} = \sum \theta_i < \sum \theta'_i =   \frac{\sum_{i=1}^n e^{-g_i(\theta'_i)}}{\sum_{i=1}^n e^{-g_i(\theta'_i)} + C}.
\end{equation*}
On the other hand, $\theta'_i \geq \theta_i$ and $g_i$ being nondecreasing implies that  $e^{-g_i(\theta_i)} \geq e^{-g_i(\theta'_i)}$ for all $1 \leq i \leq n$, which is in contradiction with the above inequality.
\end{proof}

\begin{proof}[Proof of Lemma~\ref{lem:limit-function-inverse}]
Sending $n$ to infinity in the three conditions of Definition~\ref{def:Upsilon-functions} and using the fact that $g_n$ converges pointwise to $g$ implies that $g$ is in $\Upsilon(C)$. To show \eqref{eq:lim-gn-1--g-1}, 
given $t \notin (\bigcup_{n} D_n ) \cup D$, define $x_n := g_n^{-1}(t)$ and $x := g^{-1}(t)$. Since $g_n \in \Upsilon(C)$, we have $x_n \leq g_n(x_n) \leq x_n + C$ or $x_n \in [t-C, t]$, which is a compact set. Hence, it suffices to show that any subsequential limit of $x_n$ is equal to $x$. Thus, without loss of generality, we may assume that $x_n \rightarrow x'$, and we show that $x' = x$. 

If $x'< x$, since $t$ is a continuity point for $g$, Lemma~\ref{lem:Upsilon-properties} part $(iv)$ implies that $g(x') < t$. Thereby, $g(x') \leq t - \delta$ for some $\delta > 0$. Now,
\begin{equation*}
  \begin{split}
    g_n(x_n) & = g_n(x_n) - g_n(x') + g_n(x') - g(x') + g(x') \\
    &\leq |x_n - x'| + |g_n(x') - g(x')| + g(x'),
  \end{split}
\end{equation*}
where the last inequality employs the fact that  $g_n \in \Upsilon(C)$. Since $x_n \rightarrow x'$ and $g_n$ converges pointwise to $g$, for large $n$ the first two terms could be made smaller than $\delta/3$. Thus, for large $n$, $g_n(x_n) \leq t- \delta /3$, which is a contradiction with $g_n(x_n) = t$. The assumption $x' > x$ similarly results in contradiction. As a result, $x' = x$, and the proof is complete. 
\end{proof}


\section{Characterization of the Mean Excess Function for Galton Watson Processes}
\label{sec:char-mean-excess}

In this section, we prove Theorem~\ref{thm:fixed-point-generalized-geeneralization-v1}. This is done in two steps. First, in Section~\ref{sec:lower-bound}, we prove that for any set of fixed points $\{Q_l\}_{l \geq 2}$, the LHS of \eqref{eq:mean-excecss-characterization-general-main-tatement} is no less than the RHS. This is proved in Proposition~\ref{prop:fixed-point--lower-bound-general}. Later, in Section~\ref{sec:existence-good-q}, we show that there exists a set of fixed points achieving the maximum in \eqref{eq:mean-excecss-characterization-general-main-tatement}. This is done in Proposition~\ref{prop:fixed-point-existence-Q}.

\subsection{Lower Bound}
\label{sec:lower-bound}

In this section, we use the 
indexing notation
$\nvertex$ and $\nedge$, which was introduced in Section~\ref{sec:unim-galt-wats}. The level of a vertex  $(s_1, e_1, i_1, \dots, s_k, e_k, i_k) \in \nvertex$  is defined to be $k$, and the level of $\emptyset$ is defined to be zero. Likewise, the level of an edge $(s_1, e_1, i_1, \dots, s_k, e_k) \in \nedge$ is defined to be $k$. 
Also, recall that for $v \in \nvertex$, $s \geq 2$ and $e \geq 1$, $(v, s, e)$ is an element in $\nedge$ obtained by  concatenating $(s, e)$ to the end of the string representing $v$ in $\nvertex$. Likewise, for $s \geq 2$, $e \geq 1$ and 
$1 \leq i \leq s-1$, 
$(v, s, e, i) \in \nvertex$ is defined similarly.  

We need the following tool before proving our lower bound. In the following lemma, $\Pi: \mbH_{*}(\reals) \rightarrow \mH_{*}$ is the projection defined in \eqref{eq:projection-mbH*-mH*}.

\begin{lem}
  \label{lem:kolmogorov}
Let $t \in \reals$ together with  distributions $P$ and $\{Q_k\}_{k \geq 2} \in \mQ$ be given as in Theorem \ref{thm:fixed-point-generalized-geeneralization-v1}. 
  Given a probability distribution $W$ on $\natszf$, there is a random marked rooted tree $(\bbT_W,\emptyset)$, with marks taking values in $\reals$, and with vertex set and edge set $\nvertex$ and $\nedge$, respectively, such that the underlying unmarked rooted tree is a Galton Watson tree such that the type of the root is distributed according to $W$, and the type of a non--root vertex $v = (s_1, e_1, i_1, \dots, s_r, e_r, i_r)$ in the subtree below $v$ is distributed according to 
  $\hat{P}_{s_r}$. 
  Moreover, the marks of $\bbT_W$ satisfy 
  \begin{equation}
    \label{eq:TW-marks-relation}
    \xi_{\bbT_W}(e,i) = t - \sum_{e' \ni i, e' \neq e} \left [ 1 - \sum_{j \in e', j \neq i} \xi_{\bbT_W}(e', j)^+ \right ]_0^1,
  \end{equation}
for all $(e, i) \in \evpair(\bbT_W)$. 
Furthermore, for any $L \geq 1$, conditioned on the structure of the tree up to depth $L$, the set of marks from edges in level $L$ towards vertices in that level are independent, and  for any edge--vertex pair $(e, i)$ both in level $L$, $\xi_{\bbT_W}(e, i)$ is distributed according to $Q_k$ where $k$ is the size of $e$. 

In particular, when $W = P$, the measure $\nu \in \mP(\mbH_*(\reals))$, which is defined to be the law of $[\bbT_P, \emptyset]$, is unimodular and $(\Pi)_*\nu = \ugwt(P)$. 
\end{lem}

\begin{proof}
  We first generate the collection of random variables $\Gamma_\emptyset, (\Gamma_v, X_v)_{v \in \nvertex \setminus \{\emptyset\}}$, such that $\Gamma_v$ for $v \in \nvertex$ takes value in $\natszf$ and $X_v$ for $v \in \nvertex \setminus \{\emptyset \}$ takes values in $\reals$, with the following properties: $(i)$ $(\Gamma_v)_{v \in \nvertex}$ are independent from each other such that $\Gamma_\emptyset$ has law $W$ and for $v = (s_1, e_1, i_1, \dots, s_r, e_r, i_r) \in \nvertex$, $\Gamma_v$ has law $\hat{P}_{s_r}$; $(ii)$ For any $v \in \nvertex \setminus \{\emptyset\}$, we have 
  \begin{equation}
    \label{eq:Xv-relation}
    X_v  = t- \sum_{k=2}^{h(\Gamma_v)} \sum_{l=1}^{\Gamma_v(k)} \left [ 1 - X_{(v, k, l , 1)}^+ - \dots - X_{(v, k, l, k-1)}^+ \right]_0^1;
  \end{equation}
$(iii)$ For any $L \geq 1$, $X_v$ for nodes $v$ at level $L$ are independent and for $v = (s_1, e_1, i_1, \dots, s_L, e_L, i_L)$ at level $L$, $X_v$ is  distributed according to 
$Q_{s_L}$.

We construct the law of the above random variables satisfying the above conditions  using Kolmogorov's extension theorem (see, for instance, \cite{tao2011introduction}). For an integer $L \geq 1$, define $A_L$ to be the set of nodes in $\nvertex$ with level at most $L$. For each $L \geq 1$, we introduce  the law of a subset of the above family of random variables, namely $\Gamma_\emptyset, (X_v, \Gamma_v)_{v \in A_L \setminus \{\emptyset\}}$, and denote this law by $\nu_L$. To start with, we generate $\Gamma_v$, $v \in A_L$ independently such that $\Gamma_\emptyset$ has law $W$ and $\Gamma_v$ for $v=(s_1, e_1, i_1, \dots, s_r, e_r, i_r)$, $X_v$ has law $\hat{P}_{s_r}$. In the next step, we generate $X_v$ for nodes $v$ with depth equal to $L$ independently such that $X_v$ for $v = (s_1, e_1, i_1, \dots, s_L, e_L, i_L)$ has law $Q_{s_L}$. 
Next, we define $X_v$ for nodes at levels $1$ through $L-1$ using the relation \eqref{eq:Xv-relation} inductively starting from level $L-1$ all the way up to level $1$. Using the fact that  $Q_k = F_{P,t}^{(k)}(\{Q_l\}_{l \geq 2})$, see \eqref{eq:generalized-fixed-point}, and also that $\Gamma_v$ for $v=(s_1, e_1, i_1, \dots, s_r, e_r, i_r)$ has law $\hat{P}_{s_r}$, it is evident that the set of measures $\{\nu_L\}_{L \geq 1}$ are consistent. Therefore, Kolmogorov's extension theorem implies that the set of random variables with the conditions stated above exist. 

Now, we turn these random variables into a marked random rooted tree $\bbT_W$ having vertex set and edge set $\nvertex$ and $\nedge$, respectively. To do so, we first construct the underlying unmarked tree $\bbT_W$ given the types $\Gamma_v$, $v \in \nvertex$. In the next step, for any edge $e$ and vertex $v$ being at the same level in the tree, we set $\xi_{\bbT_W}(e, v) := X_v$. 
It can be easily seen that the ``upward'' marks, i.e. marks from edges towards nodes above them, are immediately unambiguously defined.
To see this, 
for an edge $e$ at level 1, we define $\xi_{\bbT_W}(e, \emptyset)$ using \eqref{eq:TW-marks-relation}.
 We then inductively go down one level at a time, to define all 
 the other upward marks.

Now, we show that the measure $\nu$, which is defined to be the law of $[\bbT_P, \emptyset]$, is unimodular. Our proof technique is similar to that of Lemma~\ref{lem:int-f-dvmu-gen-UGWT} in Appendix~\ref{sec:unimodularity-ugwtp}. Take a Borel function $f: \mbH_{**}(\reals) \rightarrow [0,\infty)$ and note that due to our above construction,
\begin{equation*}
  \int f d \vnu = \int \partial f d \nu = \ev{\sum_{k=2}^{h(\Gamma_\emptyset)} \sum_{l=1}^{\Gamma_\emptyset(k)} f(\bbT_P, (k, l), \emptyset)}.
\end{equation*}
Using the symmetry in the construction, we have
\begin{equation*}
\int f d \vnu = \sum_{\gamma \in \natszf} P(\gamma) \sum_{k=2}^{h(\gamma)} \gamma(k) \ev{f(\bbT_P, (k, 1), \emptyset) | \Gamma_\emptyset = \gamma}.
\end{equation*}
As all the terms are nonnegative, we may change the order of summation. Also using the definition of $\hat{P}$, if $\Gamma$ is a random variable with law $P$, we get
\begin{equation*}
\int f d \vnu = \sum_{k=2}^\infty \ev{\Gamma(k)} \sum_{\gamma \in \natszf} \hat{P}_k(\gamma) \ev{f(\bbT_P, (k, 1), \emptyset) | \Gamma_\emptyset = \gamma + \typee_k}.
\end{equation*}
Now, for each $k \geq 2$, define $\wtP_k$ to be the law of the random variable $\Gamma_k + \typee_k$ where $\Gamma_k$ has law $\hat{P}_k$. With this, for each $k \geq 2$, the inner summation over $\gamma$ in the above expression could be interpreted as an expectation with respect to a tree with root type distribution $\wtP_k$, i.e. $\bbT_{\wtP_k}$. In fact, this shows that 
\begin{equation}
  \label{eq:int-f-dvmu-marked-kolmogorov}
   \int f d \vnu = \sum_{k=2}^\infty \ev{\Gamma(k)} \ev{f(\bbT_{\wtP_k}, (k, 1), \emptyset)}.
\end{equation}
Now, note that for each $k \geq 2$, due to the definition of $\wtP_k$, the tree $(\bbT_{\wtP_k})_{(k,1) \rightarrow \emptyset}$ rooted at $\emptyset$ (which we recall is obtained by removing $(k,1)$ and then taking the subtree rooted at $\emptyset$) has an underlying unmarked structure which is precisely $\gwt_k(P)$. This, in particular, implies that $\xi_{\bbT_{\wtP_k}}((k,1), \emptyset)$ has law $Q_k$. Also, by construction, for $1 \leq i \leq k-1$, $(\bbT_{\wtP_k})_{(k,1) \rightarrow (k, 1, i)}$ have independent underlying unmarked structures, all with law $\gwt_k(P)$. Moreover, the downward marks in $(\bbT_{\wtP_k})_{(k,1) \rightarrow \emptyset}$ and $(\bbT_{\wtP_k})_{(k,1) \rightarrow (k, 1, i)}$ for $1 \leq i \leq k-1$ are independent from each other and have the same distribution. Thereby, $\xi_{\bbT_{\wtP_k}}((k,1), v)$ for $v \in (k,1)$ are independent  and all have distribution $Q_k$. Now, we claim that for all $v \in (k,1)$, $(\bbT_{\wtP_k})_{(k,1) \rightarrow v}$ are equal in distribution. To see this, note that for all $v \in (k,1)$, the unmarked structure of  $(\bbT_{\wtP_k})_{(k,1) \rightarrow v}$ is $\gwt_k(P)$ and  the downward marks are constructed following the same recipes. As was discussed above, to construct upward marks, we start from the root and go down inductively. The fact that $\xi_{\bbT_{\wtP_k}}((k,1), v)$ for $v \in (k,1)$ are i.i.d. with law $Q_k$ guarantees that the downward marks in $(\bbT_{\wtP_k})_{(k,1) \rightarrow v}$ are equally distributed for all $v \in (k,1)$. This in particular implies that  for $1 \leq i \leq k-1$, 
\begin{equation*}
  \ev{f(\bbT_{\wtP_k}, (k, 1), \emptyset)}  = \ev{f(\bbT_{\wtP_k}, (k, 1), (k,1,i))}.
\end{equation*}
Using this and writing \eqref{eq:int-f-dvmu-marked-kolmogorov} for $\int \nabla f d \mu$, we conclude that $\int f d \vnu = \int \nabla f d \vnu$ which completes the proof of the unimodularity of $\nu$. 
\end{proof}

Before proving our lower bound, we state a modified version of our variational representation in Proposition~\ref{prop:variational} and a general lemma. The proof of the following lemmas are given at the end of this section.

\begin{lem}
  \label{lem:extended-variational-formula}
  Assume $\mu$ is a distribution on $\mH_*$ with $\deg(\mu) < \infty$  and $\nu$ is a unimodular distribution on $\mbH_*(\reals)$ such that $\left (\Pi \right)_* \nu = \mu$.
  Then, for any Borel allocation $\Theta: \mH_{**} \rightarrow [0,1]$ and any function $f: \mbH_*(\reals) \rightarrow [0,1]$, we have 
  \begin{equation*}
    \int (\partial \Theta - t)^+ d \mu  \geq \int \tilde{f}_\text{min} d \vnu - t \int f d \nu,
  \end{equation*}
  where
  \begin{equation*}
    \tilde{f}_\text{min}([\bH, e, i]) := \frac{1}{|e|} \min_{j \in e} f([\bH,  j]).
  \end{equation*}
\end{lem}

\begin{lem}
  \label{lem:x_i<x_j+==>sum-x_i+<1}
  Assume $x_1, \dots, x_n$ are real numbers. Then $x_i < \left [1 - \sum_{j \neq i} x_j^+ \right ]_0^1$ for all $1 \leq i \leq n$ if and only if $\sum x_i^+ < 1$.
\end{lem}

\begin{prop}
  \label{prop:fixed-point--lower-bound-general}
  Assume $P$ is a distribution on $\natszf$ such that 
$\ev{\norm{\Gamma}_1} < \infty$ 
where $\Gamma$ has law $P$. Then, with $\mu = \ugwt(P)$, for any $t \in \reals$, and any set of probability distributions on real numbers $\{Q_k\}_{k \geq 2}$ such that for all $k \geq 2$ we have $Q_k = F_{P, t}^{(k)}(\{Q_l\}_{l \geq 2})$, it holds that
  \begin{equation*}
    \Phi_\mu(t) \geq  \left ( \sum_{k=2}^\infty \frac{\ev{\Gamma(k)}}{k} \pr{ \sum_{i=1}^k X_{k,i}^+ < 1}\right ) - t \pr{\sum_{k=2}^{h(\Gamma)} \sum_{i=1}^{\Gamma(k)} Y_{k,i} > t}.
  \end{equation*}
  Here, in the first expression, $\Gamma$ is a random variable on $\natszf$ with law $P$ and $\{X_{k,i}\}_{k,i}$ are independent such that $X_{k,i}$ has law $Q_k$. Also, in the second expression, $\Gamma$ has law $P$ and $\{ Y_{k,i} \}_{k,i}$ are independent from each other and from $\Gamma$ such that $Y_{k,i}$ has the  law of the random variable $ [ 1 - (Z_{1}^+ + \dots + Z_{k-1}^+)]_0^1$ where $Z_j$ are i.i.d.\ with law $Q_k$.
\end{prop}

\begin{proof}
Note that the condition $\ev{\norm{\Gamma}_1} < \infty$ guarantees that $\deg(\mu)< \infty$. 
  Define the functions $F: \mbH_{**}(\reals) \rightarrow [0,1]$ and $f: \mbH_* \rightarrow [0,1]$ as 
\begin{equation*}
  F([\bH, e, i]) = \left [ 1 - \sum_{j \in e, j\neq i} \xi_{\bH}(e, j)^+ \right]_0^1,
\end{equation*}
and $  f([\bH, i]) := \oneu{\partial F([\bH,i]) > t}$.
  Using the unimodular measure $\nu$ constructed in Lemma~\ref{lem:kolmogorov} and the variational characterization in Lemma~\ref{lem:extended-variational-formula}, we have 
  \begin{equation}
    \label{eq:lower-bound-proof-var-char}
    \Phi_\mu(t) \geq \int \tilde{f}_\text{min} d \vnu - t \int f d \nu,
  \end{equation}
where 
\begin{equation*}
  \tilde{f}_\text{min}(\bH, e, j) = \frac{1}{|e|} \min_{j \in e} f(\bH, j) = \frac{1}{|e|} \one{\partial F (\bH,j) > t \,\,\forall\, j \in e}.
\end{equation*}
Following the proof of Lemma~\ref{lem:kolmogorov}, and in particular Equation~\eqref{eq:int-f-dvmu-marked-kolmogorov} therein, we have 
\begin{equation}
\label{eq:int-tildefmin-lowebound-var}
 \int \tilde{f}_\text{min} d \vnu = \sum_{k \geq 2} \ev{\Gamma(k)} \ev{\tilde{f}_\text{min}(\bbT_{\wtP_k},(k,1),\emptyset)}.
\end{equation}
Due to the definition of $f$, $  \tilde{f}_\text{min}(\bbT_{\wtP_k}, (k,1), \emptyset) = \frac{1}{k} \one{\partial F(\bbT_{\wtP_k}, v) > t \,\,\, \forall \, v \in (k,1)}$. For $v \in (k,1)$, 
\begin{align*}
\partial F(\bbT_{\wtP_k}, v) &= \sum_{e' \ni v} F(\bbT_{\wtP_k}, e', v) \\
&= \left[ 1 - \sum_{w \in (k,1), w \neq v} \xi_{\bbT_{\wtP_k}}((k,1),w)^+ \right]_0^1 +\\
&\qquad \sum_{e' \ni v,e' \neq (k,1)} \left [ 1 - \sum_{w \in e', w \neq v} \xi_{\bbT_{\wtP_k}}(e',w)^+ \right]_0^1.
\intertext{Using \eqref{eq:TW-marks-relation}, this yields}
\partial F(\bbT_{\wtP_k}, v) &= \left[ 1 - \sum_{w \in (k,1), w \neq v} \xi_{\bbT_{\wtP_k}}((k,1),w)^+ \right]_0^1 + t - \xi_{\bbT_{\wtP_k}}((k,1),v).
\end{align*}
Therefore, $\partial F(\bbT_{\wtP_k}, v) > t$ for all $v \in (k,1)$ if and only if for all $v \in (k,1)$, 
\begin{equation*}
  \xi_{\bbT_{\wtP_k}}((k,1),v) < \left[ 1 - \sum_{w \in (k,1), w \neq v} \xi_{\bbT_{\wtP_k}}((k,1),w)^+ \right]_0^1.
\end{equation*}
Using Lemma~\ref{lem:x_i<x_j+==>sum-x_i+<1}, this is equivalent to $\sum_{v \in (k,1)} \xi_{\bbT_{\wtP_k}}((k,1), v)^+ < 1$. Therefore, 
\begin{equation*}
  \ev{\tilde{f}_\text{min}(\bbT_{\wtP_k},(k,1),\emptyset)} = \frac{1}{k} \pr{\sum_{v \in (k,1)} \xi_{\bbT_{\wtP_k}}((k,1), v)^+ < 1}.
\end{equation*}
But as was shown in Lemma~\ref{lem:kolmogorov}, $\xi_{\bbT_{\wtP_k}}((k,1), v)$ for $v \in (k,1)$ are i.i.d. with law $Q_k$. Consequently, substituting in \eqref{eq:int-tildefmin-lowebound-var} we get
\begin{equation}
  \label{eq:int=tildefmin-lowerbound-simplified}
  \int \tilde{f}_\text{min} d \vnu = \sum_{k=2}^\infty \frac{\ev{\Gamma(k)}}{k} \pr{ \sum_{i=1}^k X_{k,i}^+ < 1},
\end{equation}
where $X_{k,i}$, $k \geq 2, 1 \leq i \leq k$ are independent such that $X_{k,i}$ has law $Q_k$. 

On the other hand,
\begin{align*}
  \int f d \nu = \pr{\partial F(\bbT_P, \emptyset) > t} &= \pr{\sum_{k=2}^{h(\Gamma_\emptyset)} \sum_{i=1}^{\Gamma_\emptyset(k)} F(\bbT_P, (k,i), \emptyset) > t} \\
&= \pr{\sum_{k=2}^{h(\Gamma_\emptyset)} \sum_{i=1}^{\Gamma_{\emptyset}(k)} \left [ 1 - \sum_{j=1}^{k-1} \xi_{\bbT_P}((k,i), (k,i,j))^+ \right]_0^1 > t}.
\end{align*}
But, as we saw in Lemma~\ref{lem:kolmogorov}, $\xi_{\bbT_P}((k,i), (k,i,j))$ for $k\geq 2$, $i \leq \Gamma_\emptyset(k)$, $1 \leq j \leq k-1$, are independent, and $\xi_{\bbT_P}((k,i), (k,i,j))$ has law $Q_k$. Thereby,
\begin{equation}
\label{eq:int-f-dnu-lowebound}
  \int f d \nu = \pr{\sum_{k=2}^{h(\Gamma)} \sum_{i=1}^{\Gamma(k)} Y_{k,i} > t},
\end{equation}
with $Y_{k,i}$ being independent from each other such that $Y_{k,i}$ has the law of the random variable $[1 - (Z_1^+ + \dots + Z_{k-1}^+)]_0^1$ with $Z_j$'s being i.i.d. with law $Q_k$. The proof is complete by substituting \eqref{eq:int=tildefmin-lowerbound-simplified} and \eqref{eq:int-f-dnu-lowebound} into~\eqref{eq:lower-bound-proof-var-char}.
\end{proof}

\begin{proof}[Proof of Lemma~\ref{lem:extended-variational-formula}]
By interpreting $\Theta$ as a function on $\mH_{**}(\reals)$ via $\Theta(\bH, e, i) := \Theta(H, e, i)$ (where we recall that $H$ is the underlying unmarked hypergraph associated to $\bH$), and also using the inequality $x^+ \geq xy$ which holds for $y \in [0,1]$, we have 
  \begin{equation}
\label{eq:proof-ext-var-1}
    \int (\partial \Theta - t)^+ d \mu \int (\partial \Theta -t)^+ d \nu \geq \int f \partial \Theta d \nu - t \int f d \nu.
  \end{equation}
  Since $\deg(\mu) < \infty$, all the integrals are finite and well defined.
  Due to the definition of $\vnu$ and the unimodularity of $\nu$ we have $\int f \partial \Theta d \nu = \int f \Theta d \vnu = \int \nabla (f \Theta) d \nu$.
On the other hand, 
  \begin{equation*}
    \nabla (f \Theta)(\bH, e, i) = \frac{1}{|e|} \sum_{j \in e} f(\bH,  j) \Theta(H, e, j) \geq \frac{1}{|e|} \min_{j \in e} f(\bH, j) = \tilde{f}_\text{min}(\bH, i),
  \end{equation*}
  where the inequality holds since $\sum_{j \in e} \Theta(H, e, j) = 1$ and $\Theta([H, e, j]) \geq 0$ for all $j \in e$. Substituting this into \eqref{eq:proof-ext-var-1} completes the proof.
\end{proof}

\begin{proof}[Proof of Lemma~\ref{lem:x_i<x_j+==>sum-x_i+<1}]
  First, assume that $x_i < \left [ 1 - \sum_{j \neq i} x_j^+ \right ]_0^1$ for all $i$. If $x_i \leq 0$ for all $i$ then nothing remains to be proved. Hence, assume that $x_i > 0$ for some $i$. Since $x_i < \left [1 - \sum_{j \neq i} x_j^+ \right ]_0^1$, we have $0<\sum_{j \neq i} x_j^+ < 1$, which means that  $1 - \sum_{j \neq i} x_j^+ \in [0,1]$. Therefore, $\left [1 - \sum_{j \neq i} x_j^+\right ]_0^1  = 1- \sum_{j \neq i} x_j^+$. On the other hand,
\begin{equation*}
  x_i^+ = x_i < \left [1 - \sum_{j \neq i} x_j^+\right ]_0^1 = 1 - \sum_{j \neq i} x_j^+,
\end{equation*}
which implies $\sum x_i^+ < 1$. 

In order to prove the other direction, take $1 \leq i \leq n$ and note that 
\begin{equation*}
x_i \leq x_i^+ = \sum_{k=1}^n x_k^+ - \sum_{j \neq i} x_j^+ < 1 - \sum_{j \neq i} x_j^+.
\end{equation*}
Moreover, $\sum_{j \neq i} x_j^+ \leq \sum x_i^+ < 1$. Thereby, $1 - \sum_{j \neq i} x_j^+ \in [0,1]$ and  $1 - \sum_{j \neq i} x_j^+ = \left [ 1 -\sum_{j \neq i} x_j^+ \right ]_0^1$. Substituting this into the above inequality completes the proof.
\end{proof}

\subsection{Upper Bound}
\label{sec:existence-good-q}

In this section, we show that there exists a family of probability distributions $\{Q_k\}_{k \geq 2}$ satisfying the fixed points equations \eqref{eq:general-fixed-points} and achieving the maximum on the RHS of
\eqref{eq:mean-excecss-characterization-general-main-tatement}.

\begin{prop}
  \label{prop:fixed-point-existence-Q}
  Assume $P$ is a distribution on $\natszf$ such that 
$\ev{\norm{\Gamma}_1} < \infty$, 
where $\Gamma$ has law $P$. Let $\mu = \ugwt(P)$. 
Given $t \in \reals$, there exists a family of probability  distributions $\{Q_l\}_{l\geq 2}$ on the set of real numbers such that, for each $k \geq 2$,  $Q_k = F^{(k)}_{P, t}(\{Q_l\}_{l \geq 2})$, and such that we have
  \begin{equation*}
    \Phi_\mu(t) = \left ( \sum_{k=2}^\infty \frac{\ev{\Gamma(k)}}{k} \pr{ \sum_{i=1}^k X_{k,i}^+ < 1}\right ) - t \pr{\sum_{k=2}^{h(\Gamma)} \sum_{i=1}^{\Gamma(k)} Y_{k,i} > t}.
  \end{equation*}
  Here, in the first expression, $\Gamma$ is a random variable on $\natszf$ with law $P$ and $\{X_{k,i}\}_{k,i}$ are i.i.d.\ such that $X_{k,i}$ has law $Q_k$. Also, in the second expression, $\Gamma$ has law $P$ and $\{ Y_{k,i} \}_{k,i}$ are independent from each other and from $\Gamma$ such that $Y_{k,i}$ has the  law of the random variable $ [ 1 - (Z_{1}^+ + \dots + Z_{k-1}^+)]_0^1$ where $Z_j$ are i.i.d. with law $Q_k$.
\end{prop}

\begin{proof}
The condition $\ev{\norm{\Gamma}_1}< \infty$ guarantees that $\deg(\mu) < \infty$.   
Therefore, Proposition~\ref{prop:epsilon-to-zero} implies that there exists a sequence $\epsilon_n \downarrow 0$ such that the sequence of $\epsilon_n$--balanced allocations $\Theta_{\epsilon_n}$ converges $\vmu$--a.e. to a Borel allocation $\Theta_0$ which is balanced with respect to $\mu$.
As $\vmu$ is supported on $\mT_{**}$, Proposition~\ref{prop:everyting-shows-at-the-root} then implies that for $\vmu$--almost all $[T, e, i] \in \mT_{**}$, we have that for all $(e', i') \in \evpair(T)$, $\Theta_{\epsilon_n}(T, e',i') \rightarrow \Theta_0(T, e',i')$. Since all hypertrees in $\mT_{**}$ are locally finite, this means that for $\vmu$--almost all $[T, e, i] \in \mT_{**}$, we have that for all $i' \in V(T)$, $\partial \Theta_{\epsilon_n}(T, i) \rightarrow \partial \Theta_0(T, i)$. Recall from Remark~\ref{rem:Theta-epsilon-H**-canonical} that $\Theta_{\epsilon_n}(T, e, i) = \theta^T_{\epsilon_n}(e, i)$ where $\theta^T_{\epsilon_n}$ is the canonical $\epsilon_n$--balanced allocation for $T$. Thereby, for $\vmu$--almost $[T, e, i] \in \mT_{**}$, the conditions of Proposition~\ref{prop:f-epsilon-converges-f-recursive-hypertree} are satisfied. Thereby, for $\vmu$--almost all $[T, e, i]$, for all $(e', i') \in \evpair(T)$, $\rft{e'}{i'}^{\epsilon_n}(.)$ converges pointwise to some $\rft{e'}{i'}(.)$. Moreover, for $\vmu$--almost all $[T, e, i]$, we have that for all $(e', i') \in \evpair(T)$, 
  \begin{equation}
    \label{eq:fixed-point-a.s.}
    \rft{e'}{i'}^{-1}(t) = t - \sum_{e'' \ni i': e'' \neq e'} \left [1 - \sum_{j \in e'', j \neq i'} \left ( \rft{e''}{j}^{-1}(t) \right)^+ \right ]_0^1,
  \end{equation}
and
  \begin{equation}
    \label{eq:partial-Theta0>t<=>sum-rft}
    \partial \Theta_0(T, i') > t \qquad \Longleftrightarrow \qquad \sum_{e'' \ni i'} \left [ 1  - \sum_{j \in e'', j \neq i'} \left ( \rft{e''}{j}^{-1}(t) \right)^+ \right ]_0^1 > t.
  \end{equation}
For $[T, e, i]$ such that $\Theta_{\epsilon_n}(e', i')$ is not convergent for some $(e', i') \in \evpair(T)$, we may define $\rft{e''}{i''}^{-1}(.)$ arbitrarily for $(e'', i'') \in \evpair(T)$. This will not impact our argument, as this happens only on a measure zero set. 
Using Lemma~\ref{lem:A-as-Atilde-ae} part $(ii)$, we conclude that for $\mu$--almost all $[T, i] \in \mT_*$, we have that for all $(e', i') \in \evpair(T)$, \eqref{eq:fixed-point-a.s.} and \eqref{eq:partial-Theta0>t<=>sum-rft} hold. 

With this, we define the functions $F$ and $G$ on $\mT_{**}$ as follows:
  \begin{equation*}
    G(T, e, i) := \rft{e}{i}^{-1}(t),
  \end{equation*}
and
  \begin{equation*}
    F(T, e, i) := \left [ 1 - \sum_{j \in e, j \neq i} \left ( \rft{e}{j}^{-1}(t) \right)^+ \right ]_0^1 = \left [ 1 - \sum_{j \in e, j \neq i}  G(T, e, j)^+ \right ]_0^1.
  \end{equation*}
Moreover, define the function $f: \mT_* \rightarrow \{0,1\}$ as $f(T, i) = \oneu{\partial F(T, i) > t}$. From \eqref{eq:partial-Theta0>t<=>sum-rft}, $\mu$--a.s.\ we have $f = \oneu{\partial \Theta_0 > t}$.   Hence, using the variational characterization in Proposition~\ref{prop:variational}, we have
  \begin{equation}
    \label{eq:var-char-for-rho}
    \int (\partial \Theta_0 - t)^+ d \mu = \int \tilde{f}_\text{min} d \vmu - t \int f d \mu,
  \end{equation}
    where $\tilde{f}_\text{min}: \mH_{**} \rightarrow [0,1]$ is defined as 
  \begin{equation*}
    \tilde{f}_\text{min}(T, e, i) = \frac{1}{|e|} \min_{j \in e} f(T, j).
  \end{equation*}
With this and the definition of $f$, we have
  \begin{equation*}
    \tilde{f}_\text{min}(T, e, i) =
    \begin{cases}
      \frac{1}{|e|} & \partial F(T, j) > t \,\, \forall j \in e,\\
      0 & \text{otherwise}.
    \end{cases}
  \end{equation*}
  From \eqref{eq:fixed-point-a.s.}, for $\vmu$--almost every $[T, e, i]$ we have 
  \begin{equation*}
    G(T,e, j) = t - \sum_{e' \ni j, e' \neq e} F(T, e', j) = t - \partial F(T, j) + F(T, e, j).
  \end{equation*}
  Therefore, $\vmu$--almost everywhere, $\partial F(T, j) > t$ iff $F(T, e, j) > G(T, e, j)$. Consequently, $\vmu$--a.e., we have 
  \begin{equation}
    \label{eq:rho-tilde-min--Y>x}
    \tilde{f}_\text{min}(T, e, i)  =
    \begin{cases}
      \frac{1}{|e|} & \qquad F(T, e, j) > G(T, e, j) \,\,\, \forall j \in e,\\
      0 & \qquad \text{otherwise}.
    \end{cases}
  \end{equation}
  Note that, by definition, we have $F(T, e, j) = \left [ 1 - \sum_{l \in e, l \neq j} G(T, e, l)^+ \right ]_0^1$. Thereby, using Lemma~\ref{lem:x_i<x_j+==>sum-x_i+<1} in Section~\ref{sec:lower-bound}, $\vmu$--a.e. we have
  \begin{equation*}
    \tilde{f}_\text{min}(T, e, i)  =
    \begin{cases}
      \frac{1}{|e|} & \sum_{j \in e} G(T, e, j)^+ < 1,\\
      0 & \text{otherwise}.
    \end{cases}
  \end{equation*}
Note that the condition $\ev{\norm{\Gamma}_1} < \infty$ in particular implies that for all $k \geq 2$ we have $\ev{\Gamma(k)} < \infty$. 
With this, using Lemma~\ref{lem:int-f-dvmu-gen-UGWT} in Appendix~\ref{sec:unimodularity-ugwtp}, we have 
\begin{equation*}
  \int \tilde{f}_\text{min} d \vmu = \sum_{k=2}^\infty \ev{\Gamma(k)} \sum_{\gamma \in \natszf} \hat{P}_k(\gamma) \ev{f(\bT, (k, 1), \emptyset) | \Gamma_\emptyset = \gamma+\typee_k},
\end{equation*}
where $\bT$ is the random rooted hypertree of Definition~\ref{def:gwt_k(P)}. Following the argument in the proof of Proposition~\ref{prop:ugwt-unimodular} in Appendix~\ref{sec:unimodularity-ugwtp}, this can be written as 
\begin{equation*}
  \int \tilde{f}_\text{min} d\vmu = \sum_{k=2}^\infty \ev{\Gamma(k)} \ev{\tilde{f}_\text{min}(\tilde{\bT}_k, (k, 1), \emptyset)},
\end{equation*}
where for $k \geq 2$, $\tilde{\bT}_k$ is a  tree with root $\emptyset$
  that has an edge $(k,1)$ of size $k$ connected to the root, with the 
  type of the other edges connected to the root being $\hat{P}_k$, 
  and with the subtrees at the other vertices of all the edges (including the
  edge $(k,1)$) generated according to the rules of $\ugwt(P)$. 
Now, using the definition of $\tilde{f}_\text{min}$, we have 
\begin{equation*}
  \int \tilde{f}_\text{min} d\vmu = \sum_{k=2}^\infty \frac{\ev{\Gamma(k)}}{k} \pr{\sum_{v \in (k,1)} G(\tilde{\bT}_k, (k, 1), v)^+ < 1}.
\end{equation*}
Now, for every $k$ and $1 \leq i \leq k-1$, let $\tilde{\bT}_{k, i}$ be the hypertree below vertex $(k, 1, i)$ rooted at $(k, 1, i)$. Moreover, let $\tilde{\bT}_{k, k}$ be the hypertree rooted at $\emptyset$ obtained from $\tilde{\bT}_k$ by removing the edge $(k, 1)$ and all its subtree. 
Now, due to the construction of $\tilde{\bT}_k$, $\tilde{\bT}_{k, i}$, for $ 1\leq i \leq k$ are i.i.d.\ $\gwt_k(P)$ hypertrees. 
Hence,  $G(\tilde{\bT}_k, (k, 1), v)$ for $v \in (k,1)$ are independent and identically distributed. Let $Q_k$ be the common distribution. This means that 
\begin{equation}
  \label{eq:good-Q-int-f-dvmu-simplified}
  \int \tilde{f}_\text{min} d\vmu= \sum_{k=2}^\infty \frac{\ev{\Gamma(k)}}{k} \pr{\sum_{i=1}^k X_{k,i}^+ < 1},
\end{equation}
where for each $k \geq 2$, $X_{k,i}$, $1 \leq i \leq k$ are i.i.d. with law $Q_k$. 

On the other hand, with $\bT$ being  the random rooted hypertree of Definition~\ref{def:gwt_k(P)}, we have
\begin{align*}
  \int f d \mu &= \pr{\partial F(\bT, \emptyset) > t}= \pr{\sum_{k=2}^{h(\Gamma_\emptyset)} \sum_{i=1}^{\Gamma_\emptyset(k)} F(\bT, (k, i), \emptyset) > t} \\
  &= \pr{\sum_{k=2}^{h(\Gamma_\emptyset)} \sum_{i=1}^{\Gamma_\emptyset(k)} \left [ 1 - \sum_{j=1}^{k-1} G(\bT, (k, i), (k, i, j))^+ \right]_0^1 > t},
\end{align*}
where $\Gamma_\emptyset$ is the type of the root in $\bT$. 
But by definition, $G(\bT, (k, i), (k, i, j)) = \rho_{\bT_{(k,i) \rightarrow (k,i,j)}}^{-1}(t)$. But $\bT_{(k,i) \rightarrow (k,i,j)}$ for $2 \leq k \leq \Gamma_\emptyset$, $1 \leq i \leq \Gamma_\emptyset(k)$, $1 \leq j \leq k-1$ are independent and $\bT_{(k,i) \rightarrow (k,i,j)}$ has law $\gwt_k(P)$. Comparing this to the above definition of the distributions $Q_k$, $k \geq 2$, we realize that $G(\bT, (k, i), (k, i, j))$ for $2 \leq k \leq \Gamma_\emptyset$, $1 \leq i \leq \Gamma_\emptyset(k)$, $1 \leq j \leq k-1$ are independent and $G(\bT, (k, i), (k, i, j))$ has law $Q_k$. Consequently, as $\Gamma_\emptyset$ in the above expression has law $P$, we have
\begin{equation}
  \label{eq:good-Q-int-f-dmu-simplified}
  \int f d \mu = \pr{\sum_{k=2}^{h(\Gamma)} \sum_{i=1}^{\Gamma(k)} Y_{k,i} > t},
\end{equation}
where $\Gamma$ has law $P$ and $\{ Y_{k,i} \}_{k,i}$ are independent from each other and from $\Gamma$ such that $Y_{k,i}$ has the  law of the random variable $ [ 1 - (Z_{1}^+ + \dots + Z_{k-1}^+)]_0^1$ where $Z_j$ are i.i.d.\ with law $Q_k$. This together with \eqref{eq:good-Q-int-f-dvmu-simplified} and the variational expression~\eqref{eq:var-char-for-rho}, completes the proof.
\end{proof}


\section{Convergence of Maximum Load}
\label{sec:conv-maxim-load}

In this section,  we first  introduce our configuration model and conditions under which it converges to the unimodular Galton--Watson hypertree model defined in Section~\ref{sec:unim-galt-wats}. This is done in Section~\ref{sec:conf-model-hypergr} below, specifically Theorem~\ref{thm:conf-model-convegence-Un}. We then state the conditions under which we prove Theorem~\ref{thm:max-load} in Section~\ref{sec:maxload-prop}, Proposition~\ref{prop:max-load-detail} and give the proof. 


Before introducing our configuration model, we need to formally define multihypergraphs. Here, we only work with \emph{finite} multihypergraphs.


\begin{definition}
  \label{def:multi-hypergraph}
  A finite multihypergraph $H = \langle V, E = (e_j, j \in  J) \rangle$ consists of a finite vertex set $V$ together with a finite edge index set $J$ such that each hyperedge $e_j$ is a multiset of vertices in  $V$.
\end{definition}

Here, the assumption of $e_j$ being a multiset allows for vertices to appear more than once in each edge. In this case, we call such an edge a ``self loop''. Moreover, it might be the case that $e_j = e_{j'}$ for $j \neq j' \in I$, in which case we call $e_j$ and $e_{j'}$ ``multiple edges''. 

\subsection{Configuration model on Hypergraphs}
\label{sec:conf-model-hypergr}

We proposed a generalized Galton Watson process for hypertrees in Section~\ref{sec:unim-galt-wats} and showed it is unimodular. In this section, we propose a configuration model which converges to it in the local weak sense under certain conditions.

Assume that, for each integer $n$, a type sequence $\gn = (\gn_1, \dots, \gn_n)$ is given such that $\gn_i \in \natszf$ and
\begin{subequations}
  \begin{gather}
    \gn_i(k) = 0 \qquad \forall\, 1 \leq i \leq n, k > n, \text{and} \label{eq:gn-0-for-bigger-n}\\
    k \bigg| \sum_{i=1}^n \gn_i(k) \qquad \forall\,  2 \leq k \leq n, \label{eq:k|sum-gn-k} 
  \end{gather}
\end{subequations}
where the latter means that $k$ divides $\sum_{i=1}^n \gn_i(k)$.
In what follows, we generate a random multihypergraph on the vertex set $\{1, \dots, n \}$ such that the type of node $i$ is $\gn_i$.  For each $2 \leq k \leq n$ and $1 \leq i \leq n$, we attach $\gn_i(k)$ many objects $e^k_{i,1}, \dots, e^k_{i, \gn_i(k)}$ called ``$1/k$--edges'' to the node $i$ .
For each $k$, let $\Delta^{(n)}(k)$ be defined as the set of all $1/k$--edges, i.e.
\begin{equation*}
  \Delta^{(n)}(k) := \bigcup_{i=1}^n \{ e^k_{i,1}, \dots, e^k_{i, \gn_i(k)} \},
\end{equation*}
and let $\Delta^{(n)} := \cup_k \Delta^{(n)}(k)$ be the set of all ``partial edges'', where by a partial edge we mean a $1/k$--edge for some $k$.
Also, let $\Delta^{(n)}_i$  be the set of all partial edges connected to a node $i \in \{1, \dots, n\}$, i.e.
\begin{equation*}
  \Delta^{(n)}_i = \bigcup_{k: \gn_i(k) > 0} \{ e^k_{i, 1}, \dots, e^k_{i,\gn_i(k)} \}.
\end{equation*}
For a partial edge $e \in \Delta^{(n)}$, define $\nu(e)$ to be the node it corresponds to, i.e. $\nu(e^k_{i,j}) = i$. Also, define $|e|$ to be the size of  $e$, i.e. $|e^k_{i,j}| = k$. 

We say that a permutation $\sigma_k$ is a $k$--matching on the set $\Delta^{(n)}(k)$ if it is a permutation on $\Delta^{(n)}(k)$ with no fixed points and also with all the cycles having  size exactly equal to $k$. 
Due to the condition $k | \sum_{i=1}^n \gn_i(k)$, such $k$--matchings exist for all $k$ such that $\Delta^{(n)}(k) \neq \emptyset$. In fact, if for a finite nonempty set $A$, whose cardinality  is divisible by $k$, we denote the set of $k$--matchings on $A$ with $\mM_k(A)$, it could be easily checked that $  |\mM_k(A)| = \frac{|A|!}{k^{|A|/k} (|A|/k)!}$.

Given this, for each $k$ such that $\Delta^{(n)}(k) \neq \emptyset$, we pick $\sigma_k$ uniformly at random from $\mM_k(\Delta^{(n)}(k))$, 
independently over $k$.
With this, we generate a random multihypergraph $H_n$ on the set of vertices $\{1, \dots, n\}$. This is done by identifying each cycle of $\sigma_k$ such that $\Delta^{(n)}(k) \neq \emptyset$ of the form $(e, \sigma_k(e), \dots, \sigma_k^{(k-1)}(e))$ with the edge $\{\nu(e), \nu(\sigma_k(e)), \dots, \nu(\sigma_k^{(k-1)}(e)) \}$ in $H_n$. Here,  $e \in \Delta^{(n)}(k)$ and $\sigma_k^{(l)}$ 
denotes the  permutation $\sigma_k$ begin applied $l$ times.
Note that it is possible that two realizations of permutations as above result in the same multihypergraph $H_n$. As an example, if $n=3$, $\gn_1(2) = 2$, $\gn_2(2) = \gn_3(2) = 1$, and $\gn_i(k) = 0$ for $k \neq 2$, $1 \leq i \leq 3$, then the two permutations $\sigma_2$ and $\sigma'_2$ on $\Delta^{(n)}(2)$ presented in the cycle notation as $\sigma_2 = ((e^2_{1,1}, e^2_{2,1})(e^2_{1,2}, e^2_{3,1})$ and $\sigma'_2 = (e^2_{1,2}, e^2_{2,1})(e^2_{1,1}, e^2_{3,1})$ would result in the same multihypergraph (which turns out to be a simple graph in this example).

Note that, in general, $H_n$ might not be simple. In particular, it might contain  edges which contain a vertex more than once (we call such an edge a ``self loop''), or it might be the case that two edges exist in $H_n$ with the exact same multiset of vertices (we call such an edge a ``multiple edge''), or these two can happen simultaneously.
Having generated $H_n$, we may generate a simple hypergraph $H_n^e$ by deleting all such self loops and multiple edges in $H_n$, i.e.\ we first remove all self loops and, subsequently, we delete all edges with the same set of endpoints. 
See Figure~\ref{fig:conf-example} for an example of an instance of the configuration model.

\begin{figure}
  \begin{center}
  \begin{minipage}{0.3\linewidth}
    \centering
    \begin{tikzpicture}
    \node[Node,label={above:$1$}] (a1) at (90:1) {};
    \node[Node,label={left:$2$}] (a2) at (210:1) {};
    \node[Node,label={right:$3$}] (a3) at (-30:1) {};
    \node[Node,label={left:$4$}] (a4) at ($(a2)+(0,-1.7)$) {};
    \node[Node,label={right:$5$}] (a5) at ($(a3)+(0,-1.7)$) {};
    
    \draw[Orange,thick, dashed, dash pattern = on 2pt off 1pt] (a1) -- node[left,scale=0.8,black] {$e^3_{1,1}$} ($(a1)+(-120:0.5)$);
    \draw[Orange,thick, dashed, dash pattern = on 2pt off 1pt] (a1) -- node[right,scale=0.8,black] {$e^3_{1,2}$} ($(a1)+(-60:0.5)$);

    \draw[Orange,thick, dashed, dash pattern = on 2pt off 1pt] (a2) -- node[below,near end,scale=0.8,black] {$e^3_{2,2}$} ($(a2)+(0:0.5)$);
    \draw[Orange,thick, dashed, dash pattern = on 2pt off 1pt] (a2) -- node[above=1mm,scale=0.8,black,black] {$e^3_{2,1}$} ($(a2)+(60:0.5)$);
    \draw[magenta,thick, dashed, dash pattern = on 2pt off 1pt] (a2) -- node[left,near end, scale=0.8,black] {$e^2_{2,1}$} ($(a2)+(-90:0.5)$);

    \draw[Orange,thick, dashed, dash pattern = on 2pt off 1pt] (a3) -- node[below,near end,scale=0.8,black] {$e^3_{3,2}$} ($(a3)+(180:0.5)$);
    \draw[Orange,thick, dashed, dash pattern = on 2pt off 1pt] (a3) -- node[above=1mm,scale=0.8,black] {$e^3_{3,1}$} ($(a3)+(120:0.5)$);
    \draw[Orange,thick, dashed, dash pattern = on 2pt off 1pt] (a3) -- node[right,near end, scale=0.8,black] {$e^3_{3,3}$} ($(a3)+(-90:0.5)$);

    \draw[Orange,thick, dashed, dash pattern = on 2pt off 1pt] (a4) -- node[right=1mm,scale=0.8,black] {$e^3_{4,1}$} ($(a4)+(30:0.5)$);
    \draw[magenta,thick, dashed, dash pattern = on 2pt off 1pt] (a4) -- node[left,near end,scale=0.8,black] {$e^2_{4,1}$} ($(a4)+(90:0.5)$);

    \draw[Orange,thick, dashed, dash pattern = on 2pt off 1pt] (a5) -- node[left,scale=0.8,black] {$e^3_{5,1}$} ($(a5)+(150:0.5)$);
    \draw[magenta,thick, dashed, dash pattern = on 2pt off 1pt] (a5) -- node[above=1mm,scale=0.8,black] {$e^2_{5,1}$} ($(a5)+(90:0.5)$);
    \draw[magenta,thick, dashed, dash pattern = on 2pt off 1pt] (a5) -- node[right=1mm,scale=0.8,black] {$e^2_{5,2}$} ($(a5)+(30:0.5)$);
    
    \node at (0,-3) {$(a)$};
    \end{tikzpicture}
  \end{minipage}%
\begin{minipage}{0.3\linewidth}
  \centering
    \begin{tikzpicture}
    \node[Node,label={[label distance=3mm]above:$1$}] (b1) at (90:1) {};
    \node[Node,label={[label distance=3mm]left:$2$}] (b2) at (210:1) {};
    \node[Node,label={[label distance=3mm]right:$3$}] (b3) at (-30:1) {};
    \node[Node,label={[label distance=3mm]left:$4$}] (b4) at ($(a2)+(0,-1.7)$) {};
    \node[Node,label={[label distance=3mm]right:$5$}] (b5) at ($(a3)+(0,-1.7)$) {};
    
    \draw[Cyan, rounded corners=8pt, thick] ($(b1) + (90:0.3)$) -- ($(b2) + (210:0.3)$) -- ($(b3) + (-30:0.3)$) -- cycle;
    \draw[Cyan, rounded corners=12pt, thick] ($(b1) + (90:0.5)$) -- ($(b2) + (210:0.5)$) -- ($(b3) + (-30:0.5)$) -- cycle;
    
    \draw[Cyan, rounded corners=7pt, thick] ($(b3)+(0.2,0.35)$) -- ($(b4)+(-0.3,-0.15)$) -- ($(b5)+(0.2,-0.15)$) -- cycle;

    \draw[Cyan, thick] (b2) -- (b4);
    \draw[Cyan, thick] (b5) -- (b5);
    
    \draw[Cyan,thick] (b5) to[in=-10,out=-90,loop] (b5);
    \node at (0,-3) {$(b)$};
    \end{tikzpicture}
  \end{minipage}%
\begin{minipage}{0.3\linewidth}
  \centering
    \begin{tikzpicture}
    \node[Node,label={[label distance=3mm]above:$1$}] (b1) at (90:1) {};
    \node[Node,label={[label distance=3mm]left:$2$}] (b2) at (210:1) {};
    \node[Node,label={[label distance=3mm]right:$3$}] (b3) at (-30:1) {};
    \node[Node,label={[label distance=3mm]left:$4$}] (b4) at ($(a2)+(0,-1.7)$) {};
    \node[Node,label={[label distance=3mm]right:$5$}] (b5) at ($(a3)+(0,-1.7)$) {};
    
    
    \draw[Cyan, rounded corners=7pt, thick] ($(b3)+(0.2,0.35)$) -- ($(b4)+(-0.3,-0.15)$) -- ($(b5)+(0.2,-0.15)$) -- cycle;

    \draw[Cyan, thick] (b2) -- (b4);
    \draw[Cyan, thick] (b5) -- (b5);

    \node at (0,-3) {$(c)$};
    \end{tikzpicture}
  \end{minipage}
\end{center}
  \caption{An example of an instance of the  configuration model on $n=5$ vertices with vertex types  $\gn_1 = (0,2)$, $\gn_2 = (1,2)$, $\gn_3 = (0,3)$, $\gn_4=(1,1)$, $\gn_5=(1,2)$. Here, a type $\gamma$ is represented by a vector where the first coordinate is $\gamma(2)$, the second is $\gamma(3)$ and so on. $(a)$ illustrates the set of partial edges connected to vertices. $(b)$ depicts the multihypergraph $H_n$ formed by matching the partial edges using the permutations $\sigma_2$ and $\sigma_3$ which are represented in the cycle notation as $\sigma_2 = (e^2_{2,1},e^2_{4,1})(e^2_{5,1},e^2_{5,2})$ and $\sigma_3 = (e^3_{1,1}, e^3_{2,1}, e^3_{3,1})(e^3_{1,2},e^3_{2,2},e^3_{3,2})(e^3_{3,3},e^3_{4,1},e^3_{5,1})$. $(c)$ illustrates the simple hypergraph $H_n^e$ formed by removing the multiple edges of size 3 on vertices $1,2,3$ and also the self loop on vertex $5$. }
  \label{fig:conf-example}
\end{figure}
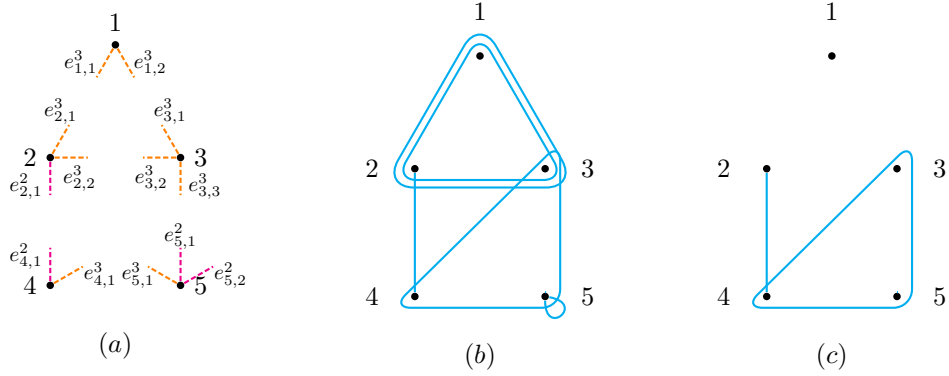

If one fixes a type sequence $\gn = (\gn_1, \dots, \gn_n)$ for each $n$, then the above configuration model generates a sequence of random hypergraphs $H^e_n$. Since $H^e_n$ is random, $u_{H^e_n}$ is also a random probability distribution over $\mH_*$. Let $\ev{u_{H^e_n}}$ be the expectation with respect to the randomness of $H^e_n$, i.e. a weighted average over all the possible configurations of $H^e_n$ (the number of such configurations is indeed finite for each $n$). Hence, $\ev{u_{H^e_n}}$ is a sequence of probability distributions over $\mH_*$ and one can ask whether it has a weak limit. On the other hand, one can build all $H^e_n$ on a common probability space under which $H^e_n$ are independent for different $n$. Then, for some $\mu \in \mP(\mH_*)$, we say that $u_{H^e_n} \Rightarrow \mu$ almost surely when outside a measure zero set in this common probability space, this convergence holds. The following theorem shows that under some conditions on the type sequence, $H^e_n$ has a local weak limit. See Appendix~\ref{sec:proof-theorem-lws-conf-model} for a proof. 

\begin{thm}
\label{thm:conf-model-convegence-Un}
Assume that a probability distribution $P$ on $\natszf$ is given and define $I := \{ k \geq 2: P(\{\gamma \in \natszf: \gamma(k) > 0 \}) > 0 \}$. 
Furthermore, let $\gn= (\gn_1, \dots, \gn_n)$ be a type sequence satisfying~\eqref{eq:gn-0-for-bigger-n} and \eqref{eq:k|sum-gn-k} such that
\begin{subequations}
\begin{gather}
  \gn_i(k) = 0 \qquad \forall k \notin I \quad \forall n,i \label{eq:type-zero-outside-index-set}, \\
    \lim_{n \rightarrow} \frac{1}{n} \sum_{i=1}^n \oneu{\gn_i = \gamma} = P(\gamma) \qquad \forall \gamma \in \natszf \label{eq:fraction-converge-P(gamma)},\\
    \limsup_{n \rightarrow \infty} \frac{1}{n} \sum_{i=1}^n \snorm{\gn_i}_1^2 < \infty. \label{eq:second-moment-of-one-norm-bounded}
  \end{gather}
\end{subequations}
Additionally, assume that there are constants $c_1, c_2, c_3, \epsilon>0$ such that for $n$ large enough,
\begin{subequations}
  \begin{gather}
            \max_{1 \leq i \leq n} \snorm{\gn_i}_1 \leq c_1 (\log n)^{c_2}, \label{eq:norm-1-gamma-log} \\
            \max_{1 \leq i \leq n} h(\gn_i) \leq c_1 (\log n)^{c_2} \label{eq:norm-infty-gamma-log}, \\
        \forall 2 \leq k \leq n \quad \Delta^{(n)}(k) = \emptyset \quad \text{ or } \quad |\Delta^{(n)}(k)| \geq c_3 n^\epsilon. \label{eq:delta-n-k-big-enough}
  \end{gather}
\end{subequations}
Then, if $H^e_n$ is the random hypergraph generated from the configuration model corresponding to $\gn$ described above, we have $u_{H^e_n} \Rightarrow \ugwt(P)$ almost surely. 
\end{thm}

\begin{rem}
\label{rem:I-everything}
  Note the  the index set $I$ would let us include only certain edge sizes in our model. For instance, when $I = \{2\}$, this theorem reduces to a statement for the graph pairing model, and when $I = \{k:k \geq 2\}$ it allows for all edge sizes to be present. 
\end{rem}

\begin{rem}
\label{rem:sum-gni-Ev-Gamma-k}
  For a fixed $k$ and for each $n$, define $X_n$ to be an integer valued  random variable taking value $\gn_i(k)$ with probability $1/n$ for $1 \leq i \leq n$. Then, the condition~\eqref{eq:second-moment-of-one-norm-bounded} implies that the sequence $\{X_n\}$ is uniformly integrable. Also, \eqref{eq:fraction-converge-P(gamma)} implies that $X_n \stackrel{d}{\rightarrow} \Gamma(k)$ where $\Gamma$ has law $P$. Thus, $\ev{X_n} \rightarrow \ev{\Gamma(k)}$, i.e. 
  \begin{equation}
      \lim_{n \rightarrow \infty} \frac{1}{n} \sum_{i=1}^n \gn_i(k) = \ev{\Gamma(k)} \quad \forall k >1 \label{eq:ev-type-first-moment-converge}.
  \end{equation}
This identity is  useful in our future analysis.
\end{rem}

\begin{rem}
It can be proved that, under some regularity conditions for $P$, when the type sequence is generated i.i.d., that the conditions of Theorem~\ref{thm:conf-model-convegence-Un}  are satisfied with probability one. However, we omit the proof of such a claim here, since it is not central to our discussion.
\end{rem}

We will use the following simplified version of the above theorem in this section.


\begin{cor}
  \label{cor:conf-model-convergence-finite-index-set}
  Assume that a probability distribution $P$ over $\natszf$ is given such that for a finite set $I \subset \{2, 3, \dots, \}$, if $\Gamma$ is a random variable with law $P$, we have $\pr{\Gamma(k) > 0} > 0$ for $k \in I$ and $\pr{\Gamma(k) > 0} = 0$ for $k \notin I$. Furthermore, assume $\gn = (\gn_1, \dots, \gn_n)$ is a type sequence satisfying \eqref{eq:gn-0-for-bigger-n}, \eqref{eq:k|sum-gn-k}, \eqref{eq:type-zero-outside-index-set} and \eqref{eq:fraction-converge-P(gamma)}. 
Additionally, assume that 
for some $\theta > 0$, 
  \begin{equation}
    \label{eq:exp-theta-mgf-finite}
    \limsup_{n \rightarrow \infty} \frac{1}{n} \sum_{i=1}^n e^{\theta \snorm{\gn_i}_1} < \infty.
  \end{equation}
Then, $u_{H^e_n} \Rightarrow \ugwt(P)$ almost surely.
\end{cor}

\begin{proof}
We check that in this regime, all the conditions of Theorem~\ref{thm:conf-model-convegence-Un} are satisfied.    
Note that 
   \begin{align*}
     \frac{1}{n} \sum_{i=1}^n \gn_i(k)^2 \leq \frac{2!}{\theta^2} \frac{1}{n} \sum_{i=1}^n e^{\theta \gn_i(k)} \leq \frac{2!}{\theta^2} \frac{1}{n} \sum_{i=1}^n e^{\theta \snorm{\gn_i}_1},
   \end{align*}
which, together with \eqref{eq:exp-theta-mgf-finite}, implies \eqref{eq:second-moment-of-one-norm-bounded}.
In order to show \eqref{eq:norm-1-gamma-log}, note that \eqref{eq:exp-theta-mgf-finite} implies there is a constant $\lambda < \infty$ such that $\frac{1}{n} \sum_{i=1}^n e^{\theta \snorm{\gn_i}_1} < \lambda$ for all $n$. Now, 
$\exp(\theta \max_{1 \leq i \leq n} \snorm{\gn_i}_1) \leq \sum_{i=1}^n e^{\theta \snorm{\gn_i}_1} \leq n \lambda$. Hence, $  \max_{1 \leq i \leq n} \snorm{\gn_i}_1 \leq (\log n + \log \lambda)/\theta$, which shows \eqref{eq:norm-1-gamma-log}. On the other hand, 
$h(\gn_i) \leq \max\{ k: k \in I\}$ 
for all $i$ and $n$. Hence, \eqref{eq:norm-infty-gamma-log} also holds.

For a fixed $k \in I$, using \eqref{eq:ev-type-first-moment-converge}, which follows from \eqref{eq:second-moment-of-one-norm-bounded} which was proved above, for $n$ large enough we have
$|\Delta^{(n)}(k)| = \sum_{i=1}^n \gn_i(k) \geq n \ev{\Gamma(k)}/2$.
Since there are finitely many $k \in I$, we have that for $n$ large enough
$|\Delta^{(n)}(k)| \geq n \min_{k' \in I} \ev{\Gamma(k')}/2$.
But for $k' \in I$, we have $\ev{\Gamma(k')} \geq \pr{\Gamma(k') > 0} > 0$ by assumption. Consequently, \eqref{eq:delta-n-k-big-enough} holds with $\epsilon = 1$ and  $c_3:= \min_{k' \in I} \ev{\Gamma(k')}/2 > 0$. This means that  all the conditions of Theorem~\ref{thm:conf-model-convegence-Un} are satisfied and  $u_{H^e_n} \Rightarrow \ugwt(P)$ almost surely.
\end{proof}

\subsection{Statement of the result}
\label{sec:maxload-prop}

Here, we state the conditions under which we prove Theorem~\ref{thm:max-load}. 

\begin{prop}
\label{prop:max-load-detail}
  Assume that a probability distribution $P$ over $\natszf$ is given such that for a finite set $I \subset \{2, 3, \dots, \}$, if $\Gamma$ is a random variable with law $P$, we have $\pr{\Gamma(k) > 0} > 0$ for $k \in I$ and $\pr{\Gamma(k) > 0} = 0$ for $k \notin I$. 
 Moreover, assume that for all $k\in I$, $\ev{\Gamma(k)} < \infty$.
Also, with $k_\text{min}$ being the minimum element in $I$, assume that $\pr{\Gamma(k_\text{min}) = 0} + \pr{\Gamma(k_\text{min}) = 1} < 1$. Moreover, assume that a sequence of types $\gn = (\gn_1, \dots, \gn_n)$ is given satisfying 
  \eqref{eq:gn-0-for-bigger-n}, \eqref{eq:k|sum-gn-k}, \eqref{eq:type-zero-outside-index-set}~\eqref{eq:fraction-converge-P(gamma)} and \eqref{eq:exp-theta-mgf-finite}. 
  Then, if $H_n^e$ is the simple hypergraph generated by the configuration model in Section~\ref{sec:conf-model-hypergr}, $\varrho(H_n^e)$ converges in probability to $\varrho(\mu)$, where $\mu = \ugwt(P)$. 
\end{prop}

Before proving this proposition, we need the following two lemmas, whose proofs are given at the end of this section. 

\begin{lem}
\label{lem:maxload-binom}
  With the assumptions of Proposition~\ref{prop:max-load-detail}, there is a constant $c_4 > 0$ such that for $n$ large enough, for any subset $S \subset \{1, \dots, n\}$, the number of edges in $H_n^e$ with all endpoints in $S$ is stochastically dominated by 
  the sum of $s$ independent Bernoulli random variables, each having
  mean $c_4 s^{k_\text{min} - 1}/n^{k_\text{min}-1}$,
  where $s := \sum_{i \in S} \snorm{\gn_i}_1$. 
\end{lem}


\begin{lem}
\label{lem:maxload-tozero}
  With the assumptions of Proposition~\ref{prop:max-load-detail}, if $t > \frac{1}{k_\text{min}-1}$, there exists $\delta>0$ such that if $Z^{(n)}_{\delta, t}$ denotes the number of subsets $S$ of $\{1, \dots, n\}$ with size at most $n \delta$ such that $H_n^e$ has at least $t |S|$ many edges with all endpoints in $S$, we have $\pr{Z^{(n)}_{\delta, t}>0} \rightarrow 0$ as $n \rightarrow \infty$. 
\end{lem}

\begin{proof}[Proof of Proposition~\ref{prop:max-load-detail}]
Let $\mu_n$ denote $u_{H^e_n}$, which is a random probability distribution on $\mH_*$. Then, Corollary~\ref{cor:conf-model-convergence-finite-index-set} 
 guarantees that $\mu_n \Rightarrow \mu$ almost surely. As a result, if $\mL_{n}$ is the law of the balanced load for $H^e_n$ and $\mL$ is the law of $\partial \Theta$, where $\Theta$ is the balanced allocation corresponding to $\mu$, then, using 
Theorem~\ref{thm:balanced--properties},
we have  $\mL_{n} \Rightarrow \mL$ almost surely. Now, let $t := \varrho(\mu)$, and fix $\epsilon > 0$. Due to the definition of $\varrho(\mu)$, $\mL((t-\epsilon, \infty)) > 0$. As a result, using the portmanteau theorem  (see, for instance, \cite[Theorem~2.1]{billingsley2013convergence}), since $\mu_n \Rightarrow \mu$ almost surely, we have 
\begin{equation*}
  \liminf_n \mL_n((t-\epsilon, \infty)) > 0 \qquad a.s..
\end{equation*}
This means that 
\begin{equation}
\label{eq:pr-varrho-Hen<t-epislon-zero}
  \pr{\varrho(H^e_n) \leq t - \epsilon} = \pr{\mL_{n}((t - \epsilon, \infty)) = 0} \rightarrow 0.
\end{equation}

Now we show that $\pr{\varrho(H^e_n) \geq t + \epsilon}$ also converges to zero. 
To do so, fix some $\delta > 0$ and note that  
\begin{equation}
\label{eq:pr-varrho-Hen>t+1}
\begin{split}
  \pr{\varrho(H^e_n) \geq t + \epsilon} & = \pr{\mL_n([t+\epsilon, \infty)) > 0} \\
  &= \pr{\mL_n([t+\epsilon, \infty)) > \delta} + \pr{0 < \mL_n([t+\epsilon, \infty)) \leq \delta}.
\end{split}
\end{equation}
The portmanteau theorem implies that 
$ \pr{\mL_n([t+\epsilon, \infty)) > \delta}$ converges to zero. Now, we argue that the second term also converges to zero. If $\theta_n$ denotes the balanced allocation on $H^e_n$,  then, by the definition of $\mL_n$, on the event $0 < \mL_n([t+\epsilon, \infty)) \leq \delta$, the set $S:= \{ 1 \leq i \leq n: \partial \theta_n(i) \geq t +\epsilon \}$ is non--empty and $|S| \leq \delta n$. Now if $e \in E(H^e_n)$ with $i, j \in e$ such that $i \in S$ and $j \in S^c$, then, since $\partial \theta_n(j) < t + \epsilon \leq \partial \theta_n(i)$, we have $\theta_n(e, i) = 0$. Hence, for all $i \in S$, $\partial \theta_n(i) = \sum_{e \subseteq S, e\in E(H^e_n)} \theta_n(e, i)$. Thus, 
\begin{equation*}
  \sum_{i \in S} \partial \theta_n(i) = | \{ e\in E(H^e_n): e \subseteq S \} | = |E_{H^e_n}(S)|.
\end{equation*}
On the other hand, we have $\sum_{i \in S} \partial \theta_n(i) \geq |S|(t+\epsilon)$. Hence, $|E_{H^e_n}(S)| \geq (t+\epsilon)|S|$,
where $E_{H_n^e}(S)$ denotes the set of edges in $H_n^e$ with all endpoints in $S$. 
As a result, if $Z^{(n)}_{\delta, t+\epsilon}$ denotes the number of subsets $S \subset \{1, \dots, n\}$ with size at most $n \delta$ such that $H_n^e$ has  at least $(t+\epsilon)|S|$ many edges with all endpoints in $S$, we have 
\begin{align*}
  \pr{0 < \mL_n([t+\epsilon, \infty)) < \delta} &\leq \mathbb{P}\Bigg (\exists\, S \subseteq \{1, \dots, n\}, \,\, 0 < |S| \leq \delta n, \\
&\qquad \qquad |E_{H^e_n}(S)| \geq (t+\epsilon) |S| \Bigg ) \\
  &\leq \pr{Z^{(n)}_{\delta, t+\epsilon} > 0}.
\end{align*}

If we have $t \geq 1/(k_\text{min}-1)$, Lemma~\ref{lem:maxload-tozero} above implies that 
$\pr{Z^{(n)}_{\delta, t+\epsilon} > 0}$ goes to zero as $n\rightarrow \infty$ and we are done. We now argue why this is the case. For this, note that Proposition~\ref{prop:epsilon-to-zero} together with Proposition~\ref{prop:everyting-shows-at-the-root} imply that there exists a sequence of $\epsilon_m$--balanced allocations $\Theta_{\epsilon_m}$ such that for $\mu$--almost all $[H, i] \in \mH_*$, for all vertices $j \in V(H)$, $\partial \Theta_{\epsilon_m} (H, j) \rightarrow \partial \Theta_0(H, j)$ where $\Theta_0$ is a balanced allocation with respect to $\mu$. Moreover, using Remark~\ref{rem:Theta-epsilon-H**-canonical}, $\partial \Theta_{\epsilon_m}(H, j) = \partial \theta_{\epsilon_m}^{H}(j)$ where $\theta_{\epsilon_n}^H$ is the canonical $\epsilon$--balanced allocation on $H$. 

On the other hand, the assumption $\pr{\Gamma(k_\text{min}) = 0} + \pr{\Gamma(k_\text{min}) = 1} < 1$ guarantees that for all integer $M \geq 2$,  there is a nonzero probability under $\mu$ that the Galton--Watson process has a finite sub--hypertree containing the root, and  having $M$ edges with all these edge having size $k_\text{min}$. It can be easily seen that a finite hypertree with $M$ edges all having size $c$ has $1+M(c-1)$ vertices and there is a balanced allocation on such a hypertree such that all the vertices get the same amount of load, which is equal to $\frac{M}{1+M(c-1)}$. 
Motivated by the discussion in the previous paragraph, for $\mu$--almost all $[H, i] \in \mH_*$, $\partial \Theta_0(H, i) = \lim_{m \rightarrow \infty} \partial \theta_{\epsilon_m}^{H}(i)$ where $\theta_{\epsilon_m}^H$ is the canonical $\epsilon_m$--balanced allocation on $H$. This, together with Proposition~\ref{prop:monotonicity-canonical-eba}, implies that for each integer $M$,  there is a nonzero probability under $\mu$ that $\partial \Theta_0(H, i)$  is at least $M / (1+M(k_\text{min}-1))$. Sending $M \rightarrow \infty$, this means that $t = \varrho(\mu) \geq 1/(k_\text{min}-1)$. As was discussed above, using Lemma~\ref{lem:maxload-tozero} above,  $\pr{Z^{(n)}_{\delta, t+\epsilon} > 0} \rightarrow 0$. Thus, $\pr{\varrho(H^e_n) \geq t + \epsilon}$ goes to zero as $n$ goes to infinity. This together with \eqref{eq:pr-varrho-Hen<t-epislon-zero} proves that $\varrho(H^e_n) \stackrel{p}{\rightarrow} \varrho(\mu)$.
\end{proof}

\begin{proof}[Proof of Lemma~\ref{lem:maxload-binom}]
  For $k \in I$, let $s_k := \sum_{i \in S} \gn_i(k)$ and $m_k := |\Delta^{(n)}(k)|$. 
As the set of edges in $H_n^e$ is a subset of that of $H_n$, we may prove the result for $H_n$ instead. This allows us to directly analyze the configuration model. 
  Let $A$ be the set of the partial edges  connected to vertices in $S$. Note that $|A| = s$. 
  Arbitrarily order the elements in $A$. 
At time $t =1$, we pick the smallest element in $A$. Let $k_1$ be the size of this edge. Then, we choose $k_1-1$ other partial edges in $\Delta^{(n)}(k_1)$ uniformly at random to form an edge in $H_n$. We continue this process until all the elements in $A$ are used up. More precisely, at time $t$, we pick the smallest available partial edge in $A$, namely $e_t$, and if $k_t$ is the size of $e_t$, we match $e_t$ with $k_t-1$ other elements in the available partial edges in $\Delta^{(n)}(k_t)$ uniformly at random to form an edge in $H_n$. At time $t$, let $s_{k_t}(t)$ and $m_{k_t}(t)$ be the number of available partial edges of size $k_t$ in $A$ and $\Delta^{(n)}(k_t)$, respectively. With this, if $p_t$ denotes the probability that $e_t$ is matched with partial edges all inside $A$, 
  \begin{equation*}
    p_t = \frac{s_{k_t}(t)-1}{m_{k_t}(t) - 1} \times \frac{s_{k_t}(t) - 2}{m_{k_t}(t) - 2} \times \dots \times \frac{s_{k_t}(t) - (k_t-1)}{m_{k_t}(t) - (k_t-1)} \leq \left ( \frac{s_{k_t}(t)}{m_{k_t}(t)} \right)^{k_t-1}. 
  \end{equation*}
Note that if $l_t$ denotes the number of partial edges of size $k_t$ among $e_1, \dots, e_{t-1}$, $m_{k_t}(t)$ is precisely $m_{k_t} - l_t k_t$. On the other hand, $s_{k_t}(t)\leq s_{k_t} - l_t$, where equality holds only  if all the partial edges of size $k$ among $e_1, \dots, e_{t-1}$ are matched with partial edges outside $A$. With this, 
\begin{equation*}
  p_t \leq 
  \left ( \frac{k_t s_{k_t}}{m_{k_t}} \right)^{k_t-1}.
\end{equation*}
This is because, if $k_t s_{k_t} \leq m_{k_t}$, the even
stronger inequality $p_t \leq \left ( \frac{s_{k_t}}{m_{k_t}} \right)^{k_t-1}$
holds, while
otherwise the RHS is at least 1, and the inequality is trivial. 
Furthermore, as we saw in the proof of Corollary~\ref{cor:conf-model-convergence-finite-index-set}, there exists $\alpha >0$ such that for $n$ large enough  and all $k \in I$,  $m_k \geq n \alpha$. This together with the fact that $s_k \leq s$ for all $k \in I$ implies
\begin{equation*}
  p_t \leq k_\text{max}^{k_\text{max}-1} \left ( \frac{s}{n \alpha} \right)^{k_\text{min} - 1},
\end{equation*}
where $k_\text{max}$ denotes the maximum element in $I$. 
As this upper bound is a constant, and the above process can continue for at most $s$ steps until we match all partial edges in $A$, the number of edges with all endpoints in $S$ is stochastically dominated by $\text{Binomial}(s,  k_\text{max}^{k_\text{max}-1} \left ( \frac{s}{n \alpha} \right)^{k_\text{min} - 1})$. 
The proof is complete by setting $c_4 := k_\text{max}^{k_\text{max}-1} / \alpha^{k_\text{min}-1}$. 
\end{proof}

\begin{proof}[Proof of Lemma~\ref{lem:maxload-tozero}]
  For positive integers $L$ and $r$,
let $X^{(n)}_{L, r}$ denote the number of subsets $S \subset \{1, \dots, n\}$ with size $L$ such that $H_n^e$ has at least $r$ many edges with all endpoints in $S$. 
With this, $\ev{Z^{(n)}_{\delta, t}} = \sum_{L=1}^{\lfloor n\delta \rfloor} \ev{X^{(n)}_{L, \lceil Lt \rceil}}$. Now, fix integers $L$ and $r$ such that $ L \leq n \delta$, with $\delta > 0$ sufficiently small.  Let  $\mathcal{S}_L$ denote the set of $S \subset \{1, \dots, n\}$ with size equal to $L$. Using Lemma~\ref{lem:maxload-binom} and the fact that for a binomial random variable $Z$, $\pr{Z \geq r} \leq (\ev{Z})^r / r!$, we have 
  \begin{align*}
    \ev{X^{(n)}_{L, r}} &\leq \sum_{S \in \mathcal{S}_L} \frac{1}{r!} \left ( c_4 \frac{\left ( \sum_{i \in S} \snorm{\gn_i}_1\right)^{k_\text{min}}}{n^{k_\text{min}-1}} \right )^r.
\intertext{Using the inequality $x^m \leq m! e^x$ which holds for $x \geq 0$ and integer $m$, this yields}
\ev{X^{(n)}_{L, r}} &\leq \sum_{S \in \mathcal{S}_L} \frac{(rk_\text{min})! c_4^r}{r! \theta^{r k_\text{min}} n^{r(k_\text{min}-1)}} \prod_{i \in S} e^{\theta \snorm{\gn_i}_1} \\
&\leq \frac{(rk_\text{min})! c_4^r}{r! \theta^{r k_\text{min}} n^{r(k_\text{min}-1)}} \frac{1}{L!} \left ( \sum_{i=1}^n e^{\theta \snorm{\gn_i}_1} \right )^L,
\intertext{where $\theta > 0$ is as in the statement of Corollary
\ref{cor:conf-model-convergence-finite-index-set}.
Using the assumption \eqref{eq:exp-theta-mgf-finite}, there exists $\lambda >0$ such that $\sum_{i=1}^n e^{\theta \snorm{\gn_i}_1} < n \lambda$ for all $n$. Using this, together with the inequalities $L! \geq (L/e)^L$ and $(rk_\text{min})!/r! \leq (rk_\text{min})^{r(k_\text{min}-1)}$, and rearranging the terms, this yields }
\ev{X^{(n)}_{L, r}} &\leq \left ( \frac{c_5 r}{n} \right )^{r(k_\text{min}-1)} \left ( \frac{e n \lambda}{L} \right )^L,
  \end{align*}
where $c_5 := k_\text{min} c_4^{\frac{1}{k_\text{min}-1}}/ \theta^{\frac{k_\text{min}}{k_\text{min}-1}}$. 
Note that we may assume $c_5 > 1$; otherwise, we may replace it with $c_5 \vee 1$ which makes this quantity even bigger; for, the exponent of $c_5$ is positive.

Using  this bound for $r =  \lceil L t \rceil$, we have 
\begin{align*}
  \ev{X^{(n)}_{L, \lceil Lt \rceil}} &\leq \left ( \frac{c_5 \lceil Lt \rceil }{n} \right)^{\lceil Lt \rceil (k_\text{min}-1)} \left ( \frac{e n \lambda}{L} \right)^L \\
&= \left ( \frac{c_5 \lceil Lt \rceil }{L} \right)^{\lceil Lt \rceil (k_\text{min}-1)} (e\lambda)^L \left ( \frac{L}{n} \right)^{\lceil Lt \rceil (k_\text{min}-1) - L}.
\end{align*}
Using $c_5 > 1$, $L/ n < 1$, $Lt \leq \lceil Lt \rceil \leq L(t+1)$  and the assumption $t > 1/(k_\text{min}-1)$, we get the upper bound
\begin{equation*}
  \ev{X^{(n)}_{L, \lceil Lt \rceil}}  \leq f\left( \frac{L}{n} \right)^L,
\end{equation*}
where
\begin{equation*}
  f(x) := c_6 x^{t(k_\text{min}-1)-1},
\end{equation*}
with $  c_6 := e \lambda (c_5(t+1))^{(t+1)(k_\text{min}-1)}$. 
Note that $f(L/n)^L = \exp(-n g(L/n))$, where
\begin{equation*}
  g(x) := - x \log c_6 -(t(k_\text{min}-1) - 1) x \log x.
\end{equation*}
As $t > 1/(k_\text{min}-1)$, 
there exists $a >0$ such that $g(x)$ is strictly increasing and strictly positive in $(0,a)$. Now, we choose $\delta > 0$ such that $\delta <a$ and also $f(\delta) < 1$. This is possible since the assumption $t > 1/(k_\text{min}-1)$ guarantees $f(\delta) \rightarrow 0$ as $\delta \downarrow 0$. With this, for any  $0<\zeta < \delta$, we have 
\begin{align*}
  \ev{Z^{(n)}_{\delta,t}} &\leq \sum_{L=1}^{\lfloor n \delta \rfloor} f(L/n)^L \\
  &= \sum_{L=1}^{\lfloor n \zeta \rfloor} f(L/n)^L + \sum_{L= \lfloor n \zeta \rfloor +1}^{\lfloor n \delta \rfloor} \exp(-ng(L/n)).
\intertext{Using the facts that $f$ is increasing in $(0,\infty)$, $g$ is increasing in $(0,a)$ and $0<L/n\leq \delta <a$, we have}
\ev{Z^{(n)}_{\delta,t}} &\leq \left ( \sum_{L=1}^{\lfloor n \zeta \rfloor} f(\zeta)^L \right )+ n\delta \exp(- n g(\zeta)).
\intertext{But $f(\zeta) < f(\delta) < 1$. Therefore,}
\ev{Z^{(n)}_{\delta,t}} &\le \frac{f(\zeta)}{1 - f(\zeta)} + n \delta \exp(-ng(\zeta)).
\end{align*}
Now, by sending $n$ to infinity, the second term vanishes, because $g(\zeta)>0$, and we get $\limsup \ev{Z^{(n)}_{\delta,t}} \leq f(\zeta) / (1-f(\zeta))$. Furthermore, by sending $\zeta\rightarrow 0$, we get $\ev{Z^{(n)}_{\delta,t}} \rightarrow 0$ which means $\pr{Z^{(n)}_{\delta, t} > 0} \rightarrow 0$, as $Z^{(n)}_{\delta,t}$ is integer valued. 
\end{proof}


\appendix

\section{Weak uniqueness of balanced allocations}
\label{sec:weak-uniq-balan}

\begin{proof}[Proof of Proposition~\ref{prop:weak-uniqueness}]
  For a fixed $\delta> 0$, define the set 
  \begin{equation*}
    A_\delta := 
    \{i \in V(H): \pbt(i) - \pbt'(i) > \delta \}.
  \end{equation*}
By assumption, we have $\sum_{i \in V(H)} |\pbt(i) - \pbt'(i) | < \infty$.
Hence, $A_\delta$ is a finite set. Moreover
\begin{equation}
  \label{eq:sum-A-delta-e-in-out}
  \sum_{i \in A_\delta} \pbt(i) - \pbt'(i) 
  = \sum_{i \in A_\delta} \sum_{e \ni i, e \nsubseteq A_\delta} 
  \theta(e ,i) - \theta'(e ,i).
\end{equation}
Now, fix some $e \in E(H)$ such that $e \cap A_\delta \neq \emptyset$ and $e \nsubseteq A_\delta $. For $i \in e \cap A_\delta$ and $j \in e \setminus A_\delta$, 
we have 
\begin{equation*}
  \pbt(j) - \pbt'(j)\leq \delta  < \pbt(i) - \pbt(i'),
\end{equation*}
which means that
\begin{equation*}
  \pbt(j) - \pbt(i) < \pbt'(j) - \pbt'(i).
\end{equation*}
Hence it is either the case that 
$\pbt'(j) > \pbt'(i)$ 
or $\pbt(j) < \pbt(i)$. 

If $\theta(e ,i) = 0$ for all $i \in e \cap A_\delta$, then 
$\sum_{i \in e \cap A_\delta} \theta(e ,i) - \theta'(e ,i) \leq 0$.
If $\theta(e ,i^*) \neq 0$ for some $i^* \in e \cap A_\delta$, 
then $\pbt(i^*) \le \pbt(j)$ for all $j \in e \setminus A_\delta$. Consequently,  $\pbt'(j) > \pbt'(i^*)$ for all $j \in e \setminus A_\delta$; thereby,
 $\theta'(e ,j) = 0$ for all $j \in e \setminus A_\delta$.
This means that $\sum_{i \in e \cap A_\delta} \theta'(e ,i) = 1 \geq \sum_{i \in e \cap A_\delta} \theta(e ,i)$.
Hence, we have observed that, in either case, we have
$\sum_{i \in e \cap A_\delta} \theta(e ,i) - \theta'(e ,i) \leq 0$. Since this is true for all $e$ such that $e \cap A_\delta \neq \emptyset$ and $e \nsubseteq A_\delta$, substituting into \eqref{eq:sum-A-delta-e-in-out} we realize that 
\begin{equation*}
  \sum_{i \in A_\delta} \pbt(i) - \pbt'(i) \leq 0.
\end{equation*}
On the other hand, $ \sum_{i \in A_\delta} \pbt(i) - \pbt'(i) \geq \delta |A_\delta|$. Combining these two we conclude that $A_\delta = \emptyset$. Symmetrically, the set $B_\delta:= \{ i \in V(H): \pbt'(i) - \pbt(i) > \delta \}$ should be empty. Since $\delta$ is arbitrary, we conclude that $\pt \equiv \pt'$, i.e. $\pbt \equiv \pbt'$, which completes the proof.
\end{proof}


\section{$\mbH_*(\Xi)$ and $\mbH_{**}(\Xi)$ are Polish spaces}
\label{sec:mh_-mh_-are}

In this section, we prove that $\mbH_*(\Xi)$ and $\mbH_{**}(\Xi)$ are Polish spaces when $\Xi$ is a Polish space. In particular, by setting $\Xi = \{ \emptyset \}$, this means that $\mH_*$ and $\mH_{**}$ are Polish spaces.

\begin{prop}
  \label{prop:H*-H**-Polish}
  Assume $\Xi$ is a Polish space. Then, $\mbH_*(\Xi)$ and $\mbH_{**}(\Xi)$ are Polish spaces.
\end{prop}

\begin{proof}
  We give the proof for $\mbH_*(\Xi)$ here. The proof for $\mbH_{**}(\Xi)$ is similar, and is therefore omitted.
  
 First, we show $\mbH_*(\Xi)$ is separable. Since $\Xi$ is separable, it has a countable dense subset $X = \{ \zeta_1, \zeta_2, \dots \} \subseteq \Xi$. Define $A_n$ to be the set of all hypergraphs with $n$ vertices with marks taking values in $X$, i.e.
  \begin{equation*}
    A_n := \{ [\bH,i] \in \mbH_*(\Xi): |V(\bH)| = n, \xi_{\bH}(e ,i) \in X \, \forall (e, i) \in \evpair(\bH) \}.
  \end{equation*}
Since there are finitely many hypergraphs on $n$ vertices and 
$X$ is countable, we see that $A_n$ is  countable. Now, define $A = \cup_n A_n$, which is countable. We claim that $A$ is dense in $\mbH_*(\Xi)$. To see this, for $[\bH, i] \in \mbH_*(\Xi)$ and $\epsilon > 0$ given, pick $(\bH, i) \in [\bH, i]$. Then take $n$ large enough such that $\frac{1}{1+n} < \epsilon$. With $H$ being the underlying unmarked hypergraph associated to $\bH$, we now define a marked rooted hypergraph $(\bH', i')$ where the underlying unmarked hypergraph $H'$ has the property that $(H', i')$ is the truncation of $(H, i)$ up to depth $n$ and the mark function $\xi_{\bH'}$ is defined as follows. For $(\tilde{e}, \tilde{i}) \in \evpair(H')$, define $\xi_{\bH'}(\tilde{e} ,\tilde{i}) \in X$ such that $d_\Xi(\xi_{H'}(\tilde{e} ,\tilde{i}), \xi_{\bH}(\tilde{e} ,\tilde{i})) < 1/(n+1)$. In this way, we have
\begin{equation*}
  \bar{d}_* ( [\bH, i], [\bH', i']) \leq \frac{1}{1+n} < \epsilon.
\end{equation*}
But, $(\bH', i')$ is finite, and the edge marks are in $X$; hence, $[\bH', i'] \in A$. Since $\epsilon$ was arbitrary, this shows that $A$ is dense in $\mbH_*(\Xi)$. Thus, $\mbH_*(\Xi)$ is separable.

Now, we turn to showing that $\mbH_*(\Xi)$ is complete. Take a Cauchy sequence $[\bH_n, i_n]$ in $\mbH_*(\Xi)$ and let $(\bH_n, i_n)$ be an arbitrary member of $[\bH_n, i_n]$.
 Without loss of generality, by taking a subsequence if needed, we can assume that for $m > n$ we have 
\begin{equation*}
  \bar{d}_*([\bH_n, i_n], [\bH_m, i_m]) < \frac{1}{1+n}.
\end{equation*}
This means that, with $H_k$ being the underlying unmarked hypergraph associated to $\bH_k$ for $k \geq 1$, we have 
\begin{equation}
  \label{eq:Hn-in--Hm-hm--distance}
(H_n, i_n) \equiv_n (H_m, i_m) \qquad \forall m > n,
\end{equation}
 and 
\begin{equation}
  \label{eq:marks-cauchy-distance}
  d_\Xi(\xi_{\bH_n}(e', i'), \xi_{\bH_m}(\phi_{n , m}(e') ,\phi_{n, m}(i'))) < \frac{1}{1+n},
\end{equation}
for all  $e' \in E_{H_n}(V^{H_n}_{i_n, n})$ and $i'\in e'$, where $\phi_{n, m}$ is the depth $n$ isomorphism between $H_n$ and $H_m$. Note that we can choose $\phi_{n, m}$ for $n>m$ so that
\begin{equation}
  \label{eq:phi-n-m--1-step}
  \phi_{n,m} = \phi_{m-1, m} \circ \dots \circ \phi_{n, n+1}.
\end{equation}
In fact, since the RHS of \eqref{eq:phi-n-m--1-step} defines a depth $n$ isomorphism from $(H_n, i_n)$ to $(H_m, i_m)$, one can define $\phi_{n,m}$ in this way.

In view of \eqref{eq:Hn-in--Hm-hm--distance}, we can 
construct a rooted hypergraph $(H, i)$ so that $(H, i) \equiv_n (H_n, i_n)$ for all $n$. 
Further, there are depth $n$ isomorphisms, $\phi_n$,  
from $(H, i)$ to $(H_n, i_n)$, which satisfy the consistency condition
\begin{equation}
  \label{eq:phi-m--phi-n-m-phi-n}
  \phi_m = \phi_{n, m} \circ \phi_n \qquad \forall m > n.
\end{equation}

So far we have constructed a rooted hypergraph $(H, i)$ such that  $[H_n, i_n] \rightarrow [H, i]$. Now, we construct a marked rooted hypergraph $(\bH, i)$, where its underlying unmarked rooted hypergraph is $(H,i)$, and the mark function $\xi_{\bH}: \evpair(H) \rightarrow \Xi$ is defined as follows. Take $(e', i') \in \evpair(H)$ and 
 choose $d$ such that $e' \in E_H(V^H_{i, d})$. We claim that the sequence 
\begin{equation*}
  \{ \xi_{\bH_n}(\phi_n(e'), \phi_n(i')) \}_{n \geq d},
\end{equation*}
is Cauchy in $\Xi$. 
Indeed, for $m > n$, using \eqref{eq:marks-cauchy-distance}, we have
\begin{equation*}
  d_\Xi(\xi_{\bH_n}(\phi_n(e'), \phi_n(i')), \xi_{\bH_m}(\phi_{n,m}\circ\phi_n(e'), \phi_{n,m}\circ \phi_n(i'))) < \frac{1}{1+n}.
\end{equation*}
Using \eqref{eq:phi-m--phi-n-m-phi-n}, this means
\begin{equation*}
  d_\Xi(\xi_{\bH_n}(\phi_n(e'), \phi_n(i')), \xi_{\bH_m}(\phi_m(e'), \phi_m(i'))) < \frac{1}{1+n},
\end{equation*}
which means that the sequence is Cauchy in $\Xi$. Since $\Xi$ is complete, we can define $\xi_{\bH}(e' ,i')$ to be the limit of this sequence. 

Now, we show that $[\bH_n, i_n] \rightarrow [\bH, i]$. 
For a given $d \in \nats$, define
\begin{equation*}
  A_d := \{ (e',i'): e' \in E_H(V^{H}_{i,d}), i' \in e' \}.
\end{equation*}
Since $H_n$ are locally finite, $H$ is also locally finite, and thus $A_d$ is finite. On the other hand, since $\xi_{\bH_n}(\phi_n(e') ,\phi_n(i')) \rightarrow \xi_{\bH}(e' ,i')$ for all $(e', i') \in A_d$, there exists a $N_d>d$ such that for all $n > N$, we have 
\begin{equation*}
  d_\Xi(\xi_{\bH_n}(\phi_n(e'), \phi_n(i')), \xi_{\bH}(e' ,i')) < \frac{1}{1+d} \qquad \forall (e', i') \in A_d.
\end{equation*}
Moreover, since $n > N_d > d$, $[H_n , i_n] \equiv_d [H, i]$. Hence,
\begin{equation*}
  \bar{d}_* ([\bH_n, i_n], [\bH, i]) < \frac{1}{1+d} \qquad \forall n > N_d.
\end{equation*}
Since $d$ was arbitrary, this means that $[\bH_n, i_n] 
\rightarrow [\bH, i]$ and $\mbH_*(\Xi)$ is complete.
\end{proof}

As was mentioned earlier, by setting 
$\Xi = \{ \emptyset \}$, we conclude that $\mH_*$ and $\mH_{**}$ are Polish 
spaces. This is explicitly stated below as a corollary.

\begin{cor}
  \label{cor:H*-H**-Polish}
  The spaces $\mH_*$ and $\mH_{**}$ are Polish spaces.
\end{cor}


\section{Some properties of measures on $\mH_*$}
\label{sec:prop-unim-meas}

\begin{proof}[Proof of Lemma~\ref{lem:A-as-Atilde-ae}]
  \underline{Part $(i)$}: We have 
  \begin{equation*}
    \vmu(\tilde{A}) = \int \oneu{\tilde{A}} d\vmu = \int \partial \oneu{\tilde{A}} d\mu.
  \end{equation*}
  But,
  \begin{equation*}
    \partial \oneu{\tilde{A}} (H, i) = \sum_{e \ni i} \oneu{\tilde{A}}(H, e, i) = \sum_{e \ni i} \oneu{A}(H, i) = \deg_H(i) \oneu{A}(H, i).
  \end{equation*}
Hence
\begin{equation*}
 \vmu(\tilde{A}) = \int \oneu{A}(H, i) \deg_H(i) d \mu = \int \deg_H(i) d \mu = \int d \vmu = \vmu(\mH_{**}),
\end{equation*}
which shows that $\tilde{A}$ happens $\vmu$ almost everywhere and the proof is complete. 

\underline{Part $(ii)$}: Note that 
\begin{equation*}
  \vmu(B) = \int_{\mH_{**}} \one{[H, e, i] \in B} d \vmu([H, e, i]) = \int_{\mH_*} \sum_{e \ni i} \one{[H, e, i] \in B} d \mu([H, i]). 
\end{equation*}
On the other hand, $\vmu(B) = \vmu(\mH_{**}) = \deg(\mu) = \int \deg_H(i) d \mu([H, i])$. Moreover, for all $[H, i] \in \mH_*$, $\sum_{e \ni i} \one{[H, e, i] \in B} \leq \deg_H(i)$. Consequently, it must be the case that for $\mu$--almost all $[H, i] \in \mH_*$, $\sum_{e \ni i} \one{[H, e, i] \in B} = \deg_H(i)$, or equivalently $[H, e, i] \in B$ for all $e \ni i$. 
\end{proof}

\begin{proof}[Proof of Lemma~\ref{lem:a.s.in-mu--a.e.in-vmu}]
  If we define
  \begin{equation*}
    A := \{ [H, i] \in \mH_*: f_k([H, i]) \rightarrow f_0([H, i]) \},
  \end{equation*}
  and 
  \begin{equation*}
    \tilde{A} := \{ [H, e, i] \in \mH_{**} : \tilde{f}_k([H, e, i]) \rightarrow \tilde{f}_0([H, e, i]) \},
  \end{equation*}
  then we have 
  \begin{equation*}
    \tilde{A} = \{ [H, e, i] \in \mH_{**} : [H, i] \in A \}.
  \end{equation*}
Then the proof is 
an immediate conseqeunce of
Lemma~\ref{lem:A-as-Atilde-ae}.
\end{proof}

\begin{proof}[Proof of Lemma~\ref{lem:vmu-a.e.-del-mu-a.s.}]
Define $B:= \{[H,e, i] \in \mH_{**}: f_k (H, e, i) \rightarrow f_0(H, e, i)\}$. As $\vmu(B) = \vmu(\mH_{**})$, from part $(ii)$ of Lemma~\ref{lem:A-as-Atilde-ae}, for $\mu$--almost all $[H, i] \in \mH_*$, for all $e \ni i$, $f_k(H, e, i) \rightarrow f_0(H, e, i)$. This in particular implies that for $\mu$--almost all $[H, i] \in \mH_*$, $\partial f_k(H, i) \rightarrow \partial f_0(H, i)$.
\end{proof}




\section{Proof of Lemma~\ref{lem:eq-condition-local-weak-convergence}}
\label{sec:proof-lemma-equivalent-condition-for-lwc}

We first prove that if the condition mentioned in Lemma~\ref{lem:eq-condition-local-weak-convergence} is satisfied, then $\mu_n \Rightarrow  \mu$. 
Fix $\epsilon > 0$. 
Let $f: \mH_* \rightarrow \reals$ be a uniformly continuous and bounded function.
There is some $\delta > 0$ such that $|f([H, i]) - f([H', i'])| < \epsilon$ when $d_{\mH_*}([H, i], [H', i'])< \delta$. Now choose $d$ such that $1/(1+d) < \delta$. 
For all rooted trees $[H, i] \in \mT_*$, $[H, i] \in A_{(H,i)_d}$.
Hence, one can find countably many rooted trees $\{ T_j, i_j \}_{j = 1}^\infty$ with depth at most $d$ such that $A_{(T_j, i_j)}$ partitions $\mT_*$; hence, one can find finitely many $(T_j, i_j)$, $1 \leq j \leq m$ such that $\sum_{j=1}^m \mu(A_{(T_j, i_j)}) \geq 1 - \epsilon$. If $\mA$ denotes $\cup_{j=1}^m A_{(T_j, i_j)}$, then we have 
\begin{equation*}
  \begin{split}
    \left | \int f d \mu - \sum_{j=1}^m f([T_j, i_j]) \mu (A_{(T_j, i_j)}) \right | & \leq \sum_{j=1}^m \left | \int_{A_{(T_j, i_j)}} f d \mu - f([T_j, i_j]) \mu(A_{(T_j, i_j)})\right|  \\
&\qquad \qquad +\norm{f}_\infty \mu(\mA^c) \\
    &\leq \epsilon( 1 + \norm{f}_\infty),
  \end{split}
\end{equation*}
where the last inequality uses the facts that $\mu(\mA^c) < \epsilon$ and $|f([H, i]) - f([T_j, i_j])| < \epsilon$ for $[H, i] \in A_{(T_j, i_j)}$, $1 \leq j \leq m$ since $1 / (1+d) < \epsilon$. Similarly, we have 
\begin{equation*}
\begin{split}
    \left | \int f d \mu_n - \sum_{j=1}^m f([T_j, i_j]) \mu(A_{(T_j, i_j)}) \right | &\leq \left | \int f d \mu_n - \sum_{j=1}^m f([T_j, i_j]) \mu_n(A_{(T_j, i_j)}) \right | \\
&\qquad + \sum_{j=1}^m |f(T_j, i_j)| | \mu_n(A_{(T_j, i_j)}) - \mu(A_{(T_j, i_j)}) | \\
    &\leq \norm{f}_\infty \left ( 1 - \sum_{j=1}^m \mu_n(A_{(T_j, i_j)}) \right ) + \epsilon \\
    &\qquad +\norm{f}_\infty\sum_{j=1}^m | \mu_n(A_{(T_j, i_j)}) - \mu(A_{(T_j, i_j)})|.
\end{split}
\end{equation*}
Combining these two, we have 
\begin{align*}
  \left | \int f d \mu_n - \int f d \mu \right | &\leq \norm{f}_\infty \left ( 1 - \sum_{j=1}^m \mu_n(A_{(T_j, i_j)}) \right )  + \norm{f}_\infty \left | \sum_{j=1}^m \mu_n(A_{(T_j, i_j)}) - \mu(A_{(T_j, i_j)}) \right | \\
&\qquad+ \epsilon(2 + \norm{f}_\infty).
\end{align*}
Now, as $n$ goes to infinity, 
$\mu_n(A_{(T_j, i_j)}) \rightarrow \mu(A_{(T_j, i_j)})$ by assumption.
Also, $\mu(\mA^c) < \epsilon$.
 Thus, we have 
\begin{equation*}
  \limsup_{n \rightarrow \infty} \left | \int f d\mu_n - \int f d \mu \right | \leq \epsilon(2 + 2 \norm{f}_\infty).
\end{equation*}
Since $\norm{f}_\infty < \infty$ and this is true for any $\epsilon > 0$, we get $\int f d \mu_n \rightarrow \int f d \mu$; hence, $\mu_n \Rightarrow \mu$.

For the converse, note that $\oneu{A_{(T, j)}}$ is a continuous function since $\oneu{A_{(T, j)}}([H, i]) = \oneu{A_{(T, i)}}([H', j'])$ for $d_{\mH_*}([H, i], [H', i']) < 1 / (1+d)$.


\section{Some properties of Unimodular Measures} 
\label{sec:everything-shows-at}

First, we prove Proposition~\ref{prop:everyting-shows-at-the-root}.
Our proof depends on the following lemma:

\begin{lem}
  \label{lem:tau-tau1-tau2}
  Assume $\tau: \mH_{**} \rightarrow \reals$ is a measurable function and $\mu \in \mP(\mH_*)$ is a unimodular measure such that $\tau = 1$, $\vmu$--almost everywhere. Then, we have
  \begin{enumerate}
  \item With $\tau_1(H, e, i) := \one{\tau(H, e', i) = 1, \forall e' \ni i}$, it holds that $\tau_1 = 1$ $\vmu$--almost everywhere.
  \item With $\tau_2(H, e, i) := \one{\tau(H, e, i') = 1, \forall i' \in e}$, it holds that $\tau_2 = 1$ $\vmu$--almost everywhere.
  \end{enumerate}
\end{lem}
\begin{proof}
In order to prove the first part, note that from Lemma~\ref{lem:A-as-Atilde-ae} part $(ii)$, we have that for $\mu$--almost all $[H, i] \in \mH_*$, $\tau(H, e', i) = 1$ for all $e' \ni i$. Now, part $(i)$ of Lemma~\ref{lem:A-as-Atilde-ae} implies that for $\vmu$--almost all $[H, e, i] \in \mH_{**}$, $\tau(H, e, i)  = 1$ for all $e' \ni i$, which is precisely what we need to prove.


  For the second part, since $\mu$ is unimodular, we have 
  \begin{equation*}
    \int \oneu{\tau = 1} d \vmu = \int \nabla \oneu{\tau = 1} d \vmu = \int \frac{1}{|e|} \sum_{i' \in e} \oneu{\tau(H, e, i') = 1} d \vmu (H, e, i).
  \end{equation*}
  But since $\oneu{\tau = 1} =1 $ holds $\vmu$--almost everywhere and $0 \leq \frac{1}{|e|} \sum_{i' \in e} \oneu{\tau(H, e, i') = 1} \leq 1$, we conclude that 
  \begin{equation*}
    \frac{1}{|e|} \sum_{i' \in e} \oneu{\tau(H, e, i') = 1} = 1, \qquad \vmu\text{--a.e.}.
  \end{equation*}
  Since the summands are either zero or one, this means that $\tau(H, e, i') = 1$ for all $i' \in e$, $\vmu$--almost everywhere, which is what we wanted to show.
\end{proof}

\begin{proof}[Proof of Proposition~\ref{prop:everyting-shows-at-the-root}]
  Define 
\begin{equation}
  \label{eq:A-k-condition-satisfied-up-to-distance-k}
  A_k := \{ [H, e, i] \in \mH_{**} : \tau(H, e', i') = 1 \,\, \forall i' \in V(H): d_H(i, i') \leq k, \,\, \forall  e' \ni i' \}.
\end{equation}
Note that 
$A_0 = \{ [H, e, i] : \tau(H, e, i) = 1 \}$, for which 
it is known from the assumption that $\vmu(A_0) = \vmu(\mH_{**})$. Now, we want to show that $\vmu(A_k) = \vmu(\mH_{**})$ for all $k \ge 0$.
We will do this
by induction on $k$. Assume that $\vmu(A_k) = \vmu(\mH_{**})$. Hence,
with $\phi_k(H, e, i) := \one{(H, e, i) \in A_k}$, we have $\phi_{k} = 1$ $\vmu$--almost everywhere. Now, we will use Lemma~\ref{lem:tau-tau1-tau2} to prove that $\phi_{k+1} = 1$ $\vmu$--almost everywhere. 

To do so, using part 2 of Lemma~\ref{lem:tau-tau1-tau2}, if we define 
\begin{equation*}
  B^k_1 := \{ [H, e, i] : \forall i' \in e, \,\phi_k(H, e, i') = 1 \},
\end{equation*}
then we know $\vmu(B^k_1) = \vmu(\mH_{**})$. Then, applying part 1 of Lemma~\ref{lem:tau-tau1-tau2} for the function $\oneu{B^k_1}$, we get that 
\begin{equation*}
  B^k_2 := \{ [H, e, i] : \forall e' \ni i, [H, e', i] \in B^k_1 \},
\end{equation*}
has the property that $\vmu(B^k_2) = \vmu(\mH_{**})$. On the other hand, 
\begin{equation*}
  B^k_2 = \{ [H, e, i] :  \forall e' \ni i,\,\, \forall i' \in e', [H, e', i'] \in A_k \} = A_{k+1}.
\end{equation*}
Hence, we have proved that $\vmu(A_{k+1}) = \vmu(\mH_{**})$,
which is the inductive step.

Thus, $\vmu( \cap_{k \in \nats} A_k) = \vmu(\mH_{**})$. Using the fact that  the vertex set is countable, and 
that the hypergraphs corresponding to the elements of
 $\mH_*$ are connected, the property $\tau$ holds for all the directed edges $\vmu$--almost everywhere, in $A := \cap_{k \in \nats} A_k$. Hence, the proof is complete.
\end{proof}

Now, we prove that the local weak limit of finite marked hypergraphs is a unimodular probability distribution on $\mbH_*(\Xi)$.  By setting $\Xi = \{ \emptyset \}$ in the following proposition, 
we can conclude that the local weak limit of finite simple hypergraphs  is unimodular, as claimed in Section~\ref{sec:unimodularity}.

\edit
\begin{prop}
  \label{prop:uimoularity-of-networks-weak-limit}
  Assume $\{\bH_n\}$ is a sequence of finite marked hypergraphs,
  with $\xi_{\bH_n} = \xi_n$ the associated edge mark functions,
   taking values in some metric space $\Xi$.
  Now, if 
\begin{equation*}
\bmu_n := u_{\bH_n} = \frac{1}{|V(\bH_n)|} \sum_{i \in V(\bH_n)} \delta_{[(\bH_n,i), i]},
\end{equation*}
then $\bmu_n \in \mP(\mbH_*(\Xi))$  is unimodular for each $n$. Moreover, if $\bmu_n$ converge weakly to some limit $\bmu \in \mP(\mbH_*(\Xi))$ such that $\deg (\bmu) < \infty$, $\bmu$ is also unimodular.
\end{prop}

\edit
\begin{proof}
  First, we show that $\bmu_n$ is unimodular for each $n$. For this, take a Borel function $f:\mbH_{**}(\Xi) \rightarrow [0,\infty)$ and note that 
  \begin{equation*}
    \begin{split}
      \int f(\bH, e, i) d \vbmu_n &= \int \partial f (\bH, i) d \bmu_n \\
      &= \frac{1}{|V(\bH_n)|} \sum_{i \in V(\bH_n)} \partial f(\bH_n, i) \\
      &= \frac{1}{|V(\bH_n)|} \sum_{i \in V(\bH_n)} \sum_{e \ni i} f(\bH_n, e, i) \\
      &= \frac{1}{|V(\bH_n)|} \sum_{e \in E(\bH_n)} \sum_{i \in e} f(\bH_n, e, i) \\
      &= \frac{1}{|V(\bH_n)|} \sum_{e \in E(\bH_n)} \sum_{i \in e} \nabla f(\bH_n,e , i) \\
      &= \int \nabla f(\bH, e, i) d \vbmu_n.
    \end{split}
  \end{equation*}
  Since this holds for all nonnegative Borel functions $f$, $\bmu_n$ is unimodular, by definition.

Now, for each $n$, define measures $\vbmu^{(1)}_n$ and $\vbmu^{(2)}_n$ on $\mbH_{**}(\Xi)$ so that for any Borel function $f: \mbH_{**}(\Xi) \rightarrow [0,\infty)$, we have 
\begin{equation}
\label{eq:vbmu-1-n}
  \int f d \vbmu^{(1)}_n := \int \sum_{e \ni i} f(\bH, e, i) d \bmu_n([\bH, i]),
\end{equation}
and 
\begin{equation}
\label{eq:vbmu-2-n}
  \int f d \vbmu^{(2)}_n := \int \sum_{e \ni i} \frac{1}{|e|} \sum_{j \in e} f(\bH, e, j) d \bmu_n([\bH, i]).
\end{equation}
We also define $\vbmu^{(1)}$ and $\vbmu^{(2)}$ for $\bmu$ in a similar fashion. Note that the  the RHS of~\eqref{eq:vbmu-1-n} is $\int \partial f d \bmu_n = \int f d \vbmu_n$. Therefore,  $\vbmu^{(1)}_n = \vbmu_n$. Also, the RHS of \eqref{eq:vbmu-2-n} is $\int \partial \nabla f d \bmu_n = \int \nabla f d \vbmu_n $.  Hence, $\int f d \vbmu^{(2)}_n = \int \nabla f d \vbmu_n$. Since we have shown that $\bmu_n$ is unimodular, this implies that $\vbmu^{(1)}_n = \vbmu^{(2)}_n$.

Now, we claim that $\vbmu^{(1)} = \vbmu^{(2)}$. To show this, take a bounded continuous function $f: \mbH_{**}(\Xi) \rightarrow \reals$. For $k > 0$, 
define
\begin{equation*}
  f_k(\bH, e, i) :=
  \begin{cases}
    f(\bH,e, i) & \mbox{ if $\deg_{H}(i) \leq k $},\\
    0 & \text{otherwise.}
  \end{cases}
\end{equation*}
It is easy to see that $f_k$ is continuous for each $k$, as $f$ is continuous. Moreover, $\partial f_k$  and $\partial \nabla f_k$ are bounded for each $k$. This, together with 
$\bmu_n \Rightarrow \bmu$, implies that 
\begin{equation*}
  \int f_k d \vbmu^{(1)}_n = \int \partial f_k d \bmu_n \rightarrow \int \partial f_k d \bmu = \int f_k d \vbmu^{(1)},
\end{equation*}
and 
\begin{equation*}
  \int f_k d \vbmu^{(2)}_n = \int \partial \nabla f_k d \bmu_n \rightarrow \int \partial \nabla f_k d \bmu = \int f_k d \vbmu^{(2)}.
\end{equation*}
This, together with the fact that $\vbmu^{(1)}_n = \vbmu^{(2)}_n$, implies that for all $k$
\begin{equation*}
  \int f_k d \vbmu^{(1)} = \int f_k d \vbmu^{(2)}.
\end{equation*}
Note that  as all hypergraphs are locally finite, $f_k \rightarrow f$ pointwise. Thus, sending $k$ to infinity and using the dominated convergence theorem, we have that for any bounded continuous function $f: \mbH_{**}(\Xi) \rightarrow [0,\infty)$,
\begin{equation*}
  \int f d \vbmu^{(1)} = \int f d \vbmu^{(2)}.
\end{equation*}
Since $f$ can be an arbitrary bounded continuous function, 
we have $\vbmu^{(1)} = \vbmu^{(2)}$. But the definition of $\vbmu^{(1)}$ and $\vbmu^{(2)}$ then implies that for any nonnegative Borel function $f$ we have $\int f d \vbmu = \int \nabla f d \vbmu$, which means that $\bmu$ is unimodular. 
\end{proof}


\section{Proof of unimodularity of $\ugwt(P)$}
\label{sec:unimodularity-ugwtp}

Here we show that $\ugwt(P)$ is unimodular. Before that we prove the following lemma, which is useful in calculating $\int f d \mu$ when $\mu = \ugwt(P)$.

\begin{lem}
  \label{lem:int-f-dvmu-gen-UGWT}
  Let $P \in \mP(\natszf)$, and $\Gamma$ a random variable with law $P$,
  where $\ev{\Gamma(k)} < \infty$ for all $k \geq 2$. Moreover, let $\mu = \ugwt(P)$. Then, for any Borel function $f: \mH_{**} \rightarrow [0, \infty)$, we have 
  \begin{equation*}
    \int f d \vmu = \sum_{k=2}^\infty \ev{\Gamma(k)} \sum_{\gamma \in \natszf} \hat{P}_k(\gamma) \ev{f(T, (k, 1), \emptyset) | \Gamma_\emptyset = \gamma+\typee_k},
  \end{equation*}
where the expectation is with respect to the random rooted hypertree of Definition~\ref{def:gwt_k(P)}.
  Here, $\typee_k \in \natszf$ is such that $\typee_k(k) = 1$ and $\typee_k(j) = 0$ for $j \neq k$.
\end{lem}


\begin{proof}
  Due to the definition of $\vmu$, we have 
  \begin{equation*}
    \int f d \vmu = \int \partial f d \mu = \ev{ \sum_{k=2}^{h(\Gamma_\emptyset)} \sum_{i=1}^{\Gamma_\emptyset(k)} f(T, (k, i), \emptyset)}.
  \end{equation*}
Since $\Gamma_\emptyset$ has distribution $P$, we have 
\begin{equation*}
  \ev{ \sum_{k=1}^{h(\Gamma_\emptyset)} \sum_{i=1}^{\Gamma_\emptyset(k)} f(T, (k, i), \emptyset)} = \sum_{\gamma \in \natszf} P(\gamma) \ev{ \sum_{k=2}^{h(\gamma)} \sum_{i=1}^{\gamma(k)} f(T, (k, i), \emptyset) \Bigg | \Gamma_\emptyset = \gamma}.
\end{equation*}
Now, due to symmetry, conditioned on $\Gamma_\emptyset = \gamma$, for a given $k \leq h(\gamma)$, all $f(T, (k, i), \emptyset)$ for $1 \leq i \leq \gamma(k)$ have the same distribution, hence 
\begin{align*}
  \int f d \vmu &= \sum_{\gamma \in \natszf} P(\gamma) \sum_{k=2}^{h(\gamma)} \gamma(k) \ev{f(T, (k, 1), \emptyset)| \Gamma_\emptyset = \gamma} \\
&= \sum_{\gamma \in \natszf} P(\gamma) \sum_{k=2}^{\infty} \gamma(k) \ev{f(T, (k, 1), \emptyset)| \Gamma_\emptyset = \gamma},
\end{align*}
where the second equality uses the fact that $\gamma(k) = 0$ for $k > h(\gamma)$. Now, since all the terms are nonnegative, employing Tonelli's theorem to switch the order of integrals we have 
\begin{equation*}
  \int f d \vmu = \sum_{k=2}^\infty \sum_{\gamma \in \natszf: \gamma(k) > 0} P(\gamma) \gamma(k) \ev{f(T, (k, 1), \emptyset) | \Gamma_\emptyset = \gamma}.
\end{equation*}
Using the definition of $\hat{P}_k$, $P(\gamma) \gamma(k)$ is equal to $\ev{\Gamma(k)} \hat{P}_k(\gamma - \typee_k)$ for $\gamma \in \natszf$ such that $\gamma(k) > 0$, where $\typee_k \in \natszf$ is such that $\typee_k(k) = 1$ and $\typee_k(j) = 0$ for $j \neq k$. Hence, we have 
\begin{equation*}
\begin{split}
  \int f d \vmu &= \sum_{k=2}^\infty \ev{\Gamma(k)} \sum_{\gamma \in \natszf: \gamma(k) > 0} \hat{P}_k(\gamma - \typee_k) \ev{f(T, (k, 1), \emptyset) | \Gamma_\emptyset = \gamma} \\
  &= \sum_{k=2}^\infty \ev{\Gamma(k)} \sum_{\gamma \in \natszf} \hat{P}_k(\gamma) \ev{f(T, (k, 1), \emptyset) | \Gamma_\emptyset = \gamma+\typee_k}. \\
  \end{split}
\end{equation*}
We can interpret the last expression above as follows. In computing
$\int f d \vmu$, when we consider $f([T,e,\emptyset])$, the 
edge $e$ attached to the root $\emptyset$ of the tree $T$
is of size $k$ with 
probability $\Gamma(k)$. This explains the outer summation.
Since the contribution to the integral is the same whichever edge of
size $k$ connected to the root is picked, suppose the edge picked is 
the edge $(k,1)$. Then the type of the rest of the
edges connected to the root is given by $\hat{P}_k$, and so 
$\Gamma_\emptyset$ will equal $\gamma+\typee_k$ with probability 
$\hat{P}_k(\gamma)$. This explains the inner summation.
\end{proof}

Now we are ready to prove the unimodularity of $\ugwt(P)$:

\begin{proof}[Proof of Proposition~\ref{prop:ugwt-unimodular}]
  We need to prove that for a nonnegative measurable function $f: \mH_{**} \rightarrow [0,\infty)$ we have $\int f d \vmu = \int \nabla f d \vmu$. 
   Using Lemma~\ref{lem:int-f-dvmu-gen-UGWT}, we have 
  \begin{equation*}
    \begin{split}
      \int \nabla f d \vmu &\stackrel{(a)}{=} \sum_{k=2}^\infty \ev{\Gamma(k)} \ev{\nabla f(\tilde{T}_k, (k,1), \emptyset)}  \\
      &= \sum_{k=2}^\infty \ev{\Gamma(k)} \frac{1}{k} \left ( \ev{f(\tilde{T}_k, (k, 1), \emptyset)} + \sum_{i=1}^{k-1} \ev{f(\tilde{T}_k, (k, 1), (k,1, i))} \right ).
    \end{split}
  \end{equation*}
  Here, for $k \ge 2$, $\tilde{T}_k$ denotes a tree with root $\emptyset$
  that has an edge $(k,1)$ of size $k$ connected to the root, with the 
  type of the other edges connected to the root being $\hat{P}_k$, 
  and with the subtrees at the other vertices of all the edges (including the
  edge $(k,1)$) generated according to the rules of $\ugwt(P)$.
  Step (a) is justified because, for each $k \ge 2$, we have
  \[
  \ev{\nabla f(\tilde{T}_k, (k,1), \emptyset)} = \sum_{k=2}^\infty \ev{\Gamma(k)} \sum_{\gamma \in \natszf} \hat{P}_k(\gamma) \ev{\nabla f(T, (k, 1), \emptyset) | \Gamma_\emptyset = \gamma+\typee_k}. 
  \]
Now, because of the symmetry in the construction of $\tilde{T}_k$,
for each $i = 1, \ldots, k-1$, 
 $(\tilde{T}_k, (k, 1), (k, 1, i))$ has the same distribution as $(\tilde{T}_k, (k, 1), \emptyset)$. Hence, $\ev{f(\tilde{T}_k, (k, 1), (k,1, i))} = \ev{f(\tilde{T}_k, (k, 1), \emptyset)}$. Substituting and simplifying we get
  \begin{equation*}
    \int \nabla f d \vmu = \sum_{k=2}^\infty \ev{\Gamma(k)} \ev{f(\tilde{T}, (k, 1), \emptyset)} = \int f d \vmu,
  \end{equation*}
where the last equality again uses Lemma~\ref{lem:int-f-dvmu-gen-UGWT}. This completes the proof of the unimodularity of $\mu$.
\end{proof}


\section{Configuration model on hypergraphs and their local weak limit: Proof of Theorem~\ref{thm:conf-model-convegence-Un}}
\label{sec:proof-theorem-lws-conf-model}

In this section, we prove Theorem~\ref{thm:conf-model-convegence-Un}. We prove this  in two steps: first, in Section~\ref{sec:conv-expect}, we prove that $\ev{u_{H^e_n}}$ converges weakly to $\ugwt(P)$, where the expectation in $\ev{u_{H^e_n}}$ is taken with respect to the randomness in the construction of $H^e_n$. Later, in Section~\ref{sec:alsm-sure-conv}, we conclude the almost sure convergence by a concentration argument. See \cite{Bordenave14randomgraphs} for an argument on the local weak convergence of the configuration model in the graph regime.

Throughout this section, we employ the vertex and edge indexing notations $\nvertex$ and $\nedge$ defined in Section~\ref{sec:unim-galt-wats}. 
Recall that for $a \in \nvertex$ and integers $s \geq 2$, $e \geq 1$ and $1 \leq r \leq s-1$, $(a, s, e, r)$ is the element in $\nvertex$ obtained from $a$ by concatenating $s, e, r$ at the end of the string representing $a$. 
For an edge $e \in \nedge$ and an integer $r$, $(e, r)$ is defined similarly. 
Furthermore, $\mT(\nvertex, \nedge)$ in this section denotes the set of hypertrees with vertex set and edge set being subsets of $\nvertex$ and $\nedge$, respectively. Such a hypertree is treated to be rooted at $\emptyset$, unless otherwise stated. Moreover, for a sequence of types $\{ \gamma_a \}_{a \in \nvertex}$, $\mT(\{\gamma_a\}_{a \in \nvertex})$ denotes the tree in $\mT(\nvertex, \nedge)$ in which the type of each node $a \in \nvertex$ in the hypertree below that node is equal to $\gamma_a$ (recall Figure~\ref{fig:gamma-GWT} from Section~\ref{sec:unim-galt-wats} as an example). 

Before going through the proof, we need to define a procedure called the ``exploration process'' in Section~\ref{sec:exploration-process} below.

\subsection{Exploration process}
\label{sec:exploration-process}

Assume that a random hypergraph $H_n$ on the vertex set $\{1, \dots, n\}$ is obtained from a given type sequence $\gn = (\gn_1, \dots, \gn_n)$ satisfying \eqref{eq:gn-0-for-bigger-n} and \eqref{eq:k|sum-gn-k}. Note that, following our discussion in Section~\ref{sec:conf-model-hypergr}, $H_n$ is identified by a set of random matchings $\sigma_2, \dots, \sigma_n$. 
Here, we introduce a procedure that choses a node $v_0$  uniformly at random in $\{1, \dots, n\}$ and explores the local neighborhood of that node in a breadth--first manner. This process at each step produces a hypertree in $\mT(\nvertex, \nedge)$, which turns out to be locally isomorphic to the neighborhood of $v_0$ given that the local neighborhood of $v_0$ in $H_n$ is tree--like. 
A similar process in the graph regime is introduced in \cite{Bordenave14randomgraphs}.

Formally speaking, the exploration process starts at time $t=0$ with choosing a vertex $v_0 \in \{1,\dots, n\}$ uniformly at random as the root.
At each time $t\geq 1$, we have a node indexing function $\phi_t: \nvertex \rightarrow \{ 1, \dots, n\} \cup \{ \times \}$. 
To begin with, we define 
$\phi_1(\emptyset) = v_0$ and $\phi_1(\bi) = \times$ for $\emptyset \neq \bi \in \nvertex$. Also, at each time step $t$, we partition $\Delta^{(n)}$ into three sets: an ``active set'' $A_t$, a ``connected set'' $C_t$, and an ``undiscovered set'' $U_t$. These are initialized by setting $A_1 = \Delta^{(n)}_{v_0}$, $C_1 = \emptyset$ and $U_1 = \Delta^{(n)} \setminus A_1$. Moreover, at time $t$, $N_t \subset\{1, \dots, n\}$ contains the explored nodes at time $t$, and is initialized as $N_1 = \{ v_0 \}$. 

At time $t$, given the sets $A_t, U_t, C_t$ and $\phi_t$, we first form the  set $W_t :=\{(\bi, k, j): \phi_t(\bi) \neq \times, e^k_{\phi_t(\bi), j} \in A_t \}$.  
If $W_t$ is nonempty, we  define $(\bi_t, k_t, j_t)$ to be the lexicographically smallest element in $W_t$ and let $e_t := e^{k_t}_{\phi_t(\bi_t), j_t}$. In fact, $e_t$ is the partial edge chosen at time $t$ to be matched with other partial edges to form an edge. Now, define
\begin{equation*}
  e_{t, j} := \sigma_{k_t}^{(j-1)}(e_t) \qquad 1 \leq j \leq k_t - 1,
\end{equation*}
which are  the $k_t - 1$ other partial edges matched with $e_t$. 
Also, for $1 \leq j \leq k_t - 1$, let $u_{t, j} := \nu(e_{t, j})$ be the node associated to $e_{t, j}$. 
Moreover, we update the sets $C_t, A_t, U_t$ and $N_t$ as follows:
\begin{subequations}
\begin{align}
    C_{t+1} &= C_t \cup \{ e_t, e_{t, 1}, \dots, e_{t, k_t - 1} \}, \\
    A_{t+1} &= A_t \setminus \{ e_t, e_{t, 1}, \dots, e_{t, k_t - 1} \} \bigcup_{j=1}^{k_t  -1} \left ( \Delta^{(n)}_{u_{t,j}} \cap U_t \right ),\\
    U_{t+1} &= U_t \setminus \bigcup_{j=1}^{k_t - 1} \Delta^{(n)}_{u_{t, j}}, \\
    N_{t+1} &= N_t \cup \{ u_{t, 1}, \dots, u_{t, k_t - 1} \}.
\end{align}
\end{subequations}
In order to update $\phi_t$, define $\tilde{j}_t$ to be the minimum $j$ such that 
\begin{equation*}
  \phi_t((\bi_t, k_t, j, 1)) = \times.
\end{equation*}
With this,  set $\phi_{t+1}$ to be equal to $\phi_{t}$ except for the following values:
\begin{equation*}
  \phi_{t+1}((\bi_t, k_t, \tilde{j}_t, l)) = u_{t,l} \qquad 1 \leq l \leq k_t - 1.
\end{equation*}
This in particular means that the set of nodes in $\{1, \dots, n\}$ that appear in the range of $\phi_{t+1}$ is precisely $N_{t+1}$. 

At each time $t$, we define a rooted hypertree formed by the exploration process, which we denote by $R_t$, which is a member of $\mT(\nvertex, \nedge)$. $R_t$  is identified through the mapping $\phi_t$, i.e. its vertex set is $\{ \bi \in \nvertex: \phi_t(\bi) \neq \times \}$, and its edge set is $\{(\bi_r, k_r, \tilde{j}_r), 1 \leq r \leq t\}$. 
 This process continues until $A_t = \emptyset$, which results in exploring the connected component of $v_0$. Note that since the permutations determining the configuration process  are random, the exploration process is in fact a random process. Let $\mF_t$ be the sigma field generated by all the random variables defined above up to time $t$. 
 Let $\tau$ be the stopping time corresponding to $A_\tau = \emptyset$. For the sake of simplicity, for $t > \tau$, we define $R_t = R_\tau$.


\subsection{Convergence of expectation}
\label{sec:conv-expect}
In this section, we provide the proof of the convergence of $\ev{u_{H^e_n}}$. This is done in two parts. Loosely speaking, 
we first show that, for any integer $d \geq 1$, with high probability, $H_n$ rooted at a vertex chosen uniformly at random up to depth $d$ is simple. Then,  
we prove that the multihypergraph $H_n$ obtained from the configuration model is locally and asymptotically similar to the Galton--Watson process. These two together complete the proof of Theorem~\ref{thm:conf-model-convegence-Un}. The former is proved in Proposition~\ref{prop:Hn-v0-d-simple-converge-1} below and the latter in Proposition~\ref{prop:conf-model_unerased-convergence}.

Let $\bar{\mu}^e_n$ and $\mu$ denote $\ev{u_{H^e_n}}$ and  $\ugwt(P)$, respectively. In order to show that $\bar{\mu}^e_n \Rightarrow \mu$,  using Lemma~\ref{lem:eq-condition-local-weak-convergence}, it suffices to prove that for any $d \geq 1$, and with $(T, o)$ being a rooted hypertree of depth at most $d$, $\bar{\mu}^e_n(A_{(T, o)})\rightarrow \mu(A_{(T, o)})$, where we recall that $A_{(T, o)} = \{ [H, j] \in \mH_*: (H, j)_d \equiv (T, o) \}$.
Note that, with $v_0$ being chosen uniformly at random in $\{1, \dots, n\}$, we have 
\begin{equation}
  \label{eq:bar-mu-n-A-T-i--prob-Hn-v0-equiv-T-i}
  \begin{split}
    \bar{\mu}^e_n(A_{(T, o)}) &= \ev{\frac{1}{n} \sum_{i=1}^n \delta_{[H^e_n(i), i]}(A_{(T, o)})} \\
    &= \ev{\frac{1}{n} \sum_{i=1}^n \oneu{[H^e_n(i), i] \in A_{(T, o)}}} \\
    &= \frac{1}{n} \sum_{i=1}^n \pr{(H^e_n, i)_d \equiv (T,o)} \\
    &= \pr{(H^e_n, v_0)_d \equiv (T,o)}.
  \end{split}
\end{equation}
Thus, motivated by the above discussion, we need to show that for all $d\geq 1$ and $(T, o)$ with depth at most $d$, $\pr{(H^e_n, v_0)_d \equiv (T,o)} \rightarrow \mu(A_{(T, o)})$.
We prove this in two steps. First, we prove in Proposition~\ref{prop:Hn-v0-d-simple-converge-1} that the probability of $(H_n, v_0)_d$ being a simple hypertree goes to one as $n$ goes to infinity. Subsequently,  in Proposition~\ref{prop:conf-model_unerased-convergence}, we show that
$\pr{(H_n^e, v_0)_d \equiv (T,i)}$ converges to $\mu(A_{(T, i)})$. 

In the following statement, $\gn = (\gn_1, \dots, \gn_n)$ is a fixed type sequence and $H_n$ is the random multihypergraph resulting from the configuration model. Moreover, for an integer $d \geq 1$ and vertex $v \in \{1, \dots, n\}$, $(H_n, v)_d$ is the multihypergraph rooted at vertex $v$ containing nodes with distance at most $d$ from $v$ and edges in $H_n$ with all endpoints among this set of vertices. 

\begin{prop}
  \label{prop:Hn-v0-d-simple-converge-1}
Assume conditions~\eqref{eq:norm-1-gamma-log}, \eqref{eq:norm-infty-gamma-log} and \eqref{eq:delta-n-k-big-enough} are satisfied with constants $c_1, c_2, c_3, \epsilon>0$. 
Then, if 
$v_0$ is chosen uniformly at random from $\{1, \dots, n\}$,
  \begin{equation*}
    \lim_{n \rightarrow \infty} \pr{(H_n, v_0)_d \text{ is a simple hypertree} } = 1.
  \end{equation*}
\end{prop}

\begin{proof}
  Note that \eqref{eq:norm-1-gamma-log}  and \eqref{eq:norm-infty-gamma-log} imply that the degrees of all vertices and the sizes of all edges in $H_n$ are bounded by $\alpha_n := c_1 (\log n)^{c_2}$. Therefore, there are at most $\beta_n := (\alpha_n)^{2d+5}$ edges in $(H_n, v_0)_{d+1}$. Since at each step of the exploration process we form one edge, $(H_n, v_0)_{d+1}$ is completely observed up to step $\beta_n$. On the other hand, if at step $t$, the partial edge to be matched at that step, which is $e_t$ in our notation, is matched to partial edges outside $A_t$, all of the endpoints of the newly formed edge at step $t$ are in $N_t^c$, i.e. they are not observed so far. If in addition to this property, all the $k_t - 1$ vertices of these partial edges that are matched with $e_t$ are distinct, no loops or multiple edges are  formed at step $t$. If these properties hold for all $1 \leq t \leq \beta_n$, $(H_n,v_0)_d$ will be a simple  hypertree. Note that, in order to make sure that $(H_n, v_0)_d$ is a simple hypertree, we need to make sure that there is no loop in the exploration process up to depth $d+1$. This guarantees that even vertices at depth $d$ do not get connected to each other.

To formalize this, fix $1 \leq t \leq \beta_n$ and assume that the exploration process is not finished up to step $t$, and at step $t$, we need to match the partial edge $e_t$ of size $k_t$. Let $E_{t,1}$ denote the event that $e_{t,1} = \sigma_{k_t}^{(1)}(e_t) \in A_t\cap \Delta_{k_t}^{(n)}$. Moreover, for $2 \leq i \leq k_t -1$, let $E_{t,i}$ be te event that 
\begin{equation*}
  e_{t,i} = \sigma_{k_t}^{(i)}(e_t) \in (A_t \cap \Delta_{k_t}^{(n)})\cup \left ( \bigcup_{j=1}^{i-1} \Delta_{k_t}^{(n)}(v(e_{t,j})) \right ),
\end{equation*}
where we recall that $\Delta_{k_t}^{(n)}(v(e_{t,j}))$ denotes the set of partial edges of size $k_t$ connected to the vertex associated to $e_{t,j}$. Note that having chosen $e_{t, 1}, \dots, e_{t,i-1}$, there are $|\Delta_{k_t}^{(n)}| - |\Delta_{k_t}^{(n)}\cap C_t| - i$ many candidates for $\sigma_{k_t}^{(i)}(e_t)$, each having the same chance of being chosen. We claim that $|A_t\cap \Delta_{k_t}^{(n)}| \leq t \alpha_n^2$. The reason is that at each step in the exploration process, at most $\alpha_n$ many new vertices are added to $N_t$, each of whcih having at most $\alpha_n$ many partial edges of size $k_t$. On the other hand, $|\Delta_{k_t}^{(n)}(v(e_{t,j}))| \leq \alpha_n$. Consequently, 
\begin{equation*}
   \left | (A_t \cap \Delta_{k_t}^{(n)})\cup \left ( \bigcup_{j=1}^{i-1} \Delta_{k_t}^{(n)}(v(e_{t,j})) \right ) \right | \leq (t+1) \alpha_n^2.
\end{equation*}
Additionally, at each step, at most $\alpha_n$ partial edges are added to $C_t$ to form a edge; therefore, $|C_t| \leq t \alpha_n$. This, together with \eqref{eq:delta-n-k-big-enough}, implies that for $1 \leq i \leq k_t \leq \alpha_n$, we have $|\Delta_{k_t}^{(n)}| - |\Delta_{k_t}^{(n)}\cap C_t| - i \geq c_3 n^\epsilon - (t+1)\alpha_n$. Since each of these candidates have the same probability of being chosen, we have 
\begin{equation*}
  \pr{E_{t, i}} \leq \frac{(t+1)\alpha_n^2}{c_3 n^\epsilon - (t+1)\alpha_n} \qquad 1 \leq i \leq k_t - 1.
\end{equation*}
If $E_t$ denotes $\cup_{i=1}^{k_t - 1} E_{t, i}$, using $k_t \leq \alpha_n$ and the union bound, we have
\begin{equation*}
  \pr{\cup_{t=1}^{\beta_n}E_t} \leq \frac{(\beta_n+1)^2\alpha_n^3}{c_3n^\epsilon - (\beta_n +1)\alpha_n}.
\end{equation*}
Note that, $\alpha_n$ and $\beta_n$ scale logarithmically in $n$, and $d$ is fixed. Hence, due to the $c_3n^\epsilon$ term in the denominator, the above probability goes to zero as $n$ goes to infinity. As was discussed above, outside the event $\cup_{t=1}^{\beta_n}E_t$, the rooted hypergraph $(H_n, v_0)_d$ is a simple hypertree. 
\end{proof}

\begin{prop}
  \label{prop:conf-model_unerased-convergence}
With the assumptions of Theorem~\ref{thm:conf-model-convegence-Un}, for an integer $d \geq 1$ and a rooted hypertree $(T, o)$ with depth at most $d$, if $v_0$ is a node chosen uniformly at random from $\{1, \dots, n\}$, we have
  \begin{equation*}
    \lim_{n \rightarrow \infty} \pr{(H^e_n, v_0)_d \equiv (T,o)} = \mu(A_{(T, o)}).
  \end{equation*}
\end{prop}

Note that, the above proposition, together with the discussion in \eqref{eq:bar-mu-n-A-T-i--prob-Hn-v0-equiv-T-i}, implies that $\ev{u_{H^e_n}} \Rightarrow \mu$.  In Section~\ref{sec:alsm-sure-conv} below, we show that $u_{H^e_n}$ is concentrated around its mean, which completes the proof of Theorem~\ref{thm:conf-model-convegence-Un}.

\begin{proof}
Given the rooted hypertree $(T, o)$ with depth at most $d$, let $\mC_{[T, o]}$ denote the set of hypertrees $\tilde{T} \in \mT(\nvertex, \nedge)$ such that $(\tilde{T},\emptyset) \equiv (T, o)$. For $\tilde{T} \in \mC_{[T, o]}$, let $e^{\tilde{T}}_1, \dots, e^{\tilde{T}}_r$ be the edges in $\tilde{T}$ with depth at most $d-1$, ordered lexicographically in $\nedge$. Moreover, let $e^{\tilde{T}}_{r+1}, \dots, e^{\tilde{T}}_{r+l}$ denote the edges in $\tilde{T}$ with depth precisely $d$, ordered lexicographically (if there is no such edge, $l=0$). For a node $\bi$ in $\tilde{T}$, let $\gamma^{\tilde{T}}_{\bi}$ denote the type of the vertex $\bi$ in the subtree below $\bi$.
Furthermore, define 
  \begin{equation*}
    \pi_{\tilde{T}} := \pr{(\mT(\{\Gamma_a\}_{a \in \nvertex}),\emptyset)_d = (\tilde{T}, \emptyset) },
  \end{equation*}
under the probability in Definition~\ref{def:ugwt(P)}. With the above notation, we have 
\begin{equation}
\label{eq:pi-tT}
  \pi_{\tT} := P(\gamma^{\tT}_\emptyset) \prod_{t=1}^{r} \prod_{j = 1}^{|e^{\tT}_t| - 1} \hat{P}_{|e^{\tT}_t|} (\gamma^{\tT}_{(e^{\tilde{T}}_t, j)}).
\end{equation}
By the definition of  the distribution $\ugwt(P)$, we have
\begin{equation*}
  \mu(A_{[T, o]}) = \sum_{\tilde{T} \in \mC_{[T, o]}} \pi_{\tilde{T}}.
\end{equation*}

Recall that $R_t$ denotes the hypertree in $\mT(\nvertex, \nedge)$ which results from the exploration process at step $t$.
Note that if $(H_n, v_0)_d$ is a simple hypertree, we have $(H_n, v_0)_d \equiv (R_{\beta_n},\emptyset)_d$ where $\beta_n$ is defined in Proposition~\ref{prop:Hn-v0-d-simple-converge-1} above. With this, we have 
\begin{align*}
  \pr{(H_n^e, v_0)_d \equiv (T, o)} &= \pr{(H^e_n, v_0)_d \equiv (T, o) \text{ and } (H_n, v_0)_d \text{ is a simple hypertree}} \\
  &\quad + \pr{(H^e_n, v_0)_d \equiv (T, o) \text{ and } (H_n, v_0)_d \text{ is not a simple hypertree}}
\end{align*}
As is shown in Proposition~\ref{prop:Hn-v0-d-simple-converge-1}, the second term converges to zero; therefore, we need to study only the first term. 
But, if $(H_n, v_0)_d$ is a simple hypertree, $(H_n^e, v_0)_d \equiv (T, o)$ if and only if $R_{r+l+1} = \tilde{T}$ for some $\tilde{T} \in \mC_{[T, o]}$. Consequently, it suffices for us to show that
\begin{equation}
\label{eq:lim-pr-R-t-l--pi-tilde-T}
  \lim_{n \rightarrow \infty} \pr{R_{r+l+1} = \tilde{T} \text{ and } (H_n, v_0)_d \text{ is a simple hypertree}} = \pi_{\tilde{T}} \qquad \forall \tilde{T} \in \mC_{[T, o]}.
\end{equation}
For $1 \leq t \leq r$, let $E_t$ be the event defined in Proposition~\ref{prop:Hn-v0-d-simple-converge-1}. Recall that $E_t^c$ is the event that $u_{t, 1}, \dots, u_{t, k_t-1}$ are all distinct and are not in $N_t$. From Proposition~\ref{prop:Hn-v0-d-simple-converge-1}, we know that the probabilities of  both the events ``$(H_n, v_0)_d \text{ is a simple hypertree}$'' and $\cap_{t=1}^r E_t^c$ converge to 1 as $n\rightarrow \infty$. Therefore, it suffices to show that 
\begin{equation}
\label{eq:lim-pr-R-t-l-E--pi-tilde-T}
  \lim_{n \rightarrow \infty} \pr{(R_{r+l+1} = \tilde{T}) \cap ( \cap_{t=1}^r E_t^c)} = \pi_{\tilde{T}} \qquad \forall \tilde{T} \in \mC_{[T, o]}.
\end{equation}

To prove this, fix some $\tilde{T} \in \mT(\nvertex, \nedge)$ with depth at most $d$ and define $S_0$ to be the event that $\gamma^{(n)}_{v_0} = \gamma^{\tilde{T}}_{\emptyset}$. From \eqref{eq:fraction-converge-P(gamma)}, $\pr{S_0} \rightarrow P(\gamma^{\tT}_\emptyset)$ as $n\rightarrow \infty$. Moreover, for $1 \leq t \leq r$, define $\tilde{S}_t$  to be the event that 
\begin{equation*}
  \gamma^{(n)}_{u_{t, s}} = \gamma^{\tilde{T}}_{(\bi_t, k_t, \tilde{j}_t, s)} + \typee_{k_t} \qquad 1 \leq s \leq k_t - 1,
\end{equation*}
and let $S_t = \tilde{S}_t \cap E_t^c$, which is in fact the intersection of $\tilde{S}_t$ and the event
\begin{equation*}
  u_{t, j}, 1 \leq j \leq k_t - 1 \text{ are distinct and are not in } N_t.
\end{equation*}
With this, let $S^t = \cap_{i=0}^t S_i$. We claim that the event 
$(R_{r+l+1} = \tilde{T}) \cap ( \cap_{t=1}^r E_t^c)$ coinsides with $S^r$. The reason is that on the event $S^r$, for each $1 \leq t \leq r$, the type of all the $k_t - 1$ subnodes of edge formed at step $t$ matches with that of $e^{\tilde{T}}_t$. 
Moreover, on the event  $\cap_{t=1}^r E_t^c$, there is no loops or cycles formed during the exploration process. In particular, those partial edges connected to the vertex $u_{t,s}$ which are added to the active set are not used until the process goes to vertex $u_{t, s}$ itself. 
Also, note that as  $\tilde{T}$ has depth at most $d$, its structure is determined by the type of the vertices of depth at most $d-1$, which are subnodes of edges of depth at most $d-1$ in $\tilde{T}$, which are precisely $e^{\tilde{T}}_1, \dots, e^{\tilde{T}}_r$. 

Now, we prove by induction that for $0 \leq t \leq r$, 
\begin{equation*}
  \pr{S^t} \rightarrow \pi_{\tilde{T}}(t) := P(\gamma^{\tT}_\emptyset) \prod_{t'=1}^{t} \prod_{j = 1}^{|e^{\tT}_{t'}| - 1} \hat{P}_{|e^{\tT}_{t'}|} (\gamma^{\tT}_{(e^{\tilde{T}}_{t'}, j)}).
\end{equation*}
If $\pi_{\tilde{T}}(t-1) = 0$, then $\pr{S^t} \leq \pr{S^{t-1}} \rightarrow 0= \pi_{\tilde{T}}(t)$. If $\pi_{\tilde{T}}(t-1) \neq 0$, we have $\pr{S^{t-1}} > 0$ for $n$ large enough. Note that $\pr{S^t} = \pr{S^{t-1} \cap \tilde{S}_t \cap E_t^c}$. Thereby,
\begin{equation*}
  \pr{S^{t-1}} \pr{\tilde{S}_t|S^{t-1}} - \pr{E_t} \leq \pr{S^t} \leq \pr{S^{t-1}} \pr{\tilde{S}_t|S^{t-1}}.
\end{equation*}
But, we know that $\pr{E_t} \rightarrow 0$. Consequently, it suffices to prove that 
\begin{equation}
\label{eq:tildeSt-convergence}
  \pr{\tilde{S}_t | S^{t-1}} \rightarrow \prod_{j = 1}^{|e^{\tT}_t| - 1} \hat{P}_{|e^{\tT}_t|} (\gamma^{\tT}_{(e^{\tilde{T}}_t, j)}).
\end{equation}

Since we construct one edge at a time in the exploration process, conditioned on $S^{t-1}$, the first $t-1$ edges are constructed in a way consistent with $\tilde{T}$. Therefore, it is easy to see that 
\begin{equation*}
  \pr{\tilde{S}_t|S^{t-1}} = \pr{\gamma^{(n)}_{u_{t, j}} = \gamma^{\tilde{T}}_{(e^{\tilde{T}}_t, j)}+\typee_{|e^{\tilde{T}}_t|} \text{ for } 1 \leq j \leq |e^{\tilde{T}}_t| - 1 }.
\end{equation*}
For $1 \leq j \leq |e^{\tilde{T}}_t| - 1$, let $\tilde{S}_{t, j}$ denote the event that $\gamma^{(n)}_{u_{t, j}} = \gamma^{\tilde{T}}_{(e^{\tilde{T}}_t, j)}+\typee_{|e^{\tilde{T}}_t|}$. Now, we study the probability of $\tilde{S}_{t, j}$ conditioned on $S^{t-1}$ and $\tilde{S}_{t,1}, \dots, \tilde{S}_{t, j-1}$. Note that, having chosen $e_{t, 1}, \dots, e_{t, j-1}$, there are $|\Delta^{(n)}(k_t)| - |\Delta^{(n)}(k_t) \cap C_t| - j$ many candidates for $e_{t, j}$, each having the same chance. 
With $B_{t,j} := \{ e \in \Delta^{(n)}(|e^{\tilde{T}}_t|): \gamma^{(n)}_{v(e)} = \gamma^{\tilde{T}}_{(e^{\tilde{T}}_t, j)}+\typee_{|e^{\tilde{T}}_t|} \}$, the event $\tilde{S}_{t, j}$ happens iff $e_{t, j}$ is chosen among the set $B_{t, j} \setminus (C_t \cup \{e_{t, 1}, \dots, e_{t, j-1} \})$. Therefore, 
\begin{equation*}
  \pr{\tilde{S}_{t, j}|S^{t-1}, \tilde{S}_{t, 1}, \dots, \tilde{S}_{t, j-1}} = \frac{|B_{t, j} \setminus (C_t \cup \{e_{t, 1}, \dots, e_{t, j-1} \})|}{|\Delta^{(n)}(|e^{\tilde{T}}_t|)| - |\Delta^{(n)}(|e^{\tilde{T}}_t|) \cap C_t| - j}.
\end{equation*}
Note that 
\begin{align*}
  \frac{1}{n}|B_{t, j}| &= \frac{1}{n} \sum_{i=1}^n (\gamma^{\tilde{T}}_{(e^{\tilde{T}}_t, j)}(|e^{\tilde{T}}_t|)+1)\one{\gamma^{(n)}_i = \gamma^{\tilde{T}}_{(e^{\tilde{T}}_t, j)} + \typee_{|e^{\tilde{T}}_t|} }.
\end{align*}
Using~\eqref{eq:fraction-converge-P(gamma)}, we have 
\begin{equation*}
    \frac{1}{n}|B_{t, j}| \rightarrow (\gamma^{\tilde{T}}_{(e^{\tilde{T}}_t, j)}(|e^{\tilde{T}}_t|)+1)  P(\gamma^{\tilde{T}}_{(e^{\tilde{T}}_t, j)} + \typee_{|e^{\tilde{T}}_t|}).
\end{equation*}
On the other hand, 
conditioned on $S^{t-1}$, $|C_t| = \sum_{j=1}^{t-1} |e^{\tilde{T}}_j|$, which is a constant.
Consequently, 
\begin{equation}
  \label{eq:1n-Brj-converge}
    \frac{1}{n}|B_{t, j}\setminus (C_t \cup \{e_{t, 1}, \dots, e_{t, j-1} \})| \rightarrow (\gamma^{\tilde{T}}_{(e^{\tilde{T}}_t, j)}(|e^{\tilde{T}}_t|)+1)  P(\gamma^{\tilde{T}}_{(e^{\tilde{T}}_t, j)} + \typee_{|e^{\tilde{T}}_t|}).
\end{equation}
Moreover, using \eqref{eq:ev-type-first-moment-converge},
\begin{equation}
\label{eq:Delta-n-Gamma-eT-t}
  \frac{1}{n} |\Delta^{(n)}(|e^{\tilde{T}}_t|)| = \frac{1}{n} \sum_{i=1}^n \gamma^{(n)}_i(|e^{\tilde{T}}_t|) \rightarrow \ev{\Gamma(|e^{\tilde{T}}_t|)}.
\end{equation}
Note that we are conditioning on $S^{t-1}$ and assuming that $\pr{S^{t-1}} \neq 0$. On the other hand, $|e^{\tilde{T}}_t|$ is equal to the size of the partial edge $e_t$ which is a member of $\Delta^{(n)}$. Using the assumption~\eqref{eq:type-zero-outside-index-set}, we have $|e^{\tilde{T}}_t| \in I$ and hence $\ev{\Gamma(|e^{\tilde{T}}_t|)} > 0$.  Putting \eqref{eq:1n-Brj-converge} and \eqref{eq:Delta-n-Gamma-eT-t} together, we have 
\begin{equation*}
  \pr{\tilde{S}_{t, j}|S^{t-1}, \tilde{S}_{t, 1}, \dots, \tilde{S}_{t, j-1}} \rightarrow \frac{(\gamma^{\tilde{T}}_{(e^{\tilde{T}}_t, j)}(|e^{\tilde{T}}_t|)+1)  P(\gamma^{\tilde{T}}_{(e^{\tilde{T}}_t, j)} + \typee_{|e^{\tilde{T}}_t|})}{\ev{\Gamma(|e^{\tilde{T}}_t|)}} = \hat{P}_{|e^{\tilde{T}}_t|}(\gamma^{\tilde{T}}_{(e^{\tilde{T}}_t, j)})
\end{equation*}
Multiplying for $1 \leq j \leq |e^{\tilde{T}}_t|-1$, we get \eqref{eq:tildeSt-convergence} which completes the proof.
\end{proof}

\subsection{Almost sure convergence}
\label{sec:alsm-sure-conv}

In this section we prove that, with the assumptions of Theorem~\ref{thm:conf-model-convegence-Un}, $u_{H^e_n} \Rightarrow \ugwt(P)$ almost surely.

  For a fixed $n$, Let $\Delta^{(n)}(k_1) \dots \Delta^{(n)}(k_L)$ be the nonempty sets among $\Delta^{(n)}(2), \dots, \Delta^{(n)}(n)$. From \eqref{eq:norm-infty-gamma-log} we know that $L \leq c_1 (\log n)^{c_2}$ and also $k_i \leq c_1 (\log n)^{c_2}$ for $1 \leq i \leq L$. For the sake of simplicity, write $\sigma$ for $(\sigma_{k_1}, \dots, \sigma_{k_{L}})$ and $M_i$ for $\mM_{k_i}(\Delta^{(n)}(k_i))$, $1 \leq i \leq L$. From our construction, we know that $\sigma_{k_i}$ is drawn uniformly at random from $M_i$ and is independent from $\sigma_{k_j}$, $j \neq i$. With $H_n^e$ being the simple hypergraph constructed by the configuration model, for a fixed $d > 0$ and a rooted tree $(T, o)$ of depth at most $d$, define
  \begin{equation*}
    F(\sigma) := \frac{1}{n} \sum_{i=1}^n \oneu{(H^e_n, i)_d \equiv (T, i)}.
  \end{equation*}
From Proposition~\ref{prop:conf-model_unerased-convergence} we know that $\lim_{n \rightarrow \infty} \ev{F(\sigma)} = \mu(A_{(T, o)})$.
We will show that $F$ is concentrated around its mean via a bounded difference argument. 

Now, for each $1 \leq j \leq L$, fix a permutation $\pi_{k_j} \in M_j$ and define $\pi = (\pi_{k_1}, \dots, \pi_{k_L})$. Moreover, fix $1 \leq i \leq L$ and  $e, e' \in \Delta^{(n)}(k_i)$. With this, define $\pi'_{k_i} := \text{swap}_{e,e'} \circ \pi_{k_i} \circ \text{swap}_{e,e'}$, which is the conjugation of $\pi_{k_i}$ with the permutation that swaps $e$ and $e'$. In fact, the cycle representation of $\pi'_{k_i}$ is obtained by swapping $e$ and $e'$ in the cycle representation of $\pi_{k_i}$. Moreover, let  $\pi' = (\pi_{k_1}, \dots, \pi'_{k_i}, \dots, \pi_{k_L})$ which differs from $\pi$ only on the $i^\text{th}$ coordinate. With this, the hypergraph obtained from $\pi$ and the hypergraph obtained from $\pi'$ differ only in at most two edges. Since all edge sizes and degrees in the graph are bounded to $c_1 (\log n)^{c_2}$, there are at most $2(c_1 (\log n)^{c_2})^{2d+1}$ many vertices in the hypergraph which have distance at most $d$ to a vertex in any of these two edges. Consequently, 
  \begin{equation}
    \label{eq:F-sigma-F-sigma'}
    | F(\pi) - F(\pi') | \leq \frac{2(c_1 (\log n)^{c_2})^{2d+1}}{n}.
  \end{equation}

Now, fix $1 \leq i \leq L$ and $\pi_{k_j} \in M_j$ for $ j \neq i$. Let $\sigma_{k_i}$ being chosen uniformly at random in $M_i$ and  define $F_i(\sigma_{k_i}) = F(\pi_{k_1}, \dots, \sigma_{k_i}, \dots, \pi_{k_L})$.
Since $\Delta^{(n)}(k_i)$ is finite, we can equip it with an arbitrary total order. Let $X_1$ be the smallest element in $\Delta^{(n)}(k_i)$ and define $Y_1 = (X_1, \sigma_{k_i}(X_1), \dots, \sigma_{k_i}^{(k_i - 1)}(X_1))$, which is in fact the orbit of $X_1$, or in the configuration model language, the edge containing the partial edge $X_1$. Let $X_2$ be the smallest element that does not appear in $Y_1$ and let $Y_2 = (X_2, \sigma_{k_i}(X_2), \dots, \sigma_{k_i}^{(k_i - 1)}(X_2))$. We continue this process inductively, i.e.\ let $X_j$ be the smallest element that has not appeared in $Y_1, \dots, Y_{j-1}$ and let $Y_j = (X_j, \sigma_{k_i}(X_j), \dots, \sigma_{k_i}^{(k_i - 1)}(X_j))$. This process yields $Y_1, \dots, Y_{|\Delta^{(n)}(k_i)|/k_i}$. For $1 \leq j \leq |\Delta^{(n)}(k_i)|/k_i$, let  $\mF_j$ be the sigma field generated by $Y_1, \dots, Y_j$. Moreover, let $Z_j = \ev{F_i(\sigma_{k_i}) | \mF_j}$ for $1 \leq j \leq |\Delta^{(n)}(k_i)| / k_i$ and let $Z_0 = \ev{F_i(\sigma_{k_i})}$. Note that $\pi_{k_j}$, $j \neq i$ are fixed; therefore, the randomness in the expression is with respect to $\sigma_{k_i}$ only. Indeed, $(Z_j, 0 \leq j \leq |\Delta^{(n)}(k_i)| / k_i)$ is a martingale. We claim that, almost surely, we have
\begin{equation*}
  |Z_{j+1} - Z_j| \leq k_i \frac{2(c_1 (\log n)^{c_2})^{2d+1}}{n}.
\end{equation*}
The reason is that changing the value of  the $k_i$ variables in $Y_{j+1}$ can change the value of $F_i$ by at most $k_i \frac{2(c_1 (\log n)^{c_2})^{2d+1}}{n}$ and the above inequality results from \eqref{eq:F-sigma-F-sigma'}. 
Using Azuma's inequality and the fact that $k_i \leq c_1 (\log n)^{c_2}$, we have 
\begin{equation}
\label{eq:concentration-Fi-1}
  \pr{|F_i(\sigma_{k_i}) - \ev{F_i(\sigma_{k_i})}| > \delta} < 2 \exp \left ( - \frac{\delta^2 n^2 }{4 |\Delta^{(n)}(k_i)| (c_1 (\log n)^{c_2})^{4d+3}} \right ).
\end{equation}
To obtain an upper bound for $|\Delta^{(n)}(k_i)|$, note that 
\begin{equation*}
  | \Delta^{(n)}(k_i) | = \left ( \sum_{j=1}^n \gn_j(k_i) \right) \leq n \left ( \frac{1}{n} \sum_{j=1}^n \snorm{\gamma^{(n)}_j}_1^2 \right ).
\end{equation*}
From \eqref{eq:second-moment-of-one-norm-bounded}, 
there is a constant $\alpha$ independent of $n$ and $i$ that $\sum \snorm{\gamma^{(n)}_j}_1^2 < \alpha n$. Hence, $|\Delta^{(n)}(k_i)| < \alpha n$.
Incorporating this into \eqref{eq:concentration-Fi-1}, we have, for $1 \leq i \leq L$,
\begin{equation}
\label{eq:concentration-Fi-2}
  \pr{|F_i(\sigma_{k_i}) - \ev{F_i(\sigma_{k_i})}| > \delta} < 2 \exp \left ( - \frac{\delta^2 n }{4\alpha (c_1 (\log n)^{c_2})^{4d+3}} \right ).
\end{equation}
Since this is true for all $i$ and $\pi_{k_j}$, $j \neq i$ and also the $\sigma_{k_j}$ are independent, using the above inequality $L$ times and using the fact that $L \leq c_1(\log n)^{c_2}$, we have 
\begin{equation*}
  \pr{|F(\sigma) - \ev{F(\sigma)}| > \delta} \leq 2c_1(\log n)^{c_2} \exp \left ( - \frac{\delta^2 n }{4\alpha (c_1 (\log n)^{c_2})^{4d+3}} \right ).
\end{equation*}
As the sum of the RHS over $n$ is finite, using the Borel--Cantelli lemma and the fact that $\ev{F(\sigma)} \rightarrow \mu(A_{(T, o)})$, we have $F(\sigma) \rightarrow \mu(A_{(T, o)})$ almost surely. But there are countably  many choices for $d$ and the rooted hypertree $(T, o)$. Thus, outside a measure zero set, $u_{H_n^e}(A_{(T, o)}) \rightarrow \mu(A_{(T, o)})$ for all rooted tree $(T, o)$ with finite depth. The proof is complete, using Lemma~\ref{lem:eq-condition-local-weak-convergence}.

\section{Proof of Proposition~\ref{prop:epsilon-balanced-uniqueness}}
\label{sec:proof-proposition-eba-uniqueness}

\begin{proof}[Proof of Proposition~\ref{prop:epsilon-balanced-uniqueness}]
Since $\Theta'_\epsilon$ is $\epsilon$--balanced, from Definition~\ref{def:epsilon-balanced-allocation-H*}, for $\vmu$--almost every $[H, e, i] \in \mH_{**}$, we have 
\begin{equation}
  \label{eq:Theta'-epsilon-balaced}
  \Theta'_\epsilon(H, e, i) = \frac{ \exp \left ( - \frac{\partial \Theta'_\epsilon(H, i) }{\epsilon} \right ) }{\sum_{j \in e} \exp \left ( - \frac{\partial \Theta'_\epsilon(H, j)}{\epsilon} \right )}. 
\end{equation}
Using Proposition~\ref{prop:everyting-shows-at-the-root}, there exists a $A \subset \mH_{**}$ such that $\vmu(A^c) = 0$ and, for all $[H, e, i] \in A$, we have 
\begin{equation*}
  \Theta'_\epsilon(H, e', i') = \frac{ \exp \left ( - \frac{\partial \Theta'_\epsilon(H, i') }{\epsilon} \right ) }{\sum_{j \in e'} \exp \left ( - \frac{\partial \Theta'_\epsilon(H, j)}{\epsilon} \right )} \qquad \forall\, (e', i') \in \evpair(H).
\end{equation*}
Now, fix some $[H, e, i] \in A$ and take an arbitrary element of this equivalence class $(H, e, i) \in [H, e, i]$. The above equation guarantees that if we define the allocation $\theta'^H_\epsilon$ on $H$ as $\theta'^H_\epsilon(e', i') := \Theta'_\epsilon(H, e', i')$ for $(e',i') \in \evpair(H)$,  then $\theta'^H_\epsilon$ is an $\epsilon$--balanced allocation on $H$. Now, assume that $\theta^{H^\Delta}_\epsilon$ is the (unique) $\epsilon$--balanced allocation on the truncated hypergraph $H^\Delta$ defined in Section~\ref{sec:epsil-balanc-alloc} (uniqueness comes from boundedness of $H^\Delta$). Proposition~\ref{prop:bounded-partial-b-increasing} then implies that $\partial \theta^{H^\Delta}_\epsilon(i') \leq \partial \theta'^H_\epsilon(i')$ for all $i' \in V(H)$. Sending $\Delta$ to infinity, this means that $\partial \theta^H_\epsilon(i') \leq \partial \theta'^H_\epsilon(i')$ for all $i' \in V(H)$, with $\theta^H_\epsilon$ being the canonical $\epsilon$--balanced allocation on $H$. Using Remark~\ref{rem:Theta-epsilon-H**-canonical} and the definition of $\theta'^H_\epsilon$, this means that for $\vmu$--almost all $[H, e, i] \in \mH_{**}$, $\partial \Theta_\epsilon(H, i) \leq \partial \Theta_{\epsilon}'(H, i)$. From part $(ii)$ of Lemma~\ref{lem:A-as-Atilde-ae}, $\mu$--almost surely we have 
\begin{equation}
  \label{eq:del-Theta<del-Theta'}
  \partial \Theta_\epsilon \leq \partial \Theta'_\epsilon.
\end{equation}

On the other hand, using unimodularity of $\mu$ and the fact that $\Theta_\epsilon$ is a Borel  allocation, we have
\begin{equation*}
  \int \partial \Theta_\epsilon d \mu = \int \Theta_\epsilon d \vmu = \int \nabla \Theta_\epsilon d \vmu = \int \frac{1}{|e|} d \vmu([H, e,i]). 
\end{equation*}
Using the same logic, $\int \partial \Theta'_\epsilon d \mu = \int \frac{1}{|e|} d \vmu([H, e,i])$. This means that $\int \partial \Theta_\epsilon d \mu = \int \partial \Theta'_\epsilon d \mu$. As $\deg(\mu) < \infty$, this common value is finite. This, together with \eqref{eq:del-Theta<del-Theta'}, implies that $\partial \Theta_\epsilon = \partial \Theta'_\epsilon$, $\mu$--almost surely. Therefore, Proposition~\ref{prop:everyting-shows-at-the-root} implies that for $\mu$--almost all $[H, i] \in \mH_*$, we have $\partial \Theta_\epsilon(H, j) = \partial \Theta'_\epsilon(H, j)$ for all $j \in V(H)$. Then, part $(i)$ of Lemma~\ref{lem:A-as-Atilde-ae} implies that for $\vmu$--almost all $[H, e, i] \in \mH_{**}$, $\partial \Theta_\epsilon(H, j) = \partial \Theta'_\epsilon(H, j)$ for all $j \in V(H)$. Thereby, using \eqref{eq:Theta'-epsilon-balaced} for $\Theta_\epsilon$ and $\Theta'_\epsilon$, we have $\Theta_\epsilon(H, e, i) = \Theta'_\epsilon(H, e, i)$ for $\vmu$--almost all $[H,e, i] \in \mH_{**}$, which completes the proof. 
\end{proof}


\section*{Acknowledgements}
\label{sec:acknowledgements}

The authors acknowledge support from the NSF grants ECCS-1343398, CNS-
1527846, CCF-1618145, the NSF Science \& Technology Center grant CCF-
0939370 (Science of Information), and the William and Flora Hewlett Foundation
supported Center for Long Term Cybersecurity at Berkeley.


\begin{thebibliography}{DM{\etalchar{+}}10}

\bibitem[AL07]{aldous2007processes}
David Aldous and Russell Lyons.
\newblock Processes on unimodular random networks.
\newblock {\em Electron. J. Probab}, 12(54):1454--1508, 2007.

\bibitem[AMO09]{agarwal2009fixed}
R.P. Agarwal, M.~Meehan, and D.~O'Regan.
\newblock {\em Fixed Point Theory and Applications}.
\newblock Cambridge Tracts in Mathematics. Cambridge University Press, 2009.

\bibitem[AS04]{aldous2004objective}
David Aldous and J~Michael Steele.
\newblock The objective method: probabilistic combinatorial optimization and
  local weak convergence.
\newblock In {\em Probability on discrete structures}, pages 1--72. Springer,
  2004.

\bibitem[AS16]{anantharam2016densest}
Venkat Anantharam and Justin Salez.
\newblock The densest subgraph problem in sparse random graphs.
\newblock {\em The Annals of Applied Probability}, 26(1):305--327, 2016.

\bibitem[Bil71]{billingsley1971weak}
Patrick Billingsley.
\newblock Weak convergence of measures.
\newblock {\em Philadelphia: SIAM}, 1971.

\bibitem[Bil99]{billingsley2013convergence}
Patrick Billingsley.
\newblock {\em Convergence of probability measures}.
\newblock John Wiley \& Sons, 1999.

\bibitem[Bor14]{Bordenave14randomgraphs}
Charles Bordenave.
\newblock Lecture notes on random graphs and probabilistic combinatorial
  optimization.
\newblock 2014.
\newblock \url{http://www.math.univ-toulouse.fr/~bordenave/coursRG.pdf}.

\bibitem[BS01]{benjamini2001schramm}
Itai Benjamini and Oded Schramm.
\newblock Recurrence of distributional limits of finite planar graphs.
\newblock {\em Electron. J. Probab.}, 6:no. 23, 13 pp. (electronic), 2001.

\bibitem[CK11]{csiszar2011information}
Imre Csiszar and J{\'a}nos K{\"o}rner.
\newblock {\em Information theory: coding theorems for discrete memoryless
  systems}.
\newblock Cambridge University Press, 2011.

\bibitem[DM{\etalchar{+}}10]{dembo2010gibbs}
Amir Dembo, Andrea Montanari, et~al.
\newblock Gibbs measures and phase transitions on sparse random graphs.
\newblock {\em Brazilian Journal of Probability and Statistics},
  24(2):137--211, 2010.

\bibitem[H{\etalchar{+}}96]{hajek1996balanced}
Bruce Hajek et~al.
\newblock Balanced loads in infinite networks.
\newblock {\em The Annals of Applied Probability}, 6(1):48--75, 1996.

\bibitem[Haj90]{hajek1990performance}
Bruce Hajek.
\newblock Performance of global load balancing by local adjustment.
\newblock {\em Information Theory, IEEE Transactions on}, 36(6):1398--1414,
  1990.

\bibitem[Mun00]{munkres2000topology}
James~R Munkres.
\newblock {\em Topology}.
\newblock Prentice Hall, 2000.

\bibitem[SS07]{shaked2007stochastic}
Moshe Shaked and J~George Shanthikumar.
\newblock {\em Stochastic orders}.
\newblock Springer Science \& Business Media, 2007.

\bibitem[Tao11]{tao2011introduction}
T.~Tao.
\newblock {\em An Introduction to Measure Theory}.
\newblock Graduate studies in mathematics. American Mathematical Society, 2011.

\end{thebibliography}

\newcommand{\etalchar}[1]{$^{#1}$}

\end{document}